\pdfoutput=1 
\documentclass[a4paper,reqno,12pt]{amsart}
\usepackage[utf8]{inputenc}
\usepackage[T1]{fontenc}
\usepackage[a4paper,scale={0.72,0.74},marginratio={1:1},footskip=7mm,headsep=10mm]{geometry}

\usepackage{amsmath, amsfonts, amssymb, amsthm, amscd, mathtools} 
\usepackage{perpage}
\usepackage{url}
\usepackage{color}
\usepackage{mathabx}
\usepackage{framed} 
\usepackage{mathrsfs}
\usepackage{mathabx}
\usepackage{stmaryrd}
\usepackage{dsfont} 
\usepackage{lmodern} 
\normalfont \DeclareFontShape{T1}{lmr}{bx}{sc} { <-> ssub * cmr/bx/sc }{} 
\SetSymbolFont{stmry}{bold}{U}{stmry}{m}{n} 
\usepackage{microtype} 
\usepackage{hyperref}
\usepackage{cleveref}
\usepackage{caption,subcaption}
\usepackage{csquotes} 
\usepackage{dsfont} 
\newfont{\indic}{bbmss12}
\def\un#1{\hbox{{\indic 1}$_{#1}$}}

\usepackage[doi=false,
	isbn=false,
	url=false,
	backend=bibtex,
	style=alphabetic,
	maxbibnames=99,
	maxalphanames=99,
	giveninits=true
	]
	{biblatex} 
\usepackage{tikz}
\usetikzlibrary{arrows,decorations.pathmorphing,backgrounds,positioning,fit,petri} 
\usepackage{tikz,tikz-cd}

\makeatletter
 \def\@textbottom{\vskip \z@ \@plus 10pt}
 \let\@texttop\relax
\makeatother

\makeatletter
\def\@secnumfont{\bfseries\scshape}

\def\section{\@startsection{section}{1}%
  \z@{.8\linespacing\@plus\linespacing}{.5\linespacing}%
  {\normalfont\large\bfseries\scshape\centering}}

\def\subsection{\@startsection{subsection}{2}%
  \z@{.5\linespacing\@plus.7\linespacing}{-.5em}%
  {\normalfont\bfseries\scshape}}

\def\subsubsection{\@startsection{subsubsection}{3}%
  \z@{.5\linespacing\@plus.7\linespacing}{-.5em}%
  {\normalfont\scshape}}

\def\paragraph{\@startsection{paragraph}{4}%
  \z@{.5\linespacing\@plus.7\linespacing}{-.5em}%
  {\normalfont\scshape}}

\def\specialsection{\@startsection{section}{1}%
  \z@{\linespacing\@plus\linespacing}{.5\linespacing}%
  {\normalfont\centering\large\bfseries\scshape}}
\makeatother

\usepackage{xargs} 
\usepackage[colorinlistoftodos,prependcaption,textsize=scriptsize]{todonotes} 

\newcommandx{\lucas}[2][1=]{\todo[inline,linecolor=orange,backgroundcolor=orange!25,bordercolor=orange,#1, author = LUCAS]{#2}} 
\newcommandx{\lorenzo}[2][1=]{\todo[inline,linecolor=blue,backgroundcolor=blue!25,bordercolor=blue,#1, author = LORENZO]{#2}} 
\newcommandx{\francesco}[2][1=]{\todo[inline,linecolor=red,backgroundcolor=red!25,bordercolor=red,#1, author = FRANCESCO]{#2}} 

\numberwithin{equation}{section}

\definecolor{shadecolor}{gray}{.94}

\makeatletter
\newenvironment{myshade}{%
  \topsep4\p@\@plus4\p@\relax%
  \MakeFramed{\advance\hsize-\width \FrameRestore}}%
 {\par\unskip\endMakeFramed}%
\makeatother

\newtheoremstyle{mytheorem}{0}{0}%
     {\itshape}
     {}
     {\bfseries}
     {. }
     {0.3ex}
     {\thmname{{\bfseries #1}}\thmnumber{ {\bfseries #2}}\thmnote{ (#3)}}  

\theoremstyle{mytheorem}

\newtheorem{theo}{Theorem}[section]
\newenvironment{theorem}{\begin{myshade}\begin{theo}}{\end{theo}\end{myshade}}

\newtheorem{prop}[theo]{Proposition}
\newenvironment{proposition}{\begin{myshade}\begin{prop}}{\end{prop}\end{myshade}}

\newtheorem{lem}[theo]{Lemma}
\newenvironment{lemma}{\begin{myshade}\begin{lem}}{\end{lem}\end{myshade}}

\newtheorem{defin}[theo]{Definition}
\newenvironment{definition}{\begin{myshade}\begin{defin}}{\end{defin}\end{myshade}}

\newtheorem{cor}[theo]{Corollary}

\newtheorem{assump}[theo]{Assumption}
\newenvironment{assumption}{\begin{myshade}\begin{assump}}{\end{assump}\end{myshade}}

\newtheoremstyle{mydefinition}{.7\linespacing\@plus.3\linespacing}{.7\linespacing\@plus.3\linespacing}%
     {\rmfamily}
     {}
     {\bfseries}
     {. }
     {0.3ex}
     {\thmname{{\bfseries #1}}\thmnumber{ {\bfseries #2}}\thmnote{ (#3)}}  

\theoremstyle{mydefinition}

\newtheorem{example}[theo]{Example}
\newtheorem{remark}[theo]{Remark}

\newenvironment{myenumerate}{%
\renewcommand{\theenumi}{\arabic{enumi}}%
\renewcommand{\labelenumi}{{\rm(\theenumi)}}%
\begin{list}{\labelenumi}
	{%
	\setlength{\itemsep}{0.4em}%
	\setlength{\topsep}{0.5em}%
	\setlength\leftmargin{2.45em}%
	\setlength\labelwidth{2.05em}%
	\setlength{\labelsep}{0.4em}%
	\usecounter{enumi}%
	}%
	}%
{\end{list}
}

{\end{list}
}

\newenvironment{ienumerate}{%
\renewcommand{\theenumi}{\roman{enumi}}%
\renewcommand{\labelenumi}{{\rm(\theenumi)}}%
\begin{list}{\labelenumi}
	{%
	\setlength{\itemsep}{0.4em}%
	\setlength{\topsep}{0.5em}%
	\setlength\leftmargin{2.45em}%
	\setlength\labelwidth{2.05em}%
	\setlength{\labelsep}{0.4em}%
	\usecounter{enumi}%
	}%
	}%
{\end{list}
}

\renewenvironment{enumerate}{
\begin{myenumerate}}%
{\end{myenumerate}}

\newenvironment{myitemize}{%
\begin{list}{$\bullet$}%
 	{%
	\setlength{\itemsep}{0.4em}%
	\setlength{\topsep}{0.5em}%
	\setlength\leftmargin{2.65em}%
	\setlength\labelwidth{2.65em}%
	\setlength{\labelsep}{0.4em}%
	}%
	}%
{\end{list}}

\renewenvironment{itemize}{
\begin{myitemize}}%
{\end{myitemize}}

\MakePerPage[2]{footnote} 

\newcommand{\R}{\mathbb{R}}

\newcommand{\N}{\mathbb{N}}

\newcommand{\sfK}{\mathsf{K}}

\newcommand{\supp}{\mathrm{supp}}

\newcommand{\cB}{{\mathcal B}}

\newcommand{\cD}{{\mathcal D}}
\newcommand{\cI}{{\mathcal I}}
\newcommand{\cJ}{{\mathcal J}}
\newcommand{\cK}{{\mathcal K}}

\newcommand{\cM}{{\mathcal M}}
\newcommand{\cN}{{\mathcal N}}

\newcommand{\cR}{{\mathcal R}}
\newcommand{\cQ}{{\mathcal Q}}
\newcommand{\cG}{{\mathcal G}}
\newcommand{\cT}{{\mathcal T}}
\newcommand{\cZ}{{\mathcal Z}}

\renewcommand{\epsilon}{\varepsilon}
\newcommand{\co}{\mathrm{coh}}

\newcommand{\ho}{\mathrm{hom}}

\newcommand{\balpha}{\boldsymbol{\alpha}}
\newcommand{\bhatalpha}{\boldsymbol{\hat\alpha}}

\newcommand{\ind}{\mathds{1}}

\newcommand{\vertiii}[1]{{\left\vert\kern-0.25ex\left\vert\kern-0.25ex\left\vert #1 
    \right\vert\kern-0.25ex\right\vert\kern-0.25ex\right\vert}}

\def\d{\mathrm{d}}
\def\and{\ and }

\newcommand\bbX{\mathbb{X}}

\newcommand\poly{\mathrm{poly}}

\newcommand\sfc{\mathsf{c}}

\title[Multilevel Schauder estimates without Regularity Structures]{Hairer's multilevel Schauder estimates\\
without Regularity Structures}

\date{\today}

\author{Lucas Broux, Francesco Caravenna, Lorenzo Zambotti}

\address[Lucas Broux]{Max Planck Institute for Mathematics in the Sciences, Inselstra{\ss}e 22, 04103 Leipzig, Germany}
\email{\href{mailto:lucas.broux@mis.mpg.de}{\nolinkurl{lucas.broux@mis.mpg.de}}}

\address[Lorenzo Zambotti]{Sorbonne Universit\'e, Laboratoire de Probabilit\'es, Statistique et Mod\'elisation, 4 Pl. Jussieu, 75005 Paris, France}
\email{\href{mailto:zambotti@lpsm.paris}{\nolinkurl{zambotti@lpsm.paris}}}

\address[Francesco Caravenna]{Dipartimento di Matematica e Applicazioni, Universit\`a degli Studi di Milano-Bicocca, via Cozzi 55, 20125 Milano, Italy}
\email{\href{mailto:francesco.caravenna@unimib.it}{\nolinkurl{francesco.caravenna@unimib.it}}}

\makeatletter
\@namedef{subjclassname@2020}{%
  \textup{2020} Mathematics Subject Classification}
\makeatother
\subjclass[2020]{46F10; 60L30}

\keywords{Distributions, Germs, Regularising Kernel, Schauder Estimates,
Reconstruction Theorem, Regularity Structures}

\begin{document}

\begin{abstract}
We investigate the regularising properties of singular kernels at the level of germs, i.e.\ 
families of distributions indexed by points in $\mathbb{R}^d$.
First we construct a suitable integration map which acts on general coherent germs.
Then we focus on germs that can be decomposed along a basis
(corresponding to the so-called modelled distributions in Regularity Structures)
and we prove a version of Hairer's multilevel Schauder estimates in this setting,
with minimal assumptions.
\end{abstract}

\maketitle

\setcounter{secnumdepth}{3} 
\setcounter{tocdepth}{2} 
\tableofcontents

\section{Introduction}

It is well-known that the convolution of a (Schwartz) distribution against a kernel admitting an integrable 
singularity on the diagonal yields a distribution with improved H\"older-Besov regularity: 
this is the content of the celebrated \emph{Schauder estimates} which are a fundamental tool in the analysis
of PDEs, since examples of regularising kernels include the heat kernel and the Green's function of 
many differential operators.

One of the key insights of Hairer's theory of Regularity Structures 
\cite{Hai14, MR3935036, chandra2018analytic, MR4210726} is that the same regularisation phenomenon 
still occurs when one
works at the level of \emph{families} of distributions, 
as formalised 
by the notion of 
\enquote{modelled distributions}
(which one should think 
of as local approximations to a distribution of interest).
The resulting \emph{multilevel Schauder estimates} \cite[Theorem~5.12]{Hai14} admit powerful 
consequences, as they allow to solve via fixed point 
many singular stochastic PDEs that are classically ill-posed, after
lifting them in a suitable space of modelled distributions;
see \cite{MR4174393,https://doi.org/10.48550/arxiv.2006.03524,Ber22} 
for expository presentations.
Let us also mention the works
\cite{https://doi.org/10.48550/arxiv.1803.07884,OW19,MR4164267,https://doi.org/10.48550/arxiv.2103.11039}
where Schauder estimates are established at the level of families of functions, in particular
with the aim of establishing a priori estimates for solutions of stochastic PDEs.

The purpose of the present paper is to formulate Hairer's 
multilevel Schauder estimates as a \emph{standalone result in distribution theory}, without any reference 
to the formalism of Regularity Structures.
In doing so, we sharpen and extend Hairer's original result under nearly optimal assumptions.

\smallskip

To provide some context, there has recently been an effort, see 
e.g.\
\cite{MR3966852,https://doi.org/10.48550/arxiv.2103.11039,MR4164267,CZ20,zorinkranich2021reconstruction},
to isolate the other key analytic result of 
Regularity Structures,  namely
the \emph{Reconstruction Theorem} \cite[Theorem~3.10]{Hai14}.
Given a family $F = (F_x)_{x \in \mathbb{R}^d}$
of distributions on $\R^d$ indexed by points in $\mathbb{R}^d$, called a \emph{germ},
the Reconstruction Theorem as presented in \cite{CZ20,zorinkranich2021reconstruction} 
roughly states the following:
\begin{quote}
Under a simple condition on the germ $F = (F_x)_{x \in \mathbb{R}^d}$
called \emph{coherence}, see \eqref{eq:coherence},
there exists a distribution $\mathcal{R}  F$, called \emph{reconstruction of $F$}, which is 
``well approximated'' by $F_x$ around any base point $x\in\R^d$
(with a quantitative bound for the difference $F_x - \mathcal{R}F$,
see \eqref{eq:reconstruction_bound}).
\end{quote}

The reconstruction map $F\mapsto \mathcal{R} F$ is better understood if one recalls the classical Taylor expansion of a smooth function: if $f\in C^\infty(\R^d)$ and $\gamma>0$, we can set 
\[
	F_x(\cdot):= \sum_{k\in\N_0^d \colon |k|< \gamma} 
	\partial^k f(x) \, \frac{(\cdot-x)^k}{k!}, \quad x\in\R^d\,,
	\qquad \mathcal{R}F:=f \,.
\]
By Taylor's theorem, we know that $|f(y)-F_x(y)|\lesssim |y-x|^\gamma$ uniformly for $x,y$ 
in compact subsets of $\R^d$, which shows that the function $f(\cdot)$ is 
well approximated by the function $F_x(\cdot)$ around any point $x\in\R^d$, 
with a precise bound. 

However, this situation is special because $f=\mathcal{R}F$ is
known in advance and we associate the family of local approximations $(F_x)_{x\in\R^d}$ to it.
In the Reconstruction Theorem, this point of view is rather reversed: the family $(F_x)_{x\in\R^d}$ is
assigned and the (unknown) distribution $f=\mathcal{R}F$ is (re-)constructed 
from $(F_x)_{x\in\R^d}$. 
We refer to \cite{MR3684891,rinaldi2021reconstruction,broux2021besov,kern} for similar results.
This point of view is
strongly inspired by the theory of rough paths, where the
analog of the Reconstruction Theorem is the Sewing Lemma 
\cite{MR1654527,MR2091358,MR2261056,Davie}.

\smallskip

Coming back to the present paper, we can formulate the Schauder estimates
in great generality, at the level of coherent germs $F = (F_x)_{x \in \mathbb{R}^d}$: we prove that
\emph{the convolution $\sfK * \cR F$ of a suitable regularising kernel $\sfK$ with a reconstruction $\cR F$
can be lifted to a map $\cK$ acting on germs~$F$}, so that the following diagram commutes:
\begin{equation*}
	\begin{tikzcd}
		F
		\arrow{r}{\mathcal{K}} \arrow[swap]{d}{\mathcal{R}} & 
		\cK F
		\arrow{d}{\mathcal{R}} \\%
		\cR F \arrow{r}{\sfK\,*\,}& \cR (\cK F)
	\end{tikzcd}
	\!\!\!\!
	\begin{tikzcd}
	\phantom{F} \\ 
	=\sfK*\cR F  \,.
	\end{tikzcd}
\end{equation*}
More precisely, our first main result can be stated as follows, where we
quantify the coherence
of a germ by an exponent $\gamma \in \R$ (see Definition~\ref{def:homogeneity_coherence})
and the regularisation of a kernel by an exponent $\beta > 0$
(see Definitions~\ref{def:regularising_kernel}
and Lemma~\ref{th:reg-trans}).

\begin{theorem}[Schauder estimates for coherent germs]\label{thm1.1}
Let $F= (F_x)_{x \in \mathbb{R}^d}$ be a $\gamma$-coherent germ
with a reconstruction $\cR F$.
Let $\sfK$ be a $\beta$-regularising kernel.

Then, assuming $\gamma+\beta \not\in\N_0$, the germ
$\cK F = ((\cK F)_x)_{x\in\R^d}$ given by
	\begin{equation} \label{eq:cKF}
		(\cK F)_x \coloneqq \sfK * F_x - \sum\limits_{|k| < \gamma + \beta} 
		D^k ( \sfK * \{ F_x - \mathcal{R} F\} ) ( x )
	\, \frac{ (\,\cdot - x)^k}{k!} 
	\end{equation}
\noindent is well-defined, 
it is $(\gamma + \beta )$-coherent, and it satisfies
$\mathcal{R} (\cK F) = \sfK * \mathcal{R} F$.

\smallskip

(The ``pointwise derivatives'' $D^k ( \ldots ) ( x )$ in \eqref{eq:cKF}
are defined by Lemma~\ref{lemma:pointwise_derivative}.)
\end{theorem}

\noindent
We refer to Theorem~\ref{thm:Schauder_for_germs} below for a more refined formulation of this result
where, as in the papers 
\cite{Hai14,MR4492924,MR4198718,MR3997634,hairer2023periodic},
we allow for non translation-invariant kernels (so we talk of \emph{integration}
$\sfK F_x$ rather than \emph{convolution} $\sfK * F_x$) and
we prove that the map $\cK$ is continuous for natural topologies on germs.
Some antecedents of Theorem~\ref{thm1.1} can be found in the works
\cite{https://doi.org/10.48550/arxiv.1803.07884,OW19,MR4164267,https://doi.org/10.48550/arxiv.2103.11039}, which
are concerned with special classes of germs given by solutions to appropriate (stochastic) PDEs.

\smallskip
A crucial property for germs in our context is \emph{homogeneity}, 
which quantifies the Hölder-like behavior of a germ through 
an exponent~$\bar\alpha \in \R$, see \eqref{eq:homogeneity} below,
and its variant \emph{weak homogeneity},
that is homogeneity modulo polynomials, see 
Definition~\ref{def:weak_homogeneity_coherence} below. 
For instance, the difference $\{F_x - \cR F\}$
which appears in \eqref{eq:cKF}
is a homogeneous germ (by the Reconstruction Theorem)
and understanding its convolution with~$\sfK$ is essential.
We prove the following result of independent interest,
which extends the classical Schauder estimates for distributions
to general homogeneous germs.

\begin{theorem}[Schauder estimates for homogeneous germs]\label{thm1.105}
Let $F= (F_x)_{x \in \mathbb{R}^d}$ be a $\bar\alpha$-homogeneous germ
and let $\sfK$ a $\beta$-regularising kernel. 
Then the germ $\sfK * F = (\sfK * F_x)_{x\in\mathbb{R}^d}$ is well-defined
and it is $(\bar{\alpha} + \beta)$-weakly homogeneous.
\end{theorem}

\noindent We refer to Theorem~\ref{thm:convolution_germs_in_G_checkG} below for a precise statement, see in particular \eqref{eq:Kconv-hom}. 

Theorems~\ref{thm1.1} and~\ref{thm1.105} show that
convolving/integrating a germ by $\sfK$ improves both coherence and homogeneity by $\beta$ 
(modulo polynomials).
The idea of considering homogeneity and coherence as independent properties comes from 
\cite{zorinkranich2021reconstruction} (in the context of 
the Reconstruction Theorem)
and we adopt it throughout this paper. 

\smallskip
Compared to Theorems~\ref{thm1.1} and \ref{thm1.105}, the Schauder estimates
in Regularity Structures \cite[Theorem~5.12]{Hai14}
have a more narrow scope: they apply to restricted classes
of germs, corresponding to so-called ``modelled distributions'',
but at the same time they yield 
sharper \emph{multilevel Schauder estimates},
which are crucial to solve singular PDEs.
We recover and extend such multilevel estimates in our framework.

Let us fix a finite family $\Pi = ( \Pi^{i} )_{i \in I}$ of germs $\Pi^i = (\Pi^i_x)_{x\in\R^d}$
which, like ordinary monomials,
can be \enquote{reexpanded} around any base point 
via some coefficients $\Gamma = (\Gamma_{xy}^{ji})$:
such a pair $M = (\Pi,\Gamma)$ is called a \emph{model}
(see Definition~\ref{def:model}). We
think of the family $\Pi = ( \Pi^{i} )_{i \in I}$ as a \emph{basis to build germs via linear combinations
$F = \langle f,\Pi \rangle$}, i.e.\
\begin{equation} \label{eq:FfPi}
	F_x = \langle f,\Pi \rangle_x = \sum_{i\in I} f^i(x) \, \Pi^i_x  \,,
\end{equation}
parametrised by real coefficients $f = (f^i(x))$.
To ensure that such germs $F= \langle f,\Pi \rangle$ are $\gamma$-coherent,
we require that
coefficients $f^i(\cdot)$ satisfy \emph{multilevel H\"older-like bounds},
which define a vector space of \emph{$\gamma$-modelled distributions $f$} for~$M$ (see Definition~\ref{def:modelled_distribution}).

Our second main result, which includes both
Hairer's multilevel Schauder estimates \cite[Theorem~5.12]{Hai14}
and Hairer's extension theorem \cite[Theorem~5.14]{Hai14},
shows that the map $F \mapsto \cK F$ from Theorem~\ref{thm1.1},
for germs $F = \langle f, \Pi \rangle$ of the form \eqref{eq:FfPi}, can be lifted to a map on modelled distributions $f \mapsto \hat{f}$
for a new model $\hat{M} = (\hat{\Pi}, \hat{\Gamma})$, such that the following diagram commutes:
\begin{equation*}
	\begin{tikzcd}
		f
		\arrow{r}{\hat{\cdot}} \arrow[swap]{d}{\langle \cdot , \Pi \rangle} & 
		\hat{f}
		\arrow{d}{\langle \cdot , \hat{\Pi} \rangle} \\%
		\langle f, \Pi \rangle \arrow{r}{\cK}& \langle \hat{f}, \hat{\Pi} \rangle
	\end{tikzcd}
	\!\!\!\!\!
	\begin{tikzcd}
	\phantom{F} \\ 
	=\cK\langle f, \Pi \rangle \rule{0pt}{1.5em}  \,.
	\end{tikzcd}
\end{equation*}
More precisely, we can prove the following.

\begin{theorem}[Multilevel Schauder estimates]\label{thm1.2}
Let $M = (\Pi,\Gamma)$ be a model and let $f$ be a $\gamma$-modelled distribution for~$M$,
so that $\langle f, \Pi \rangle$ is a $\gamma$-coherent germ
with a reconstruction $\cR\langle f, \Pi \rangle$.
Let $\sfK$ be a $\beta$-regularising kernel.

Then, assuming $\gamma+\beta \not\in \N_0$,
we can define a new explicit model $\hat{M} = (\hat{\Pi}, \hat{\Gamma})$,
and a new explicit
$(\gamma+\beta)$-modelled distribution $\hat{f}$ for $\hat{M}$, such that 
\begin{equation*}
	\langle \hat{f}, \hat{\Pi}\rangle = \cK \langle f, \Pi \rangle
\end{equation*}
where $\cK$ is the map from Theorem~\ref{thm1.1}; in particular, we have
$\mathcal{R} \langle \hat f,\hat{\Pi} \rangle = \sfK * \mathcal{R} \langle f, \Pi \rangle$. 
\end{theorem}

\noindent
We refer to Theorem~\ref{thm:multi_level_schauder_estimate} below for a precise formulation of
this result, where we also show that
the map $(M,f) \mapsto (\hat{M},\hat{f})$ is continuous in natural topologies.

\medskip

The Schauder estimates in Theorems~\ref{thm1.1},
\ref{thm1.105} and~\ref{thm1.2} are the main results of this paper.
Together with the Reconstruction Theorem from \cite{CZ20,zorinkranich2021reconstruction},
they provide \emph{a standalone formulation of the core analytic results from 
\cite{Hai14}}, without defining the notion of Regularity Structures.
We also obtain a number of improvements: let us briefly
describe the most significant ones.

\begin{enumerate}
	\item \emph{We do not assume that the kernel $\sfK$ annihilates polynomials}, i.e.\ 
	we do not require that
	$\int \sfK (x, y) \, y^k \, d y = 0$ as was assumed (for convenience) in \cite[Assumption~5.4]{Hai14}.
	Sometimes
	it is convenient (but not required) to assume that \emph{$\sfK$ preserves polynomials},
	namely that $\int \sfK (x, y) \, y^k \, d y$ is a polynomial of degree $\le |k|$,
	see e.g.\ Remark~\ref{rem:compatibility}.
	Note that this always holds in the translation
	invariant case 	$\sfK (x , y ) = \sfK ( x -y )$.

	\item \emph{We prove Schauder estimates
	for $\gamma$-coherent germs and $\gamma$-modelled distributions
	also for $\gamma \le 0$}, whereas in the literature
	it is always assumed that $\gamma>0$. Since
	the reconstruction of a $\gamma$-coherent germ is not unique when
	$\gamma \le 0$,
	\emph{a choice must be given as an input} in Theorems~\ref{thm1.1} and~\ref{thm1.2}:
	this poses no problem and, in fact, it decouples Schauder estimates
	from the Reconstruction Theorem. As mentioned
	to us by Hendrik Weber, our Schauder estimates with $\gamma \le 0$ can be useful to
	truncate modelled distributions associated with solutions to SPDEs.
	Another application of this idea can be found in \cite[Theorem~3.8]{MR4693281}.

	\item \emph{We introduce a new notion of weakly coherent germs},
	inspired by classical H\"older-Zygmund spaces, 
	see Definition~\ref{def:weak_homogeneity_coherence}.
	This allows us to give a ``conceptual'' proof of Theorem~\ref{thm1.1}
	factorised in two steps, see Section~\ref{section:proof_main_result_I},
	and also to recover in a very clear way the classical Schauder estimates for distributions,
	see Remark~\ref{remark:classical_schauder}. 
	
	\item \emph{We relax the definition of a model $M = (\Pi,\Gamma)$ from \cite{Hai14}}, as
	we do not need to impose that the reexpansion coefficients $\Gamma$ 
	satisfy a group property, an analytic bound, nor a triangular structure (see 
	Remark~\ref{remark:comparison_definitions_of_model}).
	However, we prove that these properties are preserved by the operation 
	$\Gamma \mapsto \hat{\Gamma}$, see 
	Proposition~\ref{prop:stability_of_properties_of_Gamma}.

	\item Another key property of germs, besides coherence,
	is the \emph{homogeneity}, 
	see Definition~\ref{def:homogeneity_coherence}. 
	Even though modelled distributions yield germs which are both coherent and homogeneous,
	\emph{we keep these properties as distinct as possible 
	in our discussion}, following the ideas  of \cite{zorinkranich2021reconstruction}. 
	This greater flexibility makes proofs more transparent and, moreover,
	allows to consider interesting germs 
	which need not be associated to modelled distributions. 
\end{enumerate}

In conclusion, in this paper we have improved some of the most powerful 
and beautiful results from \cite{Hai14}, presenting them
in a more general yet simpler setting, without losing in sharpness. We hope that
our formulation will make these results even more useful and widespread.

\subsection*{Organisation of the paper} 
The paper is structured as follows.
\begin{itemize}
\item In Section~\ref{section:classical_results} we set notations and recall classical results.

\item In Section~\ref{section:schauder_estimate_on_germs} we present our first main result 
Theorem~\ref{thm1.1}
(Schauder estimates for coherent germs), which we rephrase more precisely
as Theorem~\ref{thm:Schauder_for_germs}.

\item In Section~\ref{section:multi_level_schauder_estimate} we present our second main result
Theorem~\ref{thm1.2} (multilevel Schauder estimates),
which we formulate in a more detailed way in Theorem~\ref{thm:multi_level_schauder_estimate}.

\item In Section~\ref{section:proof_main_result_I} we prove Theorem~\ref{thm:Schauder_for_germs}
and other auxiliary results.

\item In Section~\ref{section:proof_main_result_II} we give the proof of
Theorem~\ref{thm:multi_level_schauder_estimate}.

\item Finally, some more technical results are deferred to the Appendix.
\end{itemize}

\subsection*{Acknowledgements} 
We thank Ismaël Bailleul, Martin Hairer, Cyril Labbé, Felix Otto, Scott Smith, Hendrik Weber
for very useful discussions.

\section{Classical results (revisited)}\label{section:classical_results}

We work in $\mathbb{R}^d$, where $d \geq 1$ is a fixed integer, with the Euclidean norm $|\cdot|$.
Balls are denoted by $B(x_0,r) = \{x\in\R^d \colon |x-x_0| \le r\}$.
We use the shorthand
\begin{equation*}
	f \lesssim g \qquad \iff \qquad \exists C < \infty \colon \quad f \le C g \,.
\end{equation*}

\smallskip

Given $r\in\N_0 = \{0,1,2,\ldots\}$, we denote by $C^r$ the space of functions $\varphi: \R^d \to \R$
which admit partial derivatives of order $k$ for all $|k| \le r$. The corresponding norm is
\begin{equation*}
	\|\varphi\|_{C^r} := \max_{|k| \le r} \, \| \partial^k \varphi \|_\infty \,,
\end{equation*}
where for a multi-index $k\in\N_0^d$ we set $|k| = k_1 + \ldots + k_d$.
Similarly, given $r, m \in \mathbb{N}_0$, we denote by $C^{m, r}$ 
the space of functions $\psi: \R^d \times \R^d \to \R$ which admit partial derivatives of order 
$(k_1, k_2)$ for all multi-indices $| k_1| \leq m$, $| k_2| \leq r$.

\subsection{Test functions and distributions}

We denote by $\cD = \cD(\R^d)$ the space of smooth \emph{test functions} 
$\varphi : \R^d \to \R$, i.e.\ $C^\infty$ with compact support.
We write $\cD(K)$ for the family of test functions
$\varphi \in \cD(\R^d)$ that are supported in $K \subset \R^d$.

Given a test-function $\varphi$, its scaled and centered version $\varphi_x^{\lambda}$ 
is defined by
\begin{equation} \label{eq:scaling}
	\varphi_x^{\lambda} ( \cdot ) \coloneqq \lambda^{- d} \varphi (\lambda^{-1} (\cdot - x)) \,,
\end{equation}
for $x \in \mathbb{R}^d$ and $\lambda > 0$.
Note that $\int\varphi_x^\lambda = \int\varphi$.

We denote by $\cD' = \cD'(\R^d)$ the space of \emph{distrbutions}, i.e.\
the linear functionals $f: \cD(\R^d) \to \R$ with the following property:
for any compact $K \subset \R^d$ there are $r = r_K \in \N_0$
and $c = c_K < \infty$ such that
\begin{equation} \label{eq:order}
	|f(\varphi)| \le c \, \|\varphi\|_{C^r}
	\qquad \forall \varphi \in \cD(K) \,.
\end{equation}
We say that \emph{$f$ is a distribution of order~$r$}
(meaning ``at most~$r$'')
if the value of~$r$ in \eqref{eq:order} can be chosen independent of~$K$,
while the constant $c$ may still depend on~$K$.
In this case, we can canonically define $f(\varphi)$ for non-smooth test functions
$\varphi \in C^r$.

The derivative of any distribution is defined by duality:
\begin{equation*}
	D^k f (\varphi) := (-1)^{|k|} f(\partial^k \varphi) 
	\qquad\forall k\in\N_0^d \,.
\end{equation*}
We will later give conditions under which \emph{pointwise derivatives $D^k f(x)$}
can be defined for suitable distributions, see Lemma~\ref{lemma:pointwise_derivative}.

\subsection{H\"older-Zygmund spaces}

For $\gamma \in \mathbb{R}$ we denote  by
$\mathcal{Z}^{\gamma} \coloneqq \mathcal{B}_{\infty, \infty, \mathrm{loc}}^{\gamma}$ 
the (local) H\"older-Zygmund spaces, see \cite[Section~14.3]{MR4174393},
which coincide with the usual (local) H\"older-Besov spaces $\mathcal{C}^\gamma$
when $\gamma$ is not an integer.
To recall their definition, we first introduce for 
$r \in \mathbb{N}_0$ and $\gamma \in \R$ 
the families of test-functions
\begin{equation}\label{eq:BB}
\begin{split}
		\mathscr{B}^r & \coloneqq \left\lbrace \varphi \in \mathcal{D} ( B(0,1) ):
		\ \| \partial^k \varphi \|_\infty\le 1 \  \text{ for all } 0 \le |k| \le r \
		\text{ (i.e.\ $\| \varphi \|_{C^r} \leq 1$)} \right\rbrace  \,, \\
		\mathscr{B}_\gamma  & \coloneqq 
		\left\lbrace \varphi \in \mathcal{D} ( B(0,1) ):
		\ \int_{\mathbb{R}^d} \varphi ( z ) z^k d z = 0 \text{ for all } 
		0 \leq | k | \leq \gamma \right\rbrace ,
\end{split}
	\end{equation}
and we denote their intersection by
	\begin{equation} \label{eq:Brgamma}
		\mathscr{B}_\gamma^r  \coloneqq 
		\mathscr{B}^r  \cap \mathscr{B}_\gamma  \,.
	\end{equation}
Note that we have $\mathscr{B}_\gamma^r = \mathscr{B}_m^r$
where $m = \lfloor \gamma \rfloor$ is the largest integer $m \le \gamma$.
Also note that  for $\gamma < 0$ the constraint $0 \leq | k | \leq \gamma$ is empty
and we have $\mathscr{B}_\gamma^r = \mathscr{B}^r$.

We can now define the spaces $\cZ^\gamma$.
Note that for $\gamma < 0$ we denote by
$r = \lfloor - \gamma + 1 \rfloor$ the smallest positive integer $r > - \gamma$.

\begin{definition}[H\"older-Zygmund spaces $\mathcal{Z}^{\gamma}$]
\label{def:Hoelder_Zygmund_spaces}
Let $\gamma \in \mathbb{R}$, we define $\mathcal{Z}^{\gamma}$ as the 
set of distributions 
$f \in \mathcal{D}^{\prime} (\mathbb{R}^d)$ such that 
\begin{equation*}
	\| f \|_{\mathcal{Z}_{K,\bar\lambda}^{\gamma}} < + \infty
\end{equation*}
for all compacts $K \subset \mathbb{R}^d$
and for some (hence any) $\bar\lambda \in [1,\infty)$, where
	\begin{equation}\label{eq:Cgamma}
		\| f \|_{\mathcal{Z}_{K,\bar\lambda}^{\gamma}} = 
			\begin{dcases}
				\sup\limits_{\substack{x \in K , \ \lambda \in ( 0, \bar\lambda], \\ 
				\varphi \in \mathscr{B}^{r} \,\text{with } r = \lfloor - \gamma + 1 \rfloor}} 
				\frac{\left| f ( \varphi_x^{\lambda} ) \right|}{\lambda^{\gamma}} 
				&\text{ if $\gamma < 0$}, \\
				\sup\limits_{\substack{x \in K \\ \psi \in \mathscr{B}^0}} 
				\left| f ( \psi_x ) \right| 
				+ \sup\limits_{\substack{x \in K , \ \lambda \in ( 0, \bar\lambda], \\ 
				\varphi \in \mathscr{B}^0_{\gamma} }} 
				\frac{\left| f ( \varphi_x^{\lambda} ) \right|}{\lambda^{\gamma}} 
				&\text{ if $\gamma \geq 0$}.
			\end{dcases}
	\end{equation}	 
We often set $\bar\lambda = 1$ and omit it from notation.
\end{definition}

For later purpose, we reformulate the condition that a distribution is of finite order.

\begin{remark}[Bounded order]
\label{rem:singint}
A distribution $f \in \cD'$ is of order~$r$, see \eqref{eq:order}, 
if and only if the following condition holds:
\begin{equation}\label{eq:order2}
	\forall z\in\R^d \,, \ \ \forall \bar\lambda \in [1,\infty): \qquad
	\sup_{\varphi \in \mathscr{B}^r, \, \lambda \in [1,\bar\lambda]} \, |f(\varphi_z^{\bar\lambda})|
	\eqcolon C(z,\bar\lambda) < \infty \,.
\end{equation}
This is also equivalent to the seemingly weaker condition
\begin{equation}\label{eq:order2bis}
	\exists z\in\R^d \ \text{ such that } \ \forall \bar\lambda \in \N = \{1,2,\ldots\}: \qquad
	\sup_{\varphi \in \mathscr{B}^r} \, |f(\varphi_z^{\bar\lambda})|
	\eqcolon C'(z,\bar\lambda) < \infty \,.
\end{equation}
We prove the equivalence between \eqref{eq:order}, \eqref{eq:order2}
and \eqref{eq:order2bis} below.

It follows by \eqref{eq:order2bis} that 
\emph{any $f \in \cZ^\gamma$ is a distribution of order~$r$},
for $r \in \N_0$ with $r > -\gamma$.
If $\gamma < 0$ this is a direct consequence of \eqref{eq:Cgamma}, since
for $z = 0$ (say) and any given $\bar\lambda \in \N$ we have
$|f(\varphi^{\bar\lambda})| \lesssim \bar\lambda^{\gamma} \lesssim 1$
for every $\varphi \in \mathscr{B}^r$.
If $\gamma \ge 0$, let us show that $f$ is a distribution of order~$0$:
by \eqref{eq:Cgamma} we know that
$|f(\psi_x)| \lesssim 1$ uniformly for $x$ in compact sets
and $\psi \in \mathscr{B}^0$; for $\bar\lambda \in \N$ and $\varphi \in \mathscr{B}^0$
we can decompose, using a partition of unity,
$\varphi^{\bar\lambda} = \sum_{k = 1}^n (\psi^{[k]})_{x_k}$ 
for suitable $x_k \in B(0,\bar\lambda+1)$, 
$\psi^{[k]} \in \mathscr{B}^0$ and $n$ (uniformly bounded given $\bar{\lambda}$);
then it follows that $|f(\varphi^{\bar\lambda})| \lesssim \sum_{k = 1}^n |f((\psi^{[k]})_{x_k})|
\lesssim 1$ uniformly over $\psi \in \mathscr{B}^0$, which proves \eqref{eq:order2bis} with
$z=0$ and $r=0$.
\end{remark}

\begin{proof}[Proof of equivalence between \eqref{eq:order}, \eqref{eq:order2} and \eqref{eq:order2bis}]
For any fixed $z\in\R^d$, we can deduce relation \eqref{eq:order2}
from \eqref{eq:order} with $K = B(z,\bar\lambda)$ because,
for fixed $\bar\lambda \in (0,\infty)$, we have
$\|\varphi_z^{\bar\lambda}\|_{C^r} \lesssim \|\varphi\|_{C^r}$
(note that $\partial^k (\varphi_z^\lambda)(\cdot) 
= \lambda^{-|k|-d} (\partial^k \varphi)(\lambda^{-1}(\cdot-z))$). 

Since \eqref{eq:order2} clearly implies \eqref{eq:order2bis}, it remains to show that
\eqref{eq:order2bis} implies \eqref{eq:order}:
given $z\in\R^d$ and a compact $K \subset \R^d$, we fix $\bar\lambda \in\N$ so that
$K \subseteq B(z,\bar\lambda)$, then it suffices to note that
any $\psi \in \cD(K)$ can be written as $\psi = a \,
\varphi_z^{\bar\lambda}$ for some $\varphi \in \mathscr{B}^r$
and $a \lesssim \|\psi\|_{C^r}$ (e.g.\ we can take $\varphi := \psi^{\bar\lambda^{-1}}_{-z}
/ \|\psi^{\bar\lambda^{-1}}_{-z}\|_{C^r}\in \mathscr{B}^r$ and
$c = \|\psi^{\bar\lambda^{-1}}_{-z}\|_{C^r} \le \bar\lambda^{r+d} \|\psi\|_{C^r}$.)
\end{proof}

\subsection{Singular kernels}

We define a class of kernels $\sfK(x,y)$ called \emph{$\beta$-regularising},
for reasons that will soon be clear. Intuitively, these kernels satisfy
\begin{equation} \label{eq:Kbasic}
	|\sfK(x,y)| \lesssim \frac{1}{|x-y|^{d-\beta}} \, \ind_{\{|x-y|\le \sfc\}}
	\qquad \text{for some} \ \beta, \sfc > 0 \,,
\end{equation}
but the precise assumptions are conveniently encoded via
a dyadic decomposition of $\sfK(x,y)$, as in \cite[Assumption~5.1]{Hai14}.
We anticipate that in the \emph{translation invariant} case $\sfK(x,y) = \sfK(x-y)$
these assumptions simplify considerably: we just require that $\sfK$ and its derivatives
satisfy a relation like \eqref{eq:Kbasic}, see Lemma~\ref{th:reg-trans}.

\begin{definition}[Regularising kernel]\label{def:regularising_kernel}
A function $\sfK \colon \mathbb{R}^d \times \mathbb{R}^d \to \mathbb{R}$ is called
\emph{regularising kernel} if there exist constants
$\beta > 0$, $m, r \in \mathbb{N}_0$ and $\rho > 0$
such that one can write
	\begin{equation}
		\sfK ( x, y ) = \sum\limits_{n = 0}^{+ \infty} \sfK_n ( x, y ) 
		\qquad \text{for a.e.\ } x,y \in \R^d \,,
		\label{eq:regularising_kernel_decomposition}
	\end{equation}
\noindent where for all $n\in\N_0$
the functions $\sfK_n \in C^{m,r}$ have the following support:
	\begin{enumerate}
		\item\label{item:regularising_3} 
		$\supp( \sfK_n ) \subset \lbrace (x, y) \colon
		\left| x-y \right| \leq \rho \, 2^{-n} \rbrace$,
	\end{enumerate}
and moreover, for any compact set $K \subset \mathbb{R}^d$,
there is a constant $c_{K} > 0$ such that
	\begin{enumerate}\setcounter{enumi}{1}
		\item\label{item:regularising_4} 
		for $k, l \in \mathbb{N}_0^d$ with $| k | \leq m$, $|l| \leq r$
		we have, for $x,y \in K$,
			\begin{equation} \label{eq:regularising_kernel_bound_on_derivatives}
				\left| \partial_1^k \partial_2^l \sfK_n ( x, y ) \right| 
				\le c_K \, 2^{(d - \beta + | l | + | k |)n} \,;
			\end{equation}

		\item\label{item:regularising_5} for $k, l \in \mathbb{N}_0^d$ with $| k |,|l| \leq r$
		we have, for $y \in K$,
			\begin{equation} \label{eq:reg-int}
				\left| \int_{\mathbb{R}^d} \left( y - x \right)^l \partial_2^k \sfK_n \left( x, y \right) 
				d x \right| \le c_K \, 2^{- \beta n}  \,.
			\end{equation}
	\end{enumerate}
We call such a function $\sfK(x,y)$ a \emph{$\beta$-regularising kernel of order $(m,r)$
with range~$\rho$.}
\end{definition}

Let us show that assumptions \eqref{item:regularising_3}, \eqref{item:regularising_4} 
and \eqref{item:regularising_5} are less restrictive than they might appear,
as they can be deduced from \eqref{eq:Kbasic} and from translation invariance.

\begin{remark}[Singular kernels]\label{rem:singular-kernel}
To fulfill assumptions \eqref{item:regularising_3} and
\eqref{item:regularising_4}, it is enough that
$\sfK$ satisfies condition \eqref{eq:Kbasic} 
and, correspondingly, for $| k | \leq m$, $|l| \leq r$,
	\begin{equation} \label{eq:Kbasic+}
		\left| \partial_1^k \partial_2^l \sfK ( x, y ) \right| 
		\lesssim \frac{1}{|x-y|^{d - \beta + | l | + | k |}}
		\, \mathds{1}_{\lbrace | y - x | \leq \rho \rbrace} \,,
	\end{equation}
uniformly for $x,y$ in compact sets. 
This can be seen via a dyadic partition of unity:
given $\chi \in \cD$ 
with $\ind_{\{|z| \le \rho \}} \le \chi(z) \le \ind_{\{|z| \le 2 \rho \}}$,
we set $\phi(z) := \chi(z) - \chi(2z)$ and define
\begin{equation} \label{eq:Kndyadic}
	\sfK_n(x,y) :=\sfK(x,y) \, \phi(2^n (x-y)) \,.
\end{equation}
Since $\sum_{n\in\N_0} \phi(2^n |z|) = \ind_{\{z\ne 0\}}$, it follows that
$\sfK( x, y) = \sum_{n = 0}^{\infty} \sfK_n (x,y) $  for $x \ne y$, and assumptions 
\eqref{item:regularising_3} and \eqref{item:regularising_4} follow by \eqref{eq:Kbasic+}.

Interestingly, in the boundary case $\beta = d$, we can weaken condition \eqref{eq:Kbasic}
allowing for a logarithmic divergence (see \cite[Remark~5.6]{Hai14}):
	\begin{equation} \label{eq:Kbasiclog}
		\left| \sfK ( x, y ) \right| 
		\lesssim \log (1+ |x-y|^{-1} )
		\, \mathds{1}_{\lbrace | y - x | \leq \rho \rbrace} \,,
	\end{equation}
and we can correspondingly weaken \eqref{eq:Kbasic+}, for $| k | \leq m$, $|l| \leq r$:
	\begin{equation} \label{eq:Kbasic+log}
		\left| \partial_1^k \partial_2^l \sfK ( x, y ) \right| 
		\lesssim \frac{\log (1+ |x-y|^{-1} )}{|x-y|^{| l | + | k |}}
		\, \mathds{1}_{\lbrace | y - x | \leq \rho \rbrace} \,.
	\end{equation}
In this case, it is convenient to modify
\eqref{eq:Kndyadic} as follows:
\begin{equation*} 
	\sfK_n(x,y) :=\sfK(x,y) \, \sum_{m=n}^\infty \frac{1}{m+1} \phi(2^m (x-y)) \,,
\end{equation*}
so that \eqref{item:regularising_3} is satisfied and we still have
$\sfK( x, y) = \sum_{n = 0}^{\infty} \sfK_n (x,y) $  for $x \ne y$.
To see that condition \eqref{item:regularising_4} is satisfied too, we
note that for $\rho \, 2^{-\ell-1} \le |x-y| \le \rho \, 2^{-\ell}$ we can bound
$|\sfK_n(x,y)| \lesssim \log (1+ \frac{2^{\ell+1}}{\rho}) 
\sum_{m \in \{\ell, \ell + 1\}} \frac{1}{m+1} \lesssim 1$ by \eqref{eq:Kbasiclog}
and, similarly,  $|\partial_1^k \partial_2^l\sfK_n(x,y)| \lesssim 2^{(|l|+|k|)\ell}$
by \eqref{eq:Kbasic+log}, uniformly over $\ell \ge n$.
\end{remark}

\begin{remark}[Translation invariance, I]\label{rem:translation}
If assumptions \eqref{item:regularising_3} and \eqref{item:regularising_4} 
are satisfied, assumption \eqref{item:regularising_5} is 
easily seen to hold for $|l| \ge |k|$.
Then the issue is whether \eqref{item:regularising_5} is satisfied for $|l| < |k|$.
This always holds in the \emph{translation invariant} case:
\begin{equation} \label{eq:trans-inv}
	\sfK_n ( x, y ) = \sfK_n ( x-y ) \qquad \forall x,y \in \R^d \,,
\end{equation}
because the integral in the l.h.s.\ of \eqref{eq:reg-int} vanishes for $|l| < |k|$,
as one sees through integration by parts,
since $\partial_2^k\sfK_n = (-1)^{|k|} \partial_1^k \sfK_n$ and $\partial_x^k (y-x)^l = 0$.
\end{remark}

In some cases, we will require a last assumption on the kernel.

\begin{assumption}[Preserving polynomials]\label{assumption:preserving_polynomial_annihilation}
Let $\sfK \colon \mathbb{R}^d \times \mathbb{R}^d \to \mathbb{R}$ admit
a decomposition $\sfK = \sum_{n = 0}^{+ \infty} \sfK_n$ as in 
\eqref{eq:regularising_kernel_decomposition}. For $\gamma \in \R$,
we say that $\sfK$ \emph{preserves polynomials at level $\gamma$} 
if, for every $n\in\N_0$ and for all $k\in\N_0^d$ with $0 \le |k| \le \gamma$,
\begin{equation}\label{eq:polypres}
	x \mapsto \int_{\mathbb{R}^d} \sfK_n ( x, y ) \, y^k \, \d y \quad
	\text{is a polynomial of degree $\le |k|$} \,.
\end{equation}
(This condition is automatically satisfied for $\gamma < 0$.)
\end{assumption}

\begin{remark}[Translation invariance, II]\label{rem:translation2}
A sufficient condition for \eqref{eq:polypres} is that
\begin{equation*}
	\int_{\mathbb{R}^d} \sfK_n ( x, y ) \, (y-x)^k \, \d y \quad
	\text{does not depend on~$x$} \,.
\end{equation*}
This condition clearly holds if the kernels $\sfK_n$ are translation invariant, see \eqref{eq:trans-inv},
in which case Assumption~\ref{assumption:preserving_polynomial_annihilation} 
is satisfied at any level $\gamma$.
\end{remark}

\begin{remark}
In \cite[Assumption~5.4]{Hai14} much more
than \eqref{eq:polypres} is required,
namely that for all multi-indices $k$ with $| k | \leq \gamma$ and any $n \in \mathbb{N}_0$
			\begin{equation*}
			\forall x \in \R^d: \qquad
				\int_{\mathbb{R}^d} \, \sfK_n ( x, y ) \,  y^k \, \d y = 0 \,.
			\end{equation*}
\end{remark}

We finally show that for translation invariant kernels 
$\sfK(x,y) = \sfK(x-y)$ the notion of $\beta$-regularising kernel
is greatly simplified.

\begin{lemma}[Translation invariant regularising kernel]\label{th:reg-trans}
Let $\beta,\rho > 0$ and $m, r \in \mathbb{N}_0$.
Fix a function $\sfK: \R^d \to \R$ such that, for all $k \in \mathbb{N}_0^d$ with $| k | \leq m+r$,
\begin{equation*} 
	\left| \partial^k \sfK ( z ) \right| 
	\lesssim \frac{1}{|z|^{d - \beta + | k |}}
	\, \mathds{1}_{\lbrace | z | \leq \rho \rbrace} \,.
\end{equation*}
Then $\sfK(x,y) := \sfK(x-y)$ is a $\beta$-regularising kernel of order $(m,r)$ with range $\rho$
and it preserves polynomial at any level~$\gamma \in \R$
(see Definition~\ref{def:regularising_kernel}
and Assumption~\ref{assumption:preserving_polynomial_annihilation}).
\end{lemma}

\begin{proof}
It suffices to apply Remarks~\ref{rem:singular-kernel}, \ref{rem:translation}
and~\ref{rem:translation2}.
\end{proof}

\begin{remark}[Scale-invariant kernels]
Given a smooth function
$\sfK \colon \mathbb{R}^d \setminus \lbrace 0 \rbrace \to \mathbb{R}$ with 
the scaling property $\sfK ( x / \lambda ) = \lambda^{d - \beta} \, \sfK ( x)$ for all $\lambda > 0$,
a $\beta$-regularising kernel is obtained by $\sfK(x,y) := \sfK(x-y) \, \chi(x-y)$,
where $\chi$ is any smooth function supported in $B(0,\rho)$.
This is a direct consequence of Lemma~\ref{th:reg-trans}, see
also \cite[Lemma~5.5]{Hai14}.

Examples of kernels falling in this situation include the Heat kernel, the Green's function of usual differential operators with constant coefficients, the Green's function of the fractional Laplacian \cite{MR3667364}, etc.
\end{remark}

\subsection{Singular integration and classical Schauder estimates}

The integration of a distribution $f \in \mathcal{D}^{\prime}(\mathbb{R}^d)$ 
with a kernel $\sfK(x,y)$ is formally defined by
\begin{equation*} 
	\sfK f (x) \coloneqq f(\sfK(x,\cdot)) = \int_{\R^d} \sfK(x,y) \, f(\d y) \,,
\end{equation*}
which makes sense when $\sfK(x,\cdot)$ is regular enough. If
$\sfK$ is singular, then one expects $\sfK f$ to be a distribution,
defined by duality on test functions $\psi \in \cD (\mathbb{R}^d)$ by
\begin{equation} \label{eq:K*}
	(\sfK f)(\psi) \coloneqq f(\sfK^*\psi) \qquad \text{where} \qquad
	(\sfK^*\psi)(y) \coloneqq \int_{\R^d} \psi(x) \, \sfK(x,y) \, \d x \,,
\end{equation}
provided $\sfK^*\psi$ is regular enough, so that $f(\sfK^*\psi)$ makes sense.

\begin{remark}[Translation invariance, III]\label{rem:translation3}
Formula \eqref{eq:K*} for $\sfK f$ is always well-defined
if the kernel $\sfK(x,y) = \sfK(x-y)$ is \emph{translation invariant}
with $\sfK: \R^d \to \R$ compactly supported and integrable:
indeed, $(\sfK^*\psi)(y) 
= \int_{\R^d} \psi(y-x) \, \sfK(-x) \, \d x$ in this case
is the \emph{convolution} of $\psi$ with $\sfK(-\,\cdot\,)$, hence $\sfK^*\psi$
is smooth (as $\psi$ is smooth)
and compactly supported (as $\psi$ and $\sfK$ are compactly supported).
\end{remark}

We now consider the case of a $\beta$-regularising kernel  $\sfK$ of order $(m,r)$, as
in Definition~\ref{def:regularising_kernel}.
From \eqref{eq:regularising_kernel_bound_on_derivatives} and Fubini's theorem
we can formally write
\begin{equation} \label{eq:K_n*}
		(\sfK^*\psi)(y) = \sum\limits_{n \in \mathbb{N}_0} ( \sfK_n^*\psi ) (y) 
		\qquad \text{where} \qquad
		(\sfK_n^*\psi)(y) \coloneqq \int_{\R^d} \psi(x) \, \sfK_n(x,y) \, \d x \, ,
\end{equation}
and note that $\sfK_n^*\psi$ is a well-posed  $C^r$ function, 
for any $n\in\N_0$, because $\sfK_n(x,\cdot) \in C^{r}$.
This means that $f (\sfK_n^*\psi)$ is well-defined
as soon as $f$ is a distribution of order~$r$,
see \eqref{eq:order} and the following lines.
In conclusion, we can set
	\begin{equation}\label{eq:definition_of_convolution}
		\sfK f ( \psi ) \coloneqq \sum\limits_{n \in \mathbb{N}_0} f (\sfK_n^* \psi) ,
	\end{equation}
provided the series converges. This is guaranteed by the next result.

\begin{proposition}[Singular integration]
\label{prop:well_posedness_convolution}
Given $r\in\N_0$, if $\sfK$ is a regularising kernel of order $(0, r)$
and $f$ is a distribution of order~$r$, then the integration $\sfK f$ is well-defined by \eqref{eq:definition_of_convolution} and it is a
distribution of order~$r$.

If $\sfK$ is a regularising kernel of order $(0, r)$ \emph{for any $r\in\N_0$},
then the integration $\sfK f$ is well-defined for any distribution $f \in \cD'$.
\end{proposition}

We can finally show that the 
integration by a $\beta$-regularising kernel $\sfK$ improves the H\"older regularity 
of a distribution by $\beta$: this result is known as the classical Schauder estimates
and can be stated as follows (see also \cite[Theorem~14.17]{MR4174393}).

\begin{theorem}[Classical Schauder estimates]\label{thm:classical_schauder}
Let $\gamma \in \mathbb{R}$.
Let $\sfK$ be a $\beta$-regularising kernel of order $(m, r)$,
where $\beta > 0$ and $m,r \in \mathbb{N}_0$ satisfy:
	\begin{equation} \label{eq:condmr}
		m > \gamma + \beta , \quad r > - \gamma .
	\end{equation}
Also assume (if $\gamma \ge 0$) 
that $\sfK$ preserves polynomials at level~$\gamma$,
see Assumption~\ref{assumption:preserving_polynomial_annihilation}.
Then, integration by $\sfK$ in \eqref{eq:definition_of_convolution} defines a continuous linear map from 
$\mathcal{Z}^{\gamma}$ to $\mathcal{Z}^{\gamma + \beta}$.
\end{theorem}

We prove Proposition~\ref{prop:well_posedness_convolution} and Theorem~\ref{thm:classical_schauder} 
in Section~\ref{section:preliminary_tools_I} below.

\section{Main result I: Schauder estimates for germs}\label{section:schauder_estimate_on_germs}

	\subsection{Germs}

Our first goal is to extend Theorem~\ref{thm:classical_schauder} in the context of germs, that is, families of distributions indexed by $\mathbb{R}^d$.
\begin{definition}[Germs]
A \emph{germ} is a family $F = (F_x)_{x \in \mathbb{R}^d}$ of distributions $F_x \in \mathcal{D}^{\prime} ( \mathbb{R}^d)$, such that for any $\varphi \in \mathcal{D} ( \mathbb{R}^d)$, the map $x \mapsto F_x ( \varphi )$ is measurable.
\end{definition}

We will denote $\mathcal{G}$ the vector space of germs.
In general, we will see a germ $F \in \mathcal{G}$ as a family of local approximations of a global distribution $f$.
The reconstruction problem, i.e.\ the problem of constructing a suitable $f$ from $F$, has been previously considered in a number of different contexts, see \cite{Hai14, CZ20,zorinkranich2021reconstruction}.
In \cite{CZ20}, it was established that this reconstruction can be performed under the assumption that $F$ satisfies properties named homogeneity and coherence, which we recall now.

\begin{definition}[Homogeneity and coherence]\label{def:homogeneity_coherence}
Let $F = (F_x)_{x\in\R^d}$ be a germ.
Let $\bar\alpha, \alpha, \gamma \in \R$ with $\bar\alpha, \alpha \le \gamma$ and $r \in \N_0$.
\begin{itemize}
\item $F$ is called \emph{$\bar\alpha$-homogeneous of order~$r$}, denoted
$F \in \cG^{\bar\alpha}_{\ho,r}$, if
the following \emph{homogeneity property} holds,
for any compact $K \subset \mathbb{R}^d$ and $\bar\lambda \in [1,\infty)$:
\begin{equation}\label{eq:homogeneity}
\begin{gathered}
	| F_x ( \varphi_x^{\lambda})| \lesssim \lambda^{\bar{\alpha}} \\
	\text{uniformly over $x \in K$, $\lambda \in (0, \bar\lambda]$, $\varphi \in \mathscr{B}_{}^{r}$} .
\end{gathered}
\end{equation}
The space of $\bar\alpha$-homogenous germs (of any order)
is $\cG^{\bar\alpha}_{\ho} := \bigcup\limits_{r\in\N_0} \cG^{\bar\alpha}_{\ho,r}$.

\item $F$ is called \emph{$(\alpha,\gamma)$-coherent
of order~$r$}, denoted $F \in \cG^{\alpha,\gamma}_{\co,r}$,
if the following \emph{coherence property} holds, 
for any compact $K \subset \mathbb{R}^d$ and $\bar\lambda \in [1,\infty)$:
\begin{equation}\label{eq:coherence}
\begin{gathered}
	| ( F_y - F_x) ( \varphi_x^{\lambda})| \lesssim \lambda^{\alpha} 
	(| y - x| + \lambda)^{\gamma - \alpha} \\
	\text{uniformly over $x,y \in K$, $\lambda \in (0, \bar\lambda]$, 
	$\varphi \in \mathscr{B}_{}^{r}$} .
\end{gathered}
\end{equation}
The space of $(\alpha,\gamma)$-coherent germs (of any order) is
$\cG^{\alpha,\gamma}_{\co} := \bigcup\limits_{r\in\N_0} \cG^{\alpha,\gamma}_{\co,r}$.

\item $F$ is called \emph{$(\alpha,\gamma)$-coherent with homogeneity $\bar{\alpha}$} if both \eqref{eq:homogeneity} and \eqref{eq:coherence} hold, for some order~$r$.
The space of such germs is denoted by
\begin{equation*}
	\cG^{\bar\alpha; \alpha,\gamma}
	= \cG^{\bar\alpha}_{\ho} \cap \cG^{\alpha,\gamma}_{\co} \,.
\end{equation*}
\end{itemize}
\end{definition}

\begin{remark}[Monotonicity]
Increasing the exponents $\bar\alpha, \alpha, \gamma$, homogeneity
and coherence become more restrictive:
$\cG^{\bar\alpha'; \alpha',\gamma'} \subseteq \cG^{\bar\alpha; \alpha,\gamma}$
for any $\bar\alpha' \ge \bar\alpha$, $\alpha'\ge\alpha$, $\gamma'\ge\gamma$.
\end{remark}

\begin{remark}[Coherence almost implies homogeneity]
Any coherent germ automatically satisfies the homogeneity relation
\eqref{eq:homogeneity} with an exponent $\bar\alpha_K$
that may depend on the compact~$K$, see \cite[Lemma~4.12]{CZ20}.
Requiring that a coherent germ is $\bar\alpha$-homogeneous simply means
that we can take $\bar\alpha_K \ge \bar\alpha$ for any compact~$K$.
\end{remark}

\begin{remark}[Homogeneity implies some coherence]
In the definition of $(\alpha,\gamma)$-coherent germs with homogeneity $\bar{\alpha}$,
we require that $\bar\alpha \le \gamma$ for convenience,
to rule out trivialities. Indeed, if a germ $F$ is $\bar\alpha$-homogeneous, then
arguing as in \cite[Proposition~6.2]{CZ20}
one can show that $F$ is $(\alpha,\gamma)$-coherent with $\gamma = \bar\alpha$
(and suitable~$\alpha$).
This shows that the ``interesting regime'' for the coherence
exponent is $\gamma \ge \bar\alpha$.
\end{remark}

\begin{remark}[Uniformity of $r$ for coherence + homogeneity]\label{rem:uniformity}
If either relation \eqref{eq:homogeneity} or \eqref{eq:coherence} 
holds for some order $r$, then it holds for all orders $r' > r$, simply because
$\mathscr{B}_{}^{r'} \subseteq \mathscr{B}_{}^{r}$. 
If these relations hold together, i.e.\ \emph{if a germ is both 
$(\alpha,\gamma)$-coherent and $\bar{\alpha}$-homogeneous}, then we can choose the ``canonical'' order $r=r_{\bar\alpha,\alpha}$ given by
\begin{equation}\label{eq:canonical-r}
\begin{split}
	r_{\bar\alpha,\alpha} 
	:=&\, \min\big\{r\in\N \colon r > \max\{-\bar\alpha,-\alpha\}\big\} \\
	=&\, (\lfloor \max\{-\bar{\alpha},-\alpha\} \rfloor + 1)^+ \,.
\end{split}
\end{equation}
Indeed, if \eqref{eq:homogeneity} and \eqref{eq:coherence} 
hold for some $r > \max\{-\bar{\alpha},-\alpha\}$, it turns out that they
also hold for $r = r_{\bar\alpha,\alpha}$, see Proposition~\ref{prop:independence_in_r}
in Appendix~\ref{section:independence_in_r}.\footnote{We also point out
\cite[Propositions~13.1 and~13.2]{CZ20}: leaving aside for simplicity the case $\gamma=0$,
it is shown that if both \eqref{eq:homogeneity} and \eqref{eq:coherence} hold 
for a \emph{single test function $\varphi \in \cD$ with $\int\varphi \ne 0$},
then they hold uniformly over $\varphi \in \mathscr{B}^r$ for $r=r_{\bar\alpha,\alpha}$
as in \eqref{eq:canonical-r} (the proof requires $\alpha \le 0$, however when $\alpha > 0$
one can show that a $(\alpha,\gamma)$-coherent germ must be constant, hence
the conclusion still holds).}
\end{remark}

\begin{remark}[Bounded order and singular integration]\label{rem:bounded-order}
If a germ $F = (F_x)_{x\in\R^d}$ is $\bar\alpha$-homogeneous of order~$r$, then
\emph{each distribution $F_x$ is of order~$r$}.
This follows by Remark~\ref{rem:singint}, see \eqref{eq:order2bis} with $z=x$, because 
$|F_x(\varphi_x^{\bar\lambda})| \lesssim \bar\lambda^{\bar\alpha} \lesssim 1$ 
uniformly for $\varphi \in \mathscr{B}^r$, by \eqref{eq:homogeneity},
for any given $\bar\lambda \in \N$.
As a consequence, \emph{we can define the germ
$\sfK F = (\sfK F_x)_{x\in\R^d}$ for any regularising kernel $\sfK$ of order $(0,r)$}, 
by Proposition~\ref{prop:well_posedness_convolution}.

Similarly, if a germ $F = (F_x)_{x\in\R^d}$ is $(\alpha,\gamma)$-coherent of order~$r$,
then for any $x,y \in\R^d$
\emph{the difference $F_y - F_x$ is a distribution of order~$r$},
because $|(F_y-F_x)(\varphi_x^{\bar\lambda})| \lesssim 1$  by \eqref{eq:coherence},
uniformly for $\varphi \in \mathscr{B}^r$,
for any given $\bar\lambda \in (0,\infty)$.
Thus \emph{we can define $\sfK (F_y - F_x)$ for any regularising kernel $\sfK$ of order $(0,r)$},
but not necessarily $\sfK F_x$.
\end{remark}

\begin{remark}[General scales]\label{remark:lower_scales}
For germs $F \in \cG^{\bar\alpha; \alpha,\gamma}$
that are \emph{both coherent and homogeneous} we can get rid of $\bar\lambda$:
if we assume that both relations \eqref{eq:homogeneity} and \eqref{eq:coherence}
hold for $\bar\lambda = 1$, then they hold for any $\bar\lambda \in [1,\infty)$.
To this goal, we claim that for
$\lambda \in [1, \bar\lambda]$ we have $|F_y(\varphi_x^\lambda)| \lesssim 1$,
which is enough
since $1  \lesssim \lambda^{\bar\alpha}$
and $1 \lesssim \lambda^\alpha (| y - x| + \lambda)^{\gamma - \alpha}$.

First note that by \eqref{eq:homogeneity} and \eqref{eq:coherence}
for $\lambda \in (0, 1]$
we have, for $x,y \in K$ and $\varphi \in \mathscr{B}_{}^{r}$,
\begin{equation} \label{eq:boundedness}
	|F_y(\varphi_x^\lambda)| \le
	|(F_y-F_x)(\varphi_x^\lambda)| + |F_x(\varphi_x^\lambda)| \lesssim
	\lambda^{\alpha} + \lambda^{\bar\alpha} \lesssim
	\lambda^{\min\{\bar\alpha,\alpha\}} \,.
\end{equation}
If we now consider $\lambda \in [1, \bar\lambda]$, 
using a partition of unity we can write
$\varphi_x^{\lambda} = \sum_{k = 1}^n (\hat{\varphi}_k)_{x_k}^{1}$ 
for suitable $x_k \in K$, 
$\hat{\varphi}_k \in \mathscr{B}^r$ and $n$ (uniformly bounded, depending on 
$\bar{\lambda}$ and $K$). Then by \eqref{eq:boundedness} for $\lambda=1$ we obtain
$|F_y(\varphi_x^\lambda)| \lesssim 1$ for $\lambda \in [1, \bar\lambda]$, as we claimed.
\end{remark}

Let us now introduce the semi-norms for the relations
of homogeneity and coherence, see \eqref{eq:homogeneity} and
\eqref{eq:coherence}:
given a compact set $K \subset \R^d$, $r \in \N$ and $\bar\lambda \in [1,\infty)$, we set
\begin{align}\label{eq:semi-norm-homogeneity}
	\| F \|_{\mathcal{G}_{\mathrm{hom}; K,\bar\lambda,r}^{\bar{\alpha}}}
	& := \sup_{\substack{x\in K, \, \lambda \in (0,\bar\lambda] \\
	\varphi \in \mathscr{B}_{}^{r}}} 
	\frac{| F_x ( \varphi_x^{\lambda})|}{\lambda^{\bar{\alpha}}} \,,
\end{align}
\begin{align}	\label{eq:semi-norm-coherence}
	\| F \|_{\mathcal{G}_{\mathrm{coh}; K,\bar\lambda,r}^{\alpha, \gamma}}
	& := \sup_{\substack{x,y\in K, \, \lambda \in (0,\bar\lambda] \\
	\varphi \in \mathscr{B}_{}^{r}}} 
	\frac{| ( F_y - F_x) ( \varphi_x^{\lambda})|}%
	{\lambda^{\alpha} (| y - x| + \lambda)^{\gamma - \alpha}} \,.
\end{align}
We next define the joint semi-norm for homogeneous and coherent germs,
where we fix 
$r = r_{\bar\alpha,\alpha}$ as in Remark~\ref{rem:uniformity},
see \eqref{eq:canonical-r}:
\begin{equation}\label{eq:semi-norm-coherence+homogeneity}
\begin{gathered}
	\| F \|_{\mathcal{G}_{K,\bar\lambda}^{\bar\alpha;\alpha, \gamma}} 
	:= \| F \|_{\mathcal{G}_{\mathrm{hom}; K,\bar\lambda,r_{\bar\alpha,\alpha}}^{\bar{\alpha}}} +
	\| F \|_{\mathcal{G}_{\mathrm{coh}; K,\bar\lambda,r_{\bar\alpha,\alpha}}^{\alpha, \gamma}} \,.
\end{gathered}
\end{equation}
We often set $\bar\lambda = 1$ and omit it from notation.
Note that a germ $F$ is $(\alpha,\gamma)$-coherent with homogeneity~$\bar\alpha$ if and only if 
$\| F \|_{\mathcal{G}_{K}^{\bar\alpha;\alpha, \gamma}}  < \infty$
for any compact set $K \subset \R^d$.

\subsection{Reconstruction}

The Reconstruction Theorem was originally formulated in \cite[Theorem~3.10]{Hai14}, 
see also \cite{OW19}.
We present here the version given by
\cite[Theorem~5.1]{CZ20}
and \cite{zorinkranich2021reconstruction}
(we exclude the case $\gamma = 0$
to avoid introducing logarithmic corrections).

\begin{theorem}[Reconstruction for $\gamma \neq 0$]
\label{thm:reconstruction}
Let $\alpha, \gamma \in \mathbb{R}$ with $\alpha \le \gamma$ and $\gamma \ne 0$. 
For any germ $F = (F_x)_{x\in\R^d}$ 
which is $(\alpha, \gamma)$-coherent, there exists a distribution 
\begin{equation*}
	\mathcal{R}^\gamma F \in \mathcal{D}^{\prime} ( \mathbb{R}^d ) \,,
\end{equation*}
called a \emph{$\gamma$-reconstruction of $F$},
which is \enquote{locally approximated 
by $F$} in the following sense: for any integer $r\in\N_0$ with $r > - \alpha$ 
and for any compact $K \subset \mathbb{R}^d$ we have
\begin{equation}\label{eq:reconstruction_bound}
\begin{gathered}
		| ( F_x - \mathcal{R}^\gamma F ) ( \varphi_x^{\lambda} ) | \lesssim \lambda^{\gamma} \\
		\text{uniformly  over $x \in K$, $\lambda \in ( 0, 1 ]$, $\varphi \in \mathscr{B}^r$} .
\end{gathered}
\end{equation}
Such a distribution $\cR^\gamma F$ is unique if and only if $\gamma > 0$. Furthermore:
\begin{itemize}
\item for any $\gamma$, one can define $\cR^\gamma F$ so that the map
$F \mapsto \mathcal{R}^\gamma F$ is linear;

\item if the germ $F$ has homogeneity $\bar{\alpha} \leq \gamma$, 
then $\mathcal{R}^\gamma F \in \mathcal{Z}^{\bar{\alpha}}$, i.e.\
	\begin{equation}\label{eq:target}
		\mathcal{R}^\gamma \colon \mathcal{G}^{\bar\alpha; \alpha, \gamma} 
		\to \mathcal{Z}^{\bar{\alpha}} ;
	\end{equation}

\item if $\bar\alpha > 0$, then $\cR^\gamma F = 0$.
\end{itemize}
\end{theorem}

With an abuse of notation, we sometimes write $f = \cR^\gamma F$ to mean
that a distribution $f\in\cD'$ is a $\gamma$-reconstruction of $F$,
i.e.\ it satisfies \eqref{eq:reconstruction_bound},
but we stress that when $\gamma < 0$ there are many such a reconstruction~$f$ is not unique.
If the value of $\gamma$ is clear from the context, we may omit it and simply write
that $f = \cR F$ is a reconstruction of~$F$.

\begin{remark}[Non uniqueness is tame]
For $\gamma < 0$ there is no unique $\gamma$-reconstruction $\cR^\gamma F$, but 
\emph{any two $\gamma$-reconstructions differ by a distribution in $\cZ^\gamma$}. This follows
comparing \eqref{eq:reconstruction_bound} and \eqref{eq:Cgamma}, 
because the precise value of $r > -\gamma$ in the definition of $\cZ^\gamma$
is immaterial,
see Proposition~\ref{prop:independence_in_r} in Appendix~\ref{section:independence_in_r}.
\end{remark}

\begin{remark}
The fact that $\mathcal{R}^\gamma = 0$ when $\bar\alpha > 0$ follows by \eqref{eq:homogeneity}
and \eqref{eq:reconstruction_bound}, which yield
$|\cR^\gamma F(\varphi_x^\lambda)| \lesssim \lambda^{\bar\alpha}+\lambda^{\gamma}
\lesssim \lambda^{\bar\alpha}$, since $\bar\alpha\le\gamma $.
If $\bar\alpha > 0$, this implies that $\cR^\gamma F = 0$.
\end{remark}

\begin{remark}[Reconstruction bounds]\label{rem:reconstruction-bounds}
If a germ $F$ is $(\alpha,\gamma)$-coherent, the germ 
$F - \cR^\gamma F = ( F_x - \mathcal{R}^\gamma F )_{x\in\R^d}$
is \emph{not only $(\alpha,\gamma)$-coherent but also $\gamma$-homogeneous},
as the bound \eqref{eq:reconstruction_bound} shows,
i.e.\ $F - \cR^\gamma F \in  \mathcal{G}_{}^{\gamma; \alpha, \gamma}$.
More precisely, by \cite{CZ20,zorinkranich2021reconstruction},
\begin{equation}\label{eq:RT-bounds}
	\left\| F - \mathcal{R}^\gamma F 
	\right\|_{\mathcal{G}_{K, \bar\lambda}^{\gamma;
		\alpha, \gamma}} 
	\lesssim \left\| F 
	\right\|_{\mathcal{G}_{\mathrm{coh}; K^{\prime}, \bar\lambda}^{\alpha,\gamma}} ,
\end{equation}
for the enlarged compact $K^{\prime} = K \oplus B (0, \bar{\lambda}+1)$.
If $F$ has homogeneity $\bar\alpha$, then
\begin{equation*}
	\left\| \mathcal{R}^\gamma F \right\|_{\mathcal{Z}_{K,\bar\lambda}^{\bar{\alpha}}} 
		\lesssim \big\| F 
		\big\|_{\mathcal{G}_{K^{\prime}, \bar\lambda}^{\bar{\alpha}; \alpha, \gamma}} .
\end{equation*}
\end{remark}

	\subsection{Schauder estimates for coherent germs}

A natural and interesting problem is to find a \enquote{nice} continuous linear map 
$\mathcal{K}$ which 
\enquote{lifts the integration with $\sfK$ on the space of coherent
(resp.\ coherent and homogeneous) germs}.
More precisely, given $\bar{\alpha}, \alpha, \gamma \in \mathbb{R}$
and a $\beta$-regularising kernel $\sfK$,
we look for a 
continuous linear map $\cK = \mathcal{K}^{\gamma,\beta}$ such that the following diagrams commute,
for suitable $\bar{\alpha}', \alpha', \gamma' \in \mathbb{R}$:
\begin{equation}\label{eq:diagram_coherent_germs}
	\begin{tikzcd}
		\mathcal{G}^{\alpha, \gamma} 
		\arrow{r}{\mathcal{K}^{\gamma,\beta}} \arrow[swap]{d}{\mathcal{R}^\gamma} & 
		\mathcal{G}^{\alpha^{\prime},  \gamma^{\prime}} 
		\arrow{d}{\mathcal{R}^{\gamma'}} \\%
		\cD' \arrow{r}{\sfK}& \cD'
	\end{tikzcd}	
	\qquad\quad
	\begin{tikzcd}
		\mathcal{G}^{\bar{\alpha} ;\alpha, \gamma} 
		\arrow{r}{\mathcal{K}^{\gamma,\beta}} \arrow[swap]{d}{\mathcal{R}^\gamma} & 
		\mathcal{G}^{\bar{\alpha}^{\prime};\alpha^{\prime},  \gamma^{\prime}} 
		\arrow{d}{\mathcal{R}^{\gamma'}} \\%
		\mathcal{Z}^{\bar{\alpha}} \arrow{r}{\sfK}& \mathcal{Z}^{\bar{\alpha}'}
	\end{tikzcd}	
\end{equation}
that is $\cR^{\gamma'}(\cK^{\gamma,\beta}F) = \sfK(\cR^\gamma F)$.
In particular, we need to assume that
\emph{the integration $\sfK (\cR^\gamma F)$ is well-defined}. 
This is a mild condition, as we now discuss.

\begin{remark}[Integration of reconstruction]\label{rem:intrec}
The integration \emph{$\sfK (\cR^\gamma F)$ is always well-defined if the kernel $\sfK$
is translation invariant}, see Remark~\ref{rem:translation3}.

For germs $F$ that are \emph{$\bar\alpha$-homogeneous} for some $\bar\alpha \in \R$,
the integration
$\sfK (\cR^\gamma F)$ is well defined 
\emph{if the regularising kernel $\sfK$ is of order $(m,r)$ with $r > -\bar\alpha$},
by Proposition~\ref{prop:well_posedness_convolution},
because $\cR^\gamma F \in \cZ^{\bar\alpha}$
(see Theorem~\ref{thm:reconstruction})
is a distribution of order~$r$ (see Remark~\ref{rem:singint}).

Finally, for \emph{non homogeneous germs $F$},
the integration $\sfK (\cR^\gamma F)$ is still well defined 
\emph{if the regularising kernel $\sfK$ is of order $(m,r)$ for any $r\in\N_0$},
again by Proposition~\ref{prop:well_posedness_convolution}.
\end{remark}

We next discuss the definition of $(\mathcal{K}^{\gamma,\beta} F)_x$.
A naive guess would be to define it as
	\begin{equation}\label{eq:KconvF}
	\sfK F_x 
	\end{equation}
but this choice of germ is typically neither coherent nor homogeneous.
However, it turns out that we can nicely modify \eqref{eq:KconvF} by 
\emph{subtracting a suitable polynomial term}.

\begin{remark}
One ``trivial'' solution would be to define $(\mathcal{K}^{\gamma,\beta} F )_x$ 
for all $x\in\R^d$ by
	\begin{equation*}
		\sfK (\mathcal{R}^\gamma F)  \,.
	\end{equation*}
However, such a germ is independent of~$x$ and does not contain $F_x$.
This is not useful for applications (e.g.\ to stochastic equations)
where one needs germs which do depend on~$x$, to reflect the local fluctuations of the noise.
\end{remark}

As a first ingredient (of independent interest), we show in the next Lemma 
how to define \emph{pointwise derivatives}
for any distribution which is ``locally homogeneous'' on test functions
that \emph{annihilate polynomials}, and we prove that subtracting
a Taylor polynomial yields a homogeneity bound
for general test functions. 

\begin{lemma}[Pointwise derivatives]\label{lemma:pointwise_derivative}
Let $f \in \cD'$ be a distribution which satisfies a \emph{``weak'' homogeneity bound
at a given point $x \in \R^d$}, for some $\delta > 0$:
\begin{equation} \label{eq:local-hom}
	\text{for any $\varphi \in \mathscr{B}_{\delta}$} : \qquad
	|f(\varphi_{x}^\lambda)| \lesssim \lambda^{\delta} \quad
	\text{uniformly for $\lambda \in (0,1]$}
\end{equation}
where we recall that functions in $\mathscr{B}_{\delta}$ annihilate polynomials of degree $\le \delta$.

Then $f$ admits ``pointwise derivatives''
of any order $<\delta$, defined by
\begin{equation} \label{eq:definition_Dkf(x)2}
	D^k f (x) := \lim_{\lambda \downarrow 0} D^{k} f(\eta_{x}^\lambda) \in \R
	\qquad \forall k \in \N_0^d \ \text{ with } \ 0 \le |k| < {\delta} \,,
\end{equation}
for any $\eta \in \cD$ with
$\int \eta = 1$ and $\int \eta(x) \, x^l \, \d x = 0$ for all $1 \le |l| < {\delta}$
(the limit does not depend on the choice of such~$\eta$).
We can thus define the \emph{Taylor polynomial at~$x$}
\begin{equation}\label{eq:Taylor-f}
	\cT^\delta_x(f)(\,\cdot\,) := \sum_{0 \le |k| < \delta} D^k f (x)\,
	\frac{(\,\cdot\,-x)^k}{k!} \,.
\end{equation}

If moreover $\delta \not\in \N$,
then $f - \cT^{\delta}_x(f)$
satisfies a \emph{``strong'' homogeneity bound at~$x$}:
\begin{equation} \label{eq:gbound}
	|(f - \cT^{\delta}_x(f))(\psi_{x}^\lambda)| \lesssim \lambda^{\delta} \quad
	\text{uniformly for $\lambda \in (0,1]$} \,,
\end{equation}
for any test function $\psi \in \cD(B(0,1))$ which needs not annihilate polynomials.

Finally, the bound \eqref{eq:gbound} holds also for $\delta \in \mathbb{N}$ 
if the following condition holds:
for any $k \in \mathbb{N}_0^d$ with $| k | = \delta$ and any $\varphi \in \mathcal{D} ( B ( 0, 1 ) )$,
one has $\sup_{\lambda \in (0,1]} | D^k f ( \varphi_x^{\lambda} ) | \lesssim 1$.
\end{lemma}

\begin{remark}
We point out that Lemma~\ref{lemma:pointwise_derivative} provides
a \emph{local version}, for a fixed base point~$x$, of the following well-known result in H\"older 
spaces (see e.g. \cite[Proposition~A.5]{broux2021besov}, or
\cite[Proposition~14.15]{MR4174393}): for a distribution 
$f \in \mathcal{D}^{\prime} ( \mathbb{R}^d )$ and an exponent 
$\delta > 0$ with $\delta \not\in \mathbb{N}$, there is equivalence between:
\begin{ienumerate}
	\item\label{it:i} $f \in \mathcal{C}^{\delta}$ i.e.\ 
	$|f ( \varphi_x^{\lambda}) | \lesssim \lambda^{\delta}$ over $x$ in compacts, $\lambda \in (0, 1]$, 
	$\varphi \in \mathscr{B}_{\delta}$.
	\item\label{it:ii} $f$ is a $C^{\lfloor \delta \rfloor}$ function and 
	$\left| f (y ) - \sum_{| k | < \delta} \frac{\partial^k f ( x )}{x !} (y - x)^k \right| \lesssim |y - x|^{\delta}$.
	\end{ienumerate}
In the case of integer exponents, only the implication \eqref{it:ii} $\Rightarrow$ \eqref{it:i} holds.
\end{remark}

Consider now an $(\alpha,\gamma)$-coherent germ $F = (F_x)_{x\in\R^d}$
and a $\gamma$-reconstruction $\cR^\gamma F$ (which is unique if $\gamma > 0$).
Given a $\beta$-regularizing kernel $\sfK$,
we will show in Theorem~\ref{thm:convolution_germs_in_G_checkG} that 
$f = \sfK \{ F_x - \mathcal{R}^\gamma F\}$ satisfies \eqref{eq:local-hom}
with $\delta = \gamma+\beta$, hence we can consider
\begin{equation}\label{eq:definition_of_KF}
\begin{split}
	(\cK^{\gamma,\beta} F )_x 
	& \coloneqq \sfK F_x - \cT^{\gamma+\beta}_x\big( \sfK \{F_x - \cR^\gamma F\} \big) \\
	& = \sfK F_x - \sum\limits_{0 \le |k| < \gamma + \beta} 
	D^k( \sfK \{ F_x - \mathcal{R}^\gamma F\} )( x )\,
	\frac{ (\,\cdot - x)^k}{k!} \,,
\end{split}
	\end{equation}
that is, we subtract 
from $\sfK F_x$ 
the Taylor polynomial at~$x$ of $\sfK \{F_x - \cR^\gamma F\}$.
In case $\gamma+\beta \le 0$, we agree that $\cT^{\gamma+\beta}_x \equiv 0$,
that is, we define $(\cK^{\gamma,\beta} F )_x  \coloneqq \sfK F_x$.

Note that the difference $(\cK^{\gamma,\beta} F )_x  - \sfK(\cR^\gamma F)$ admits the expression
\begin{equation}\label{eq:definition_of_KF_alt}
	(\cK^{\gamma,\beta} F )_x - \sfK(\cR^\gamma F)
	\coloneqq \sfK \{F_x-\cR^\gamma F\}
	- \cT^{\gamma+\beta}_x\big( \sfK \{F_x - \cR^\gamma F\} \big) \,,
\end{equation}
which is always well-defined when $F$ is a coherent germ,
as we now discuss.

\smallskip

We can now state our first main result, that we prove in Section~\ref{section:proof_main_result_I}.

\begin{theorem}[Schauder estimates for coherent germs]\label{thm:Schauder_for_germs}
Let $\alpha, \gamma \in \mathbb{R}$ with $\alpha \le \gamma$ and $\gamma \neq 0$.
Consider an $(\alpha, \gamma)$-coherent germ $F = (F_x)_{x\in\R^d}$
and a $\gamma$-reconstruction $\cR^\gamma F$
(which is unique if $\gamma>0$), see \eqref{eq:reconstruction_bound}.
Let $\beta > 0$ satisfy
	\begin{equation} \label{eq:assumption_remove_polynomial}
		\alpha + \beta \neq 0, 
		\quad \gamma + \beta \notin \mathbb{N}_0 \,,
	\end{equation}
and consider a $\beta$-regularising kernel $\sfK$  with range $\rho$
of order $(m, r)$ large enough:
	\begin{equation}\label{eq:assumption_schauder_on_germs}
		m > \gamma + \beta , \quad r > - \alpha \,.
	\end{equation}	
Then the germ $\cK^{\gamma,\beta} F - \sfK(\cR^\gamma F)$
is well-defined by \eqref{eq:definition_of_KF_alt} and it satisfies
\begin{equation}\label{eq:main-gen}
	\cK^{\gamma,\beta} F - \sfK(\cR^\gamma F) \in 
	\cG^{\gamma+\beta;(\alpha+\beta)\wedge 0,\gamma+\beta} \,,
\end{equation}
i.e.\ it is $(\gamma+\beta)$-homogeneous and
$((\alpha + \beta) \wedge 0, \gamma + \beta)$-coherent, with
the following continuity estimate: 
for any compact $K \subseteq \R^d$ and $\bar\lambda \in (0,\infty)$ we have
\begin{equation}\label{eq:cont-schauder1}
	\| \cK^{\gamma,\beta} F - \sfK(\cR^\gamma F) 
	\|_{\mathcal{G}_{K,\bar\lambda}^{\gamma+\beta; (\alpha+\beta)\wedge 0, \gamma+\beta}}
	\le \mathrm{cst}_{K,\bar\lambda} \,
	\| F - \cR^\gamma F
	\|_{\mathcal{G}_{K,\bar\lambda'}^{\gamma; \alpha, \gamma}} \,,
\end{equation}
with $\bar\lambda' = 8(\bar\lambda + \rho + \mathrm{diam}(K))$, and the RHS
can be bounded by \eqref{eq:RT-bounds}.

If we further assume that $\sfK(\cR^\gamma F)$ is well-defined,
see Remark~\ref{rem:intrec}, then also the germ $\cK^{\gamma,\beta} F$
is well-defined by \eqref{eq:definition_of_KF} and the following holds:
\begin{itemize}
\item $\cK^{\gamma,\beta}F$ is $((\alpha + \beta) \wedge 0, \gamma + \beta)$-coherent
and (with some abuse of notation)
\begin{equation} \label{eq:reco-schauder}
	\mathcal{R}^{\gamma+\beta} (\cK^{\gamma,\beta} F) = \sfK ( \mathcal{R}^\gamma F) \,,
\end{equation}
i.e.\ $\sfK (\mathcal{R}^\gamma F)$ is a $(\gamma+\beta)$-reconstruction 
of $\cK^{\gamma,\beta} F$;
	
\item if $F$ has homogeneity $\bar{\alpha} \leq \gamma$,
then $\cK^{\gamma,\beta} F$ has homogeneity $(\bar{\alpha} + \beta) \wedge 0$,
provided
\begin{equation*}
	\bar{\alpha} + \beta \neq 0 \,;
\end{equation*}

\item the map $F \mapsto \cK^{\gamma,\beta}F$ is linear and continuous, both on $\cG_{\mathrm{coh}}^{\alpha,\gamma} \to \cG_{\mathrm{coh}}^{(\alpha+\beta)\wedge 0,\gamma+\beta}$ 
and on $\cG^{\bar\alpha;\alpha,\gamma} \to 
\cG^{(\bar\alpha+\beta)\wedge 0;(\alpha+\beta)\wedge 0,\gamma+\beta}$, with the following continuity estimate: for any compact $K \subseteq \mathbb{R}^d$ and $\bar{\lambda} \in (0, \infty)$, 
	\begin{align}
		\| \cK^{\gamma,\beta}F \|_{\mathcal{G}_{\mathrm{coh}; K, \bar{\lambda}}^{(\alpha + \beta) \wedge 0, \gamma + \beta}} & \lesssim \mathrm{cst}_{K, \bar{\lambda}} \| F \|_{\mathcal{G}_{\mathrm{coh}; K^{\prime}, \bar{\lambda}^{\prime}}^{\alpha, \gamma}} , \label{eq:cont-schauder2} \\
		\| \cK^{\gamma,\beta}F \|_{\mathcal{G}_{K, \bar{\lambda}}^{(\bar{\alpha} + \beta) \wedge 0 ; (\alpha + \beta) \wedge 0, \gamma + \beta}} & \lesssim \mathrm{cst}_{K, \bar{\lambda}} \| F \|_{\mathcal{G}_{K^{\prime}, \bar{\lambda}^{\prime}}^{\bar{\alpha}; \alpha, \gamma}} , \label{eq:cont-schauder3}
	\end{align}
\noindent with $K^{\prime} = K \oplus B (0, 1 + 2 \bar{\lambda})$, $\lambda^{\prime} = 8(\bar\lambda + \rho + \mathrm{diam}(K))$.
\end{itemize}

As a consequence, if we set
\begin{gather*}
	(\bar{\alpha}', \alpha', \gamma') = 
	((\bar{\alpha} + \beta) \wedge 0, (\alpha + \beta) \wedge 0, \gamma + \beta) \,,
\end{gather*}
both diagrams in \eqref{eq:diagram_coherent_germs} commute,
where $\cK^{\gamma,\beta}$ is linear and continuous.
\end{theorem}

We finally show that the definition \eqref{eq:definition_of_KF} of the
$\cK^{\gamma,\beta} F$ is \emph{canonical}, in the following precise sense.
The proof of the next result is also given in Section~\ref{section:proof_main_result_I}.

\begin{proposition}[Canonicity of the germ $\cK^{\gamma,\beta} F$]
\label{thm:canonicity}
If we postulate that
\begin{equation} \label{eq:postulate}
	\cK^{\gamma,\beta} F = \sfK F_x - P_x \,,
\end{equation}
then the only polynomial $P_x$ of degree $< \gamma+\beta$
such that \eqref{eq:reco-schauder} holds (i.e.\ such that the diagrams in
\eqref{eq:diagram_coherent_germs} commute) is
$P_x = \cT^{\gamma+\beta}_x(\sfK\{F_x - \cR^\gamma F\})$ 
as in \eqref{eq:definition_of_KF}.
\end{proposition}

\section{Main result II: multilevel Schauder estimates}
\label{section:multi_level_schauder_estimate}

In many applications, 
the space $\mathcal{G}^{\bar{\alpha}; \alpha, \gamma}$ of all coherent and homogeneous germs
is \enquote{too big}. This happens for instance
when one wants to define singular operations on germs,
such as the product with a non smooth function, or even a distribution:
one can typically make sense of such a product \emph{only for a few germs $\Pi^i$},
hence the best one can hope is to extend the product
to \emph{those germs that are locally given by linear combinations of the $\Pi^i$'s}. This leads to the notion
of \emph{models} and \emph{modelled distributions},
which are cornerstones of the theory of regularity structures~\cite{Hai14}.

\subsection{Models and modelled distributions} \label{section:models_modelled_distributions}

We fix a family $\Pi = ( \Pi^i)_{i\in I}$ 
of germs $\Pi^i = (\Pi^i_x)_{x\in\R^d}$ on $\R^d$ indexed by a finite set~$I$.
We view this family as a \emph{basis to build  germs} $F = \langle f, \Pi \rangle$
through linear combinations 
with coefficients $f^i(x)$:
	\begin{equation} \label{eq:germbase}
		F_x = \langle f, \Pi \rangle_x :=
		\sum_{i \in I} f^i ( x ) \, \Pi_x^i \,.
	\end{equation}
We will call the basis $\Pi$ a \emph{model} and the family of coefficients $f$ a
\emph{modelled distribution}, provided they satisfy assumptions
that we now discuss.

To define a model $\Pi = ( \Pi^i)_{i\in I}$, we require that each germ $\Pi^i$
is \emph{homogeneous}, and furthermore that the vector space $\mathrm{Span} \{\Pi_x^i\colon i\in I\} \subset \cD'(\R^d)$
\emph{does not depend on~$x$} (i.e.\  $\Pi^i_x$ is a linear combination of $(\Pi^j_y)_{j\in I}$,
for any $x,y \in \R^d$). This leads to:

\begin{definition}[Model]\label{def:model}
Fix a finite set $I$ and a family $\balpha = (\alpha_i)_{i\in I}$
of real numbers. 
A pair $M = ( \Pi, \Gamma )$ is called a \emph{model} on $\R^d$ with homogeneities 
$\balpha$ if there exists an integer $r = r_{\Pi} \in \mathbb{N}_0$
(called the ``order'' of the model) such that:
\begin{enumerate}
\item $\Pi = ( \Pi^i )_{i\in I}$ is a family of germs on $\R^d$ 
that are $\alpha_i$-homogeneous of order~$r$, that is
for any compact $K\subset \R^d$ and any $\bar\lambda \in [1,\infty)$
\begin{equation}\label{eq:homogeneity-Pi}
\begin{gathered}
	| \Pi^i_x(\varphi_x^\lambda) | \lesssim \lambda^{\alpha_i} \\
	\text{uniformly over $x \in K$, $\lambda \in (0, \bar\lambda]$, $\varphi \in \mathscr{B}_{}^{r}$} \,;
\end{gathered}
\end{equation}
		
\item $\Gamma = (\Gamma^{ji}_{xy})$ are real numbers such that, for all $i\in I$
and $x,y \in \R^d$,
	\begin{equation}\label{eq:relation_Pi_Gamma}
		\Pi_y^{i} = \sum\limits_{j \in I} \Pi_{x}^j \, \Gamma_{xy}^{ji} \, .
	\end{equation}
\end{enumerate}

We denote by $\cM^{\balpha}$ 
the class of models with homogeneities $\balpha$
and we set, see \eqref{eq:semi-norm-homogeneity},
	\begin{equation}\label{eq:norm-model}
		\| \Pi \|_{\mathcal{M}^{\balpha}_{K, \bar\lambda}} 
		\coloneqq \sup_{i \in I} 
		\| \Pi^i \|_{\mathcal{G}_{\mathrm{hom}; K, \bar\lambda, r_{\Pi}}^{\alpha_i}} \,.
	\end{equation}
\end{definition}

\begin{remark}[Models in Regularity Structures]\label{remark:comparison_definitions_of_model}
Our definition of a model is more general than Hairer's original definition \cite[Definition~2.17]{Hai14},
because we do not enforce the following requirements:
\begin{enumerate}
		\item\label{it:uno}
		\emph{Group Property}: \  $\Gamma_{x y} \,\Gamma_{y z} = \Gamma_{x z}$ \
		(that is $\sum_{k\in I}\Gamma_{x y}^{jk} \,\Gamma_{y z}^{ki} = \Gamma_{x z}^{ji}$);
		\item\label{it:due} 
		\emph{Triangular Structure}: \
		$\Gamma_{x y}^{i i} = 1, \Gamma_{x y}^{ji} = 0$ if $j\ne i$ and ${\alpha}_j \ge {\alpha}_i$;
		\item\label{it:tre} 
		\emph{Analytic Bound}: \
		$| \Gamma_{x y}^{ji} | \lesssim |y - x|^{{\alpha}_i - {\alpha}_j}$. 
\end{enumerate}
Property \eqref{it:uno} is natural, in view of \eqref{eq:relation_Pi_Gamma}
(indeed, when the $\Pi^i$'s are linearly independent, 
the coefficients $\Gamma^{ji}_{xy}$ are univocally determined by \eqref{eq:relation_Pi_Gamma} 
and \eqref{it:uno} holds automatically).
The role of properties \eqref{it:due} and \eqref{it:tre} is discussed
below, see Remark~\ref{rem:modreg}.

If the analytic bound \eqref{it:tre} holds, we can define the norm
\begin{equation}\label{eq:norm-gamma}
	\| \Gamma \|_{\mathcal{M}_{K}^{\balpha}} 
	\,\coloneqq\, \max_{i, j \in I}
	\, \sup_{x, y \in K} \,
	\frac{| \Gamma_{x y}^{ji}|}{|x - y|^{\alpha_i - \alpha_j}} \,.
\end{equation}
\end{remark}

\begin{remark}[Bounded order and general scales]
Given a model $M = (\Pi, \Gamma)$ of order $r=r_\Pi$,
\emph{each $\Pi^i_x$ is a distribution of order~$r$},
see Remark~\ref{rem:bounded-order},
hence we can define $\sfK \Pi^i_x$ for any regularising kernel
of order $(0,r)$.

We also note that if $(x,y) \mapsto \Gamma^{ji}_{xy}$ is \emph{locally bounded}
(e.g., if the analytic bound \eqref{it:tre} holds),
then it is sufficient to require the
homogeneity property \eqref{eq:homogeneity-Pi} for $\bar\lambda = 1$.
This can be shown as in Remark~\ref{remark:lower_scales}, see \eqref{eq:boundedness}
and the following lines.
\end{remark}

\begin{example}[Polynomial model]\label{ex:polynomial-model}
The simplest choice of a model is obtained taking as basis of germs the usual (normalized) monomials
\begin{equation} \label{eq:bbX}
	\mathbb{X}_x^k \coloneqq \frac{(\cdot - x)^k }{k !} \,, \qquad
	k \in \N_0^d \,.
\end{equation}
More precisely, if we fix any $\ell \in \N_0$, the \emph{polynomial model at level $\ell$} is defined by
\begin{equation*}
	\Pi^\poly_{\le \ell} := \big\{\Pi^k = \bbX^k\big\}_{k\in \N_0^d \colon |k| \le \ell} \,.
\end{equation*}
It is an exercise to check that $\Pi^\poly_{\le \ell}$ is indeed a model,
as in Definition~\ref{def:model}, with
\begin{equation} \label{eq:alphaGammapoly}
	\alpha_k := |k| \,, \qquad
	(\Gamma^\poly)^{lk}_{xy} := \frac{(x-y)^{k-l}}{(k-l)!} \, \ind_{\{l \le k\}} \,,
\end{equation}
where by $l \le k$ we mean $l_1 \le k_1$, $l_2 \le k_2$, \ldots, $l_d \le k_d$.
It is also easy to check that the three additional properties \eqref{it:uno}, \eqref{it:due}
and \eqref{it:tre} described in Remark~\ref{remark:comparison_definitions_of_model}
are satisfied by the polynomial model.
\end{example}

We next define modelled distributions. Consider a germ $F = \langle f,\Pi\rangle$
as in \eqref{eq:germbase}, for some model $\Pi = (\Pi^i)_{i\in I}$.
Applying \eqref{eq:relation_Pi_Gamma}, for any $x, y \in \mathbb{R}^d$ we can write
	\begin{equation}\label{eq:FyFx}
		F_y - F_x = \sum\limits_{i \in I} \bigg\{ \sum\limits_{j \in I} \Gamma_{xy}^{i j} \,
		f^j ( y ) - f^i (x) \bigg\} \, \Pi_x^i\,.
	\end{equation}
In order to ensure that $F$ is coherent, it is natural to require scaling properties of the quantities
in brackets. This leads to the following definition:

\begin{definition}[Modelled distribution]\label{def:modelled_distribution}
Consider a model $M = (\Pi, \Gamma)$ with homogeneities
$\balpha = (\alpha_i)_{i\in I}$ and fix a real number $\gamma > \max\balpha
:= \max \{\alpha_i \colon i\in I\}$.

A measurable function $f = (f^i(x))_{i\in I}: \R^d \to \R^I$ is called \emph{modelled distribution of order $\gamma$} if for any compact set
$K \subset \R^d$ and for any $i\in I$, uniformly for $x,y\in K$,
\begin{equation*} \label{eq:rel-kappa}
	| f^i ( x ) | \lesssim 1 \qquad \text{and} \qquad
	\bigg| \sum\limits_{j \in I} \Gamma_{xy}^{i j} \, f^j ( y ) - f^i (x) \bigg|
	\lesssim |y-x|^{\gamma-\alpha_i}  \,.
\end{equation*}
We denote by $\mathcal{D}^{\gamma} =
\mathcal{D}^{\gamma}_M = \mathcal{D}_{\Gamma, \balpha}^{\gamma}$ 
the space of modelled distributions of order $\gamma$,
relative to a model $M = (\Pi,\Gamma)$ with homogeneities $\balpha$.
This is a vector space with a Fr\'echet structure
through the semi-norms
\begin{equation}\label{eq:definition_modelled_distribution_norm}
		\left\vvvert f \right\vvvert_{\mathcal{D}^{\gamma}_K} 
		=
		\left\vvvert f \right\vvvert_{\mathcal{D}_{\Gamma, \balpha; K}^{\gamma}} 
		\coloneqq \sup\limits_{\substack{x \in K, \, i \in I}} 
		| f^i ( x ) | + \sup\limits_{\substack{x, y \in K ,\, i \in I}} 
		\frac{\Big| \sum\limits_{j \in I}
		\Gamma_{x y}^{i j} \, f^j ( y ) - f^i (x) \Big|}{| y - x |^{\gamma - \alpha_{i}}} \,.
\end{equation}
\end{definition}

\begin{remark}[Consequences of additional properties]\label{rem:modreg}
Our definition of modelled distributions 
mimics Hairer's original one \cite[Definition~3.1]{Hai14}.
The additional properties of the models enforced in \cite{Hai14},
see Remark~\ref{remark:comparison_definitions_of_model},
ensure that \emph{modelled distributions can be truncated}:
given a model $M = (\Pi^i, \Gamma^{ji})_{i,j\in I}$ with 
homogeneities $\balpha = (\alpha_i)_{i\in I}$, 
denoting by $I^{\prime} \coloneqq \lbrace i \in I \colon \alpha_i < \gamma^{\prime} \rbrace$ 
the truncation of the index set $I$ at a given level 
$\gamma^{\prime} > \min\balpha$ (so that $I^\prime \ne \emptyset$), then:
\begin{itemize}

\item the truncated family
$M^{\prime} \coloneqq M|_{I^\prime} = (\Pi^{i}, \Gamma^{ji})_{i, j \in I^{\prime}}$ is also a model,
thanks to property \eqref{it:due} (triangular structure);

\item given a modelled distribution $f = (f^i)_{i\in I}$ of order $\gamma > \gamma^\prime$
relative to $M$,
the truncated function $f' \coloneq f|_{I^\prime} = (f^{i})_{i \in I^{\prime}}$
is a modelled distribution of order $\gamma^\prime$ relative to $M^\prime$,
thanks to property \eqref{it:tre} (analytic bound).

\end{itemize}
Property \eqref{it:due} 
also ensures that
the spaces $D^\gamma$ contain non-zero elements:
if $i_0 \in I$ is such that $\alpha_{i_0} = \min\balpha$, then
$\Gamma^{ji_0} = \ind_{j=i_0}$, 
hence defining $f_{i_0}(x) \equiv 1$ and $f_j(x) \equiv 0$ for $j\ne i_0$
yields $f \not\equiv 0$ with $f = (f_i)_{i\in I} \in \cD^\gamma$
(note that $\langle f, \Pi \rangle = \Pi^{i_0}$).
\end{remark}

For any modelled distribution $f$ relative to a model $M=(\Pi,\Gamma)$,
we now check that the germ $F = \langle f,\Pi\rangle$ in \eqref{eq:germbase} is 
coherent and homogeneous, see \cite[Example~4.10]{CZ20}.

\begin{proposition}[Modelled distributions yield coherent germs]
\label{prop:modelled_distribution_is_coherent_germ}
Let $M=(\Pi, \Gamma)$ be a model with homogeneities $\balpha = (\alpha_i)_{i\in I}$
and set $\bar\alpha := \min \balpha$.

For any modelled distribution $f \in \cD^\gamma_M$ of order~$\gamma$,
the germ $F = \langle f, \Pi \rangle$ in \eqref{eq:germbase}
is $\gamma$-coherent, more precisely it is
$(\bar\alpha,\gamma)$-coherent with homogeneity~$\bar\alpha$:
\begin{equation*}
	f \in \mathcal{D}^\gamma \quad \Longrightarrow \quad
	F = \langle f,\Pi \rangle \in \cG^{\bar\alpha; \bar\alpha,\gamma} \,.
\end{equation*}
Moreover, the map $f \mapsto F = \langle f, \Pi \rangle$ is continuous:
\begin{equation} \label{eq:boundnormF}
	\| F \|_{\mathcal{G}_{K,\bar\lambda}^{\bar\alpha;\bar\alpha, \gamma}}
	\le |I| \, \| \Pi \|_{\mathcal{M}^{\balpha}_{K, \bar\lambda}} 
	\left\vvvert f \right\vvvert_{\mathcal{D}_{K}^{\gamma}}
	\,.
\end{equation}
\end{proposition}

\begin{proof}
By \eqref{eq:semi-norm-homogeneity},
the homogeneity semi-norm of $F$ can be bounded by
\begin{equation} \label{eq:1bound}
	\| F \|_{\mathcal{G}_{\mathrm{hom}; K,\bar\lambda,r}^{\bar{\alpha}}}
	\le |I| \, \sup_{x\in K, i \in I} |f^i(x)| \, 
	\| \Pi^i \|_{\mathcal{G}_{\mathrm{hom}; K,\bar\lambda,r}^{\alpha_i}}
	\le |I| \, c^f_1 \,  \| \Pi \|_{\mathcal{M}^{\balpha}_{K, \bar\lambda}} \,,
\end{equation}
where $c^f_1$ denotes the first term in the r.h.s.\ of \eqref{eq:definition_modelled_distribution_norm},
see \eqref{eq:norm-model}. Turning to coherence,
by \eqref{eq:FyFx} we can bound, arguing as in \cite[Example~4.10]{CZ20},
\begin{equation*}
	|(F_y - F_x)(\varphi_x^\lambda)|
	\le |I| \, c^f_2 \, \| \Pi \|_{\mathcal{M}^{\balpha}_{K, \bar\lambda}}
	\,(\lambda+ |y-x|)^{\gamma-\bar\alpha} \, \lambda^{\bar\alpha} \,,
\end{equation*}
where $c^f_2$ denotes the second term in the r.h.s.\ of \eqref{eq:definition_modelled_distribution_norm}.
Then, by \eqref{eq:semi-norm-coherence}, we obtain
\begin{equation*}
	\| F \|_{\mathcal{G}_{\mathrm{coh}; K,\bar\lambda,r}^{\bar{\alpha},\gamma}}
	\le |I| \, c^f_2 \, \| \Pi \|_{\mathcal{M}^{\balpha}_{K, \bar\lambda}} \,,
\end{equation*}
which together with \eqref{eq:1bound} yields \eqref{eq:boundnormF}.
\end{proof}

\begin{example}[Polynomial modelled distributions]\label{ex:polynomial-modelled-distributions}
Let $f: \R^d \to \R$ be a function of class $C^\ell$, 
for some $\ell \in \N_0$. Its Taylor polynomial of order $\ell$ based at~$x$ is
\begin{equation*}
	F_x(\cdot) := \sum_{|k| \le \ell} \partial^k f(x) \, \bbX_x^k(\cdot) \,,
\end{equation*}
where $\bbX_x^k$ are normalized monomials, see \eqref{eq:bbX}.
The germ $F = (F_x)_{x\in\R^d}$ 
can be expressed as $F = \langle \boldsymbol{f}, \Pi^\poly_{\le \ell}\rangle$, see \eqref{eq:germbase},
where $\Pi^\poly_{\le \ell}$ is the polynomial model 
in Example~\ref{ex:polynomial-model} and
$\boldsymbol{f} = (\boldsymbol{f}^k(x))_{|k|\le \ell, x\in\R^d}$  is defined by 
$\boldsymbol{f}^k(x) := \partial^k f (x)$.

If  
$f$ is H\"older continuous with exponent 
$\gamma > 0$ and $\ell = \lfloor \gamma \rfloor$,
it is an exercise
to show that $\boldsymbol{f}$ is a modelled distribution
of order~$\gamma$, see e.g.\ \cite[Example~4.11]{CZ20}.
In particular, by Proposition~\ref{prop:modelled_distribution_is_coherent_germ},
the germ $F = (F_x)_{x\in\R^d} = \langle \boldsymbol{f}, 
\Pi^\poly_{\le \ell}\rangle$ of Taylor polynomials of~$f$ is $\gamma$-coherent
(more precisely: $(0,\gamma)$-coherent with homogeneity~$0$).
\end{example}

\subsection{Schauder estimates for modelled distributions}

Given a model $M=( \Pi, \Gamma)$ 
and a modelled distribution $f \in \mathcal{D}^\gamma_M$,
by Proposition~\ref{prop:modelled_distribution_is_coherent_germ} we have that
\begin{equation*}
	\text{$F = \langle f, \Pi \rangle$ in \eqref{eq:germbase} is $\gamma$-coherent} \,.
\end{equation*}
If we fix a $\gamma$-reconstruction $\cR^\gamma F$
(which is unique if $\gamma > 0$) and a $\beta$-regularising kernel $\sfK$,
the Schauder estimates
in Theorem~\ref{thm:Schauder_for_germs} yield that
\begin{equation*}
	\text{$\cK^{\gamma,\beta} F$ in \eqref{eq:definition_of_KF} is $(\gamma+\beta)$-coherent}
	\qquad \text{and} \qquad
	\cR^{\gamma+\beta}(\cK^{\gamma,\beta} F) = \sfK (\cR^\gamma F) \,.
\end{equation*}
Since the germ $F = \langle f, \Pi \rangle$ comes from 
a modelled distribution $f$, a natural question arises:
\emph{do we have $\cK^{\gamma,\beta} F = \langle \hat f, \hat \Pi \rangle$
	for some model $\hat\Pi$ and modelled distribution $\hat f \,$?}

Our next main result
shows that the answer is positive: see Theorem~\ref{thm:multi_level_schauder_estimate} below, 
which generalizes Hairer's \emph{multilevel Schauder estimates}
\cite[Theorem~5.12]{Hai14} as well as Hairer's
\emph{extension theorem} \cite[Theorem~5.14]{Hai14}.
We first need to define the new model $(\hat\Pi,\hat\Gamma)$
and the new modelled distribution $\hat f$.

\paragraph*{New model $(\hat\Pi,\hat\Gamma)$} 
The new model is labelled by a new set $\hat{I}$, obtained by adding
to $I$ all multi-indexes of homogeneity up to $\gamma+\beta$:
\begin{equation} \label{eq:hatI}
	\hat{I} := {I} \sqcup \poly(\gamma+\beta) \qquad
	\text{where} \qquad 
	\poly(t) := \{ k \in \N_0^d: \ |k| < t \} 
\end{equation}
where $\sqcup$ denotes the disjoint union and we agree that $\poly(t) = \emptyset$ for $t \le 0$.

The germs $\hat\Pi = (\hat\Pi_a)_{a\in \hat{I}}$ in the new model are defined by
\begin{equation} \label{eq:hatPi}
	\hat\Pi^a_x := \begin{cases}
	\rule{0pt}{1.3em}
	\displaystyle\sfK \Pi^i_x - \sum_{k \in \poly(\alpha_a+\beta)} 
	D^k (\sfK \Pi^i_x)(x) \, \bbX^k_x
	& \text{ if } a=i \in {I} \,, \\
	\bbX^k_x & \text{ if } a=k \in \poly(\gamma+\beta) \,,
	\end{cases} 
\end{equation}
with homogeneities $\bhatalpha = (\hat\alpha_a)_{a\in \hat{I}}$ given by
\begin{equation} \label{eq:hatalpha}
	\hat\alpha_a := \begin{cases}
	\alpha_i + \beta & \text{ if } a=i \in {I} \,, \\
	\rule[-.5em]{0pt}{1.8em}|k| & \text{ if } a=k \in \poly(\gamma+\beta)  \,.
	\end{cases}
\end{equation}
We will show that $\hat\Pi_a$ is well defined, 
thanks to Proposition~\ref{prop:well_posedness_convolution} and Lemma~\ref{lemma:pointwise_derivative},
and it satisfies the homogeneity condition \eqref{eq:homogeneity-Pi}
with exponent $\hat\alpha_a$.

We next define the coefficients $\hat\Gamma = (\hat\Gamma^{ba}_{xy})_{b,a\in \hat{I}}$.
Using labels $i,j \in {I}$
and $k,l \in \poly(\gamma+\beta)$ for clarity, we have the triangular structure
\begin{equation} \label{eq:hatGamma0}
	\hat \Gamma^{ba}_{xy} =
	\left( \begin{matrix}
	\Gamma^{ji}_{xy} & 0 \\
	\cdots & (\Gamma^\poly)^{lk}_{xy}
	\end{matrix}\right)
	= \begin{cases}
	\Gamma^{ji}_{xy} & \text{if } (b,a)=(j,i) \in {I} \times {I} \,, \\
	0 & \text{if } (b,a) = (j,k) \in {I} \times \poly(\gamma+\beta) \,, \\
	\cdots & \text{if } (b,a) = (l,i) \in \poly(\gamma+\beta) \times {I} \,, \\
	(\Gamma^\poly)^{lk}_{xy} & \text{if } (b,a) = (l,k) \in \poly(\gamma+\beta) 
	\times \poly(\gamma+\beta),
	\end{cases}
\end{equation}
where $\Gamma$ are the coefficients of the original model while $\Gamma^\poly$
are those of
the polynomial model, see \eqref{eq:alphaGammapoly}.
It only remains to define $\cdots = \hat\Gamma^{li}_{xy}$
for $l\in \poly(\gamma+\beta)$ and $i\in {I}$:
\begin{equation}\label{eq:hatGamma}
	\hat\Gamma^{li}_{xy} := 
	\sum_{\substack{j\in I \colon\\ \alpha_j+\beta > |l|}}  
	D^l (\sfK \Pi^j_x)(x) \, \Gamma^{ji}_{xy}
	\ - 
	\sum_{k\in\poly(\alpha_i+\beta)} 
	(\Gamma^\poly)^{lk}_{xy}  \, D^k (\sfK \Pi^i_y)(y) \,.
\end{equation}
(The second sum is restricted to $k \ge l$,
because $(\Gamma^\poly)^{lk}_{xy} = 0$ otherwise, see \eqref{eq:alphaGammapoly}.
Also note that $\hat\Gamma^{li}_{xy} \ne 0$ only for $|l| \le
\max\balpha+\beta = \max_{j\in I} \alpha_j + \beta$.)

We will check by direct computation that condition \eqref{eq:relation_Pi_Gamma} 
in the definition of a model
is satisfied by $\hat\Pi$ and $\hat\Gamma$, see Section~\ref{sec:hatmodel}.

\paragraph*{New modelled distribution $\hat{f}$}

Given a
modelled distribution $f = (f^i(x))_{i \in I}$ relative to the original
model $(\Pi,\Gamma)$, we define for $a \in \hat{I} = {I} \sqcup \poly(\gamma+\beta)$
\begin{equation}\label{eq:hatf}
	\hat f^a(x) :=
	\begin{cases}
	\rule{0pt}{1.1em}f^i(x) & \text{ if } a=i \in {I} \,, \\
	\rule{0pt}{3.8em}\displaystyle\begin{aligned}
	& \!\!\sum_{\substack{j \in I \colon \\
	\alpha_j + \beta > |k|}} f^j(x) \, D^k (\sfK \Pi^j_x) (x)  \\
	&\rule{0pt}{1.3em} \ \ -
	D^k (\sfK \{\langle f,\Pi\rangle_x - \cR^\gamma\langle f,\Pi\rangle\}) (x) 
	\end{aligned}
	& \text{ if } a=k \in \poly(\gamma+\beta) \,.
	\end{cases}
\end{equation}
(We point out that the three lines in the r.h.s.\ of \eqref{eq:hatf} correspond precisely
to the three terms $\cI, \cJ, \cN$ in the setting of Regularity Structures,
see \cite[(5.15)]{Hai14}.)

We will prove that $\hat{f}$ is indeed a modelled
distribution of order $\gamma+\beta$ relative to the new model
$(\hat\Pi, \hat\Gamma)$, see Section~\ref{sec:proof_hat_f_is_modelled_distribution}.

\begin{remark}
For $t \in \R$, we define the restriction $\cQ_{\le t} f$
of a modelled distribution $f = (f^i(x))_{i\in I}$ where we only keep the components
$f^i(x)$ with $\alpha_i \le t$, that is
\begin{equation*}
	\cQ_{\le t} f (x) := \big( f^i(x) \, \ind_{\{\alpha_i \le t\}} \big)_{i\in I} \,.
\end{equation*}
We can then rewrite \eqref{eq:hatf} more compactly as follows:
\begin{equation*}
	\hat f^a(x) :=
	\begin{cases}
	\rule{0pt}{1.1em}f^i(x) & \text{ if } a=i \in {I} \,, \\
	\rule{0pt}{1.8em}
	\displaystyle{D^k \big(\sfK \big\{\cR\langle f,\Pi\rangle - 
	\langle {\mathcal Q}_{\le |k|-\beta}f,\Pi\rangle_x
	\big\}\big) (x)}
	& \text{ if } a=k \in \poly(\gamma+\beta) \,.
	\end{cases}
\end{equation*}
\end{remark}

\paragraph*{Compatibility condition}

Before stating our multilevel Schauder estimates, we state
a technical condition on the model $M = ( \Pi, \Gamma)$ and the
kernel $\sfK$.

\begin{assumption}[Compatibility]\label{ass:compatibility}
A model $M = ( \Pi, \Gamma)$ with homogeneities
$\balpha = (\alpha_i )_{i \in I}$ and a $\beta$-regularising kernel $\sfK$
are called \emph{compatible} if 
\begin{equation}\label{eq:techn}
	( D^k ( \sfK \Pi_x^{i} ) )_{x \in \mathbb{R}^d} \quad \text{ is a $0$-homogeneous germ whenever} 
	\quad |k| = \alpha_i + \beta \in \N_0 \, .
\end{equation}
We denote for $K \subset \mathbb{R}^d$ and $\bar{\lambda} > 0$
	\begin{equation} \label{eq:technorm}
	[\sfK\Pi]_{K, \bar{\lambda}}
	\coloneqq \sum\limits_{i \in I} \ \sum\limits_{k\in\N_0^d\colon | k | = \alpha_i + \beta} \,
	\| D^k ( \sfK \Pi_x^{i} ) \|_{\mathcal{G}_{\mathrm{hom}; K, \bar{\lambda}, r_{\Pi}}^{0}} \,,
	\end{equation}
where we agree that $[\sfK\Pi]_{K, \bar{\lambda}} := 0$
if the sum is empty (i.e.\ $\alpha_i + \beta \not\in \N_0$ for all $i\in I$).
\end{assumption}

\begin{remark}[Compatibility is mild]\label{rem:compatibility}
Condition \eqref{eq:techn}
is \emph{trivially satisfied if $\alpha_i + \beta$ is non-integer} for any $i\in I$.
Even when some $\alpha_i + \beta$ is an integer, one can ensure
compatibility by slightly decreasing $\beta>0$
to $\beta' \in (0,\beta)$ so that all $\alpha_i + \beta'$ are non integer
(note that a $\beta$-regularising kernel $\sfK$ is also $\beta'$-regularising).

We also note that condition \eqref{eq:techn} is fulfilled when
$\sfK \Pi_x^{i}$ is a polynomial of degree $< \alpha_i + \beta$,
because $D^k(\sfK \Pi_x^{i}) \equiv 0$ in this case.
In particular, \emph{if the kernel
$\sfK$ preserves polynomials 
(see  Assumption~\ref{assumption:preserving_polynomial_annihilation}),
condition \eqref{eq:techn} only applies to non-polynomial germs $\Pi^i$}.
Thus, in practice, the assumption of compatibility is often void.
\end{remark}

\begin{remark}[Strong vs.\ weak homogeneity]
Since by assumption each germ $\Pi^i$ is $\alpha_i$-homogeneous,
\Cref{thm:convolution_germs_in_G_checkG} below implies that 
when $\alpha_i + \beta = | k |$ the germ
$( D^k ( \sfK \Pi_x^{i} ) )_{x \in \mathbb{R}^d}$ is always \emph{weakly} $0$-homogeneous, 
see \Cref{def:weak_homogeneity_coherence} below.
This further shows that \Cref{ass:compatibility} is not very demanding.
\end{remark}

\paragraph*{Multilevel Schauder estimates}

We can finally state our second main result.
Recall the order $r_\Pi \in \N_0$ of a model $M = (\Pi,\Gamma)$,
see Definition~\ref{def:model}.

\begin{theorem}[Multilevel Schauder estimates] \label{thm:multi_level_schauder_estimate}
Let $M = ( \Pi, \Gamma)$ be a model with homogeneities
$\balpha = (\alpha_i )_{i \in I}$ and order $r=r_\Pi \in \N_0$.
Fix $\gamma > \max\balpha$.

Let $f \in \cD^\gamma_M$ 
be a modelled distribution of order $\gamma$ relative to $M = (\Pi,\Gamma)$, so that
\begin{equation*}
	\langle f, \Pi \rangle := \bigg( \sum_{i \in I} f^i ( x ) \, \Pi_x^i \bigg)_{x \in \R^d}
	\quad \text{is a $\gamma$-coherent germ} \,,
\end{equation*}
and fix a $\gamma$-reconstruction $\cR^\gamma\langle f, \Pi \rangle$ 
(which is unique if $\gamma > 0$).

Fix $\beta > 0$ with $\gamma+\beta \not\in \N_0$ and let $\sfK$ be a $\beta$-regularising kernel
of order $(m,r)$ large enough:
\begin{equation} \label{eq:multi-level-cond}
		m > \gamma + \beta, \qquad r \geq r_{\Pi} \,,
\end{equation}
such that $M=(\Pi,\Gamma)$ and $\sfK$ are compatible
(see Assumption~\ref{ass:compatibility} and Remark~\ref{rem:compatibility}).

\smallskip

Then we can define:
\begin{itemize}
\item  a new model $\hat M = (\hat\Pi, \hat\Gamma)$, 
see \eqref{eq:hatPi} and \eqref{eq:hatGamma0}-\eqref{eq:hatGamma},
indexed by $\hat{I}$ in \eqref{eq:hatI}
with homogeneities $\bhatalpha = (\hat\alpha_a)_{a \in \hat{I}}$
in \eqref{eq:hatalpha} and with order $r_{\hat \Pi}=r_\Pi$;

\item a new modelled distribution $\hat f \in \cD^{\gamma+\beta}_{\hat M}$ 
of order $\gamma + \beta$ relative to
$\hat M = (\hat\Pi, \hat\Gamma)$, see \eqref{eq:hatf}, so that
\begin{equation*}
	\langle \hat f, \hat \Pi \rangle := \bigg( \sum_{a \in \hat{I}} \hat f^a ( x ) \, \hat\Pi_x^a 
	\bigg)_{x \in \R^d}
	\quad \text{is a $(\gamma+\beta)$-coherent germ} \,;
\end{equation*}
\end{itemize}
in such a way that the following equality holds, with $\cK^{\gamma,\beta}$ 
defined by \eqref{eq:definition_of_KF}:
\begin{equation} \label{eq:prophatf}
	\langle \hat f, \hat \Pi \rangle
	= \cK^{\gamma,\beta}\langle f, \Pi \rangle  \,.
\end{equation}

In particular, by Theorem~\ref{thm:Schauder_for_germs}, we have
(with some abuse of notation)
\begin{equation} \label{eq:cRhat}
	\cR^{\gamma+\beta} \langle \hat f, \hat \Pi \rangle = \sfK\, (\cR^\gamma 
	\langle f,  \Pi \rangle) \,,
\end{equation}
i.e.\ $\sfK \,(\mathcal{R}^\gamma \langle f,  \Pi \rangle )$ is a $(\gamma+\beta)$-reconstruction 
of $\langle \hat f, \hat \Pi \rangle$
(which is unique if $\gamma+\beta > 0$).
\end{theorem}

The proof of Theorem~\ref{thm:multi_level_schauder_estimate} is given in Section~\ref{sec:multilevel}, and we proceed as follows.
	\begin{itemize}
		\item In Section~\ref{sec:hatmodel}, we prove that $\hat M = (\hat\Pi, \hat\Gamma)$ 
		is indeed a model: we first check the condition of reexpansion \eqref{eq:relation_Pi_Gamma} 
		for $\hat\Pi$ and $\hat\Gamma$ by a direct computation; then we show that each $\hat\Pi^i_x$ satisfies the homogeneity relation \eqref{eq:homogeneity-Pi} with exponent $\hat{\alpha}_i$.
		
		\item In Section~\ref{sec:hatf}, we prove \eqref{eq:prophatf} by a simple calculation;
		then relation \eqref{eq:cRhat} follows as an immediate consequence of
		Theorem~\ref{thm:Schauder_for_germs}, see \eqref{eq:reco-schauder}.

		\item In Section~\ref{sec:proof_hat_f_is_modelled_distribution}, we prove that $\hat f$ is 
		indeed a modelled distribution, and we also prove a continuity estimate,
		see \eqref{eq:conhatf} below.
	\end{itemize}

\begin{remark}
We can rephrase Theorem~\ref{thm:multi_level_schauder_estimate} 
by stating that the map $\cK^{\gamma,\beta}$ acting on germs can be lifted to a map $f \mapsto \hat f$  acting on modelled
distributions, defined by \eqref{eq:hatf},
so that the following diagram commutes:
\begin{equation*}
	\begin{tikzcd}
		\cD^{\gamma}_M
		\arrow{r}{\hat{\cdot}} 
		\arrow[swap]{d}{\langle \,\cdot\,, \Pi \rangle\; } & 
		\cD^{\gamma+\beta}_{\hat M}
		\arrow{d}{\;\langle \,\cdot\,, \hat\Pi \rangle} \\%
		\mathcal{G}^{\gamma} 
		\arrow{r}{\cK^{\gamma,\beta}}
		& \mathcal{G}^{\gamma + \beta}
	\end{tikzcd}	
\end{equation*}
where we set $\mathcal{G}^{\gamma}  := \mathcal{G}^{\bar{\alpha}; \bar{\alpha},  \gamma} $
and $\mathcal{G}^{\gamma + \beta} :=
\mathcal{G}^{(\bar{\alpha} + \beta) \wedge 0; (\bar\alpha + \beta) \wedge 0, \gamma + \beta}$
for short.
\end{remark}

\subsection{Continuity and further properties}

Note that the maps $\Pi \mapsto \hat{\Pi}$ and $f \mapsto \hat f = \hat \cK^{\gamma,\beta} f$ are 
affine. We will prove that they are also continuous: recalling
\eqref{eq:norm-model} and \eqref{eq:definition_modelled_distribution_norm},
as well as \eqref{eq:technorm}, we have
\begin{align}
	 \| \hat{\Pi} \|_{\mathcal{M}_{K}^{\hat\balpha}} & \lesssim  \| \Pi \|_{\mathcal{M}_{K^{\prime}}^{\balpha}} + [\sfK\Pi]_{K^{\prime}} \, , \label{eq:conhatpi} \\
		\vvvert \hat{f} \vvvert_{\mathcal{D}_{K}^{\gamma + \beta}} 
		& \lesssim \big( \| \Pi \|_{\mathcal{M}_{K^{\prime}}^{\balpha}} 
		+ [\sfK\Pi]_{K^{\prime}} \big)
		\vvvert f \vvvert_{\mathcal{D}_{K^{\prime}}^{\gamma}} \, ,  \label{eq:conhatf}
\end{align}
for some  compact $K' \supseteq K$,
e.g.\ we can take $K' := K \oplus B(0,2)$ as the 2-enlargement of $K$.
A similar continuity bound holds
for the map $\Gamma \mapsto \hat\Gamma$, see \eqref{eq:conhatgamma} below, 
provided $\Gamma$ satisfies
the analytical bound \eqref{it:tre}
in Remark~\ref{remark:comparison_definitions_of_model}.

We now discuss \emph{enhanced} continuity estimates. Observe that
\emph{the space $\cM^{\balpha}$ of models is not a vector space},
despite the semi-norm like notation $\| \cdot \|_{\mathcal{M}^{\balpha}_{K, \bar\lambda}}$,
see \eqref{eq:norm-model}, because the relation \eqref{eq:relation_Pi_Gamma} between 
$\Pi$ and $\Gamma$ is \emph{non-linear}.
Nevertheless, given two models $M_1 = (\Pi_1, \Gamma_1)$ and $M_2 = (\Pi_2, \Gamma_2)$
(with the same homogeneities $\balpha = (\alpha_i)_{i\in I}$
and the same value of $r = r_{\Pi_1} = r_{\Pi_2}$), we can consider the distance
\begin{equation*}
	\left\| \Pi_1 - \Pi_2 \right\|_{\mathcal{M}_{K}^{\balpha}} 
\end{equation*}
which is well defined by \eqref{eq:norm-model} (even though $\Pi_1 - \Pi_2$ needs not be a model).

We next compare two modelled distributions $f_1 \in \cD^\gamma_{M_1}$ 
and $f_2 \in \cD^\gamma_{M_2}$
of the same order $\gamma$, but
relative to \emph{different models} $M_1 = (\Pi_1, \Gamma_1)$ and $M_2 = (\Pi_2, \Gamma_2)$
(with the same homogeneities $\balpha = (\alpha_i)_{i\in I}$ and $r = r_{\Pi_1} = r_{\Pi_2}$),
as in \cite[Remark~3.6]{Hai14}. To this purpose, we define 
for compacts $K \subset \mathbb{R}^d$ the distance
\begin{equation*}
\begin{split}
	& \left\vvvert f_1 ; f_2 \right\vvvert_{\mathcal{D}_{K}^{\gamma}}
	= \left\vvvert f_1 ; f_2 \right\vvvert_{\mathcal{D}_{M_1, M_2; K}^{\gamma}}
	= 
	\left\vvvert f_1 ; f_2 \right\vvvert_{\mathcal{D}_{\Gamma_1, \Gamma_2, \balpha; K}^{\gamma}}  \\
	&\qquad
	\coloneqq \sup\limits_{\substack{x \in K, \, i\in I}} 
	\left| f_1^i ( x ) - f_2^i ( x ) \right| \\
	& \qquad\qquad + \sup\limits_{\substack{x, y \in K , \, i \in I}} 
	\frac{\Big| \sum\limits_{j \in I} \big\{ (\Gamma_1)_{x y}^{i j}\, f_1^j (y) 
	- (\Gamma_2)_{x y}^{i j} \, f_2^j (y) \big\} 
	- \left( f_1^i ( x ) - f_2^i ( x ) \right) \Big|}{|x-y|^{\gamma-{\alpha}_i}} .
\end{split}
\end{equation*}
We can improve the bound \eqref{eq:conhatf} via a \emph{local Lipschitz estimate}, 
which shows that
the distance between $\hat f_1$ and $\hat f_2$
is controlled by the distances between $f_1$ and $f_2$ and between the models
$\Pi_1$ and $\Pi_2$,
if $\| \Pi_i \|_{\mathcal{M}_{K^{\prime}}^{\balpha}}$
and $\vvvert f_i \vvvert_{\mathcal{D}_{K^{\prime}}^{\gamma}}$ are uniformly bounded.

\begin{proposition}[Enhanced continuity]\label{prop:enhanced-continuity}
Given any two compatible models $\Pi_1, \Pi_2$ such that $[\sfK\Pi_1]_{K^{\prime}}
=[\sfK\Pi_2]_{K^{\prime}} = 0$,  see
\Cref{ass:compatibility}, and given any
corresponding modelled distributions $f_1, f_2$,
the following bound holds:

\begin{equation*}
		\vvvert \hat{f_1} ; \hat{f_2} \vvvert_{\mathcal{D}_{K}^{\gamma + \beta}} 
		\lesssim \| \Pi_1 \|_{\mathcal{M}_{K^{\prime}}^{\balpha}} 
		\vvvert f_1 ; f_2 \vvvert_{\mathcal{D}_{K^{\prime}}^{\gamma + \beta}} 
		+ \left\| \Pi_1 - \Pi_2 \right\|_{\mathcal{M}_{K^{\prime}}^{\balpha}} 
		\vvvert f_2 \vvvert_{\mathcal{D}_{K^{\prime}}^{\gamma}} ,
	\end{equation*}
for some enlarged compact $K^{\prime} \supset K$ 
(e.g.\ we can take $K^\prime = K \oplus B(0,2)$).

\end{proposition}

\noindent
We omit the proof of this result, since it is very similar to that of \eqref{eq:conhatf}.

We finally come back to the additional properties
\eqref{it:uno}, \eqref{it:due}, \eqref{it:tre} 
of the coefficients $\Gamma$ that one may require in a model
$M = (\Pi,\Gamma)$,
see Remark~\ref{remark:comparison_definitions_of_model}.
We show that these properties
are preserved when one considers the new model $\hat M = (\hat \Pi, \hat \Gamma)$.

\begin{proposition}[Properties of reexpansion]
\label{prop:stability_of_properties_of_Gamma}
Fix a model $M = ( \Pi, \Gamma)$, a real number $\gamma \in \R$ and a $\beta$-regularizing
kernel $\sfK$ which satisfy the assumptions of Theorem~\ref{thm:multi_level_schauder_estimate}
(that is, condition \eqref{eq:multi-level-cond} holds and $\Pi$ and $\sfK$
are compatible).

Consider the new model $\hat M = (\hat\Pi, \hat\Gamma)$, 
see \eqref{eq:hatPi} and \eqref{eq:hatGamma0}-\eqref{eq:hatGamma}.
If any of the properties \eqref{it:uno}, \eqref{it:due}, \eqref{it:tre}
in Remark~\ref{remark:comparison_definitions_of_model} is satisfied
by $\Gamma$, then the same property is satisfied by $\hat \Gamma$
(with respect to the homogeneities $\bhatalpha = (\hat\alpha_a)_{a \in \hat{I}}$
in \eqref{eq:hatalpha}).

Furthermore, if property \eqref{it:tre} holds, then recalling the norm \eqref{eq:norm-gamma}
one has the continuity estimate
\begin{equation}\label{eq:conhatgamma}
	\| \hat{\Gamma} \|_{\mathcal{M}_{K}^{\hat{\balpha}}} 
	\lesssim \big( \| \Pi \|_{\mathcal{M}_{K^{\prime}}^{\balpha}}  
	+ [\sfK\Pi]_{K^{\prime}} \big) \| \Gamma \|_{\mathcal{M}_{K}^{\balpha}} \, ,  \\
\end{equation}
where $K' := K \oplus B(0,2)$ is the 2-enlargement of $K$.
\end{proposition}

\section{Proof of our Main Result I}\label{section:proof_main_result_I}

In this section we establish our first main result,
the \emph{Schauder estimates for coherent germs} in
Theorem~\ref{thm:Schauder_for_germs}.
Along the way, we also prove Proposition~\ref{thm:canonicity}
(canonicity of the germ $\cK^{\gamma,\beta} F$),
Proposition~\ref{prop:well_posedness_convolution} (singular integration) and
Theorem~\ref{thm:classical_schauder} (classical Schauder estimates).

Rather than establishing Theorem~\ref{thm:Schauder_for_germs} by direct calculation, 
we prefer to divide our proof into two steps.
	\begin{enumerate}
		\item First we establish that the operation of 
		integration $F \mapsto \sfK F$, that is,
		$(F_x)_{x\in\R^d} \mapsto (\sfK F_x)_{x\in\R^d}$,
		maps the space of coherent and homogeneous germs
		$\mathcal{G}^{\bar{\alpha}; \alpha, \gamma}$ into  a new space
		of \emph{weakly} coherent and homogeneous  germs,
		denoted by $\mathcal{G}_{\mathrm{weak}}^{\bar{\alpha} + \beta; \alpha + \beta, \gamma + \beta}$,
		for which the coherence and homogeneity conditions 
		\eqref{eq:coherence} and \eqref{eq:boundedness} hold
		for test-functions which annihilate \emph{suitable polynomials}
(this is reminiscent of
	H\"older-Zygmund spaces $\mathcal{Z}$, see Definition~\ref{def:Hoelder_Zygmund_spaces}).
		This 
		is a direct generalisation of the classical Schauder estimates
		Theorem~\ref{thm:classical_schauder}, see Remark~\ref{remark:classical_schauder} below.
	As a consequence,
	by Remark~\ref{rem:reconstruction-bounds}, we have that
\begin{equation} \label{eq:stepuno}
	F-\cR^\gamma F \in \mathcal{G}^{\gamma; \alpha, \gamma}
	\qquad \Longrightarrow \qquad
	\sfK\{F-\cR^\gamma F\} \in
	\mathcal{G}_{\mathrm{weak}}^{\gamma + \beta; \alpha + \beta, \gamma + \beta} \,.
\end{equation}

		\item Then we prove that a \emph{weakly} $\gamma'$-coherent and 
		$\gamma'$-homogeneous germ $H$
		(with \emph{the same exponent} of coherence and homogeneity)
		can be turned into a usual coherent and homogeneous germ by
		subtracting a Taylor polynomial $\cT^{\gamma'}(H)$:
\begin{equation*}
		H \in \mathcal{G}_{\mathrm{weak}}^{\gamma'; \alpha', \gamma'}
		\qquad \Longrightarrow \qquad
		H - \cT^{\gamma'}(H) \in \mathcal{G}^{\gamma'; \alpha' \wedge 0, \gamma'} 
\end{equation*}
		(note that $\alpha'$ becomes $\alpha' \wedge 0$).
		For the germ $H = \sfK\{F-\cR^\gamma F\}$
		in \eqref{eq:stepuno}, since the difference $H - \cT^{\gamma+\beta}(H)$
		equals $\mathcal{K}^{\gamma,\beta} F - \sfK\, (\mathcal{R}^\gamma F)$,
		see \eqref{eq:definition_of_KF}, we obtain
\begin{equation*}
	F-\cR^\gamma F \in \mathcal{G}^{\gamma; \alpha, \gamma}
	\qquad \Longrightarrow \qquad
	\mathcal{K}^{\gamma,\beta} F - \sfK\, (\mathcal{R}^\gamma F)
	\in \mathcal{G}^{\gamma + \beta; (\alpha + \beta) \wedge 0, \gamma + \beta} \,.
\end{equation*}
	This implies that 
	$\mathcal{K}^{\gamma,\beta} F $ is $((\alpha + \beta) \wedge 0, \gamma + \beta)$-coherent,
	and also that $\sfK\, (\mathcal{R}^\gamma F)$ is a $(\gamma+\beta)$-reconstruction
	of $\mathcal{K}^{\gamma,\beta} F $.

	If, furthermore, we assume that $F$ has
	homogeneity $\bar\alpha$, then $\cR^\gamma F \in \cZ^{\bar\alpha}$
	and consequently $\sfK\, (\mathcal{R}^\gamma F) \in \cZ^{\bar\alpha+\beta}$
	by the classical Schauder estimates. Then $\sfK\, (\mathcal{R}^\gamma F)$
	satisfies the homogeneity bound \eqref{eq:homogeneity} with exponent 
	$(\bar\alpha+\beta)\wedge 0$, which implies that $\mathcal{K}^{\gamma,\beta} F $
	has homogeneity $(\bar\alpha+\beta)\wedge 0$ (since $\bar\alpha+\beta \le \gamma+\beta$).
	\end{enumerate}

	\subsection{Weakly coherent and homogeneous germs}

We introduce a class of
\emph{weakly coherent and homogeneous} germs,
generalising Definition~\ref{def:homogeneity_coherence}.
We recall that $\mathscr{B}_\delta^r$ denotes, for $r\in\N_0$ and $\delta\in\R$,
the space of test functions
$\varphi \in \mathscr{B}^r$ which annihilate polynomials of degree
$\le \delta$, see \eqref{eq:Brgamma}.

\begin{definition}[Weak homogeneity and weak coherence]\label{def:weak_homogeneity_coherence}
Let $F = (F_x)_{x\in\R^d}$ be a germ.
Let $\bar\alpha, \alpha, \gamma \in \R$ with $\bar\alpha, \alpha \le \gamma$ and $r \in \N_0$.
\begin{itemize}
\item $F$ is called \emph{weakly $\bar\alpha$-homogeneous of order~$r$},
denoted $F \in \mathcal{G}_{\mathrm{weak\,hom};\,r}^{\bar\alpha}$,
if for any compact $K \subset \mathbb{R}^d$ and $\bar\lambda \in [1,\infty)$ the following
bounds hold:
\begin{equation}\label{eq:weak-homogeneity}
\begin{gathered}
	| F_x ( \varphi_x^{\lambda})| \lesssim \lambda^{\bar{\alpha}} \quad \ 
	\text{and} \ \quad 	| F_x ( \psi_x)| \lesssim 1 \,, \\
	\text{uniformly over $x \in K$, $\lambda \in (0, \bar\lambda]$, 
	$\varphi \in \mathscr{B}_{\bar\alpha}^{r}$
	and $\psi \in \mathscr{B}^r$} .
\end{gathered}
\end{equation}
The space of
weakly $\bar\alpha$-homogeneous germs (of any order) is $\mathcal{G}_{\mathrm{weak\,hom}}^{\bar\alpha}
= \bigcup\limits_{r\in\N_0} \mathcal{G}_{\mathrm{weak\,hom};\,r}^{\bar\alpha}$.

\item $F$ is called \emph{weakly $(\alpha,\gamma)$-coherent of order~$r$},
denoted $F \in \mathcal{G}_{\mathrm{weak\,coh};\,r}^{\alpha,\gamma}$,
if for any compact $K \subset \mathbb{R}^d$ and $\bar\lambda \in [1,\infty)$ 
the following bounds hold:
\begin{equation}\label{eq:weak-coherence}
\begin{gathered}
	| ( F_y - F_x) ( \varphi_x^{\lambda})| \lesssim \lambda^{\alpha} 
	(| y - x| + \lambda)^{\gamma - \alpha}
	\quad \ \text{and} \quad \
	| ( F_y - F_x) ( \psi_x)| \lesssim 1 \\
	\text{uniformly over $x,y \in K$, $\lambda \in (0, \bar\lambda]$, 
	$\varphi \in \mathscr{B}_{\gamma}^{r}$ and $\psi \in \mathscr{B}^r$} .
\end{gathered}
\end{equation}
The space of weakly $(\alpha,\gamma)$-coherent germs (of any order) is
$\mathcal{G}_{\mathrm{weak\,coh}}^{\alpha,\gamma} 
= \bigcup\limits_{r\in\N_0} \mathcal{G}_{\mathrm{weak\,coh};\,r}^{\alpha,\gamma}$.

\item $F$ is called \emph{weakly $(\alpha,\gamma)$-coherent with homogeneity $\bar{\alpha}$} if both \eqref{eq:weak-homogeneity} and \eqref{eq:weak-coherence} hold, for some order $r\in\N$.
The space of such germs is
\begin{equation*}
	\mathcal{G}_{\mathrm{weak}}^{\bar\alpha; \alpha,\gamma}
	= \mathcal{G}_{\mathrm{weak\,hom}}^{\bar\alpha} \cap \mathcal{G}_{\mathrm{weak\,coh}}^{\alpha,\gamma} \,.
\end{equation*}
\end{itemize}
\end{definition}

\begin{remark}[Usual vs.\ weak homogeneity and coherence]
The first conditions in \eqref{eq:weak-homogeneity} and \eqref{eq:weak-coherence} involve different classes of test functions, namely
$\varphi \in \mathscr{B}_{\bar{\alpha}}^{r}$ and $\varphi \in \mathscr{B}_{\gamma}^{r}$, while the second conditions in 
\eqref{eq:weak-homogeneity} and \eqref{eq:weak-coherence}
involve $\psi \in \mathscr{B}_{}^{r}$. 
However, when $\bar{\alpha} < 0$ and $\gamma < 0$
we have $\mathscr{B}_{\bar{\alpha}}^{r} = \mathscr{B}_{\gamma}^{r} = \mathscr{B}_{}^{r}$,
hence \eqref{eq:weak-homogeneity} and \eqref{eq:weak-coherence} reduce to the usual homogeneity and coherence conditions  
\eqref{eq:homogeneity} and \eqref{eq:coherence}.

In particular, \emph{coherent and homogeneous germs are
weakly coherent and homogeneous}: $\cG^{\bar\alpha; \alpha,\gamma} \subseteq 
\mathcal{G}_{\mathrm{weak}}^{\bar\alpha; \alpha,\gamma}$, and the inclusion is an equality
when $\bar{\alpha} < 0$ and $\gamma < 0$.
\end{remark}

\begin{remark}[General scales]\label{remark:lower_scales_II}
As in the Remark~\ref{remark:lower_scales},
for germs $F = (F_x)_{x\in\R^d}$ that are both weakly homogeneous and
weakly coherent we can get rid of $\bar\lambda$, i.e.\ if both relations
\eqref{eq:weak-homogeneity} and \eqref{eq:weak-coherence} holds for $\bar\lambda = 1$,
then they hold for any $\bar\lambda \in [1,\infty)$. To this purpose, 
for $\lambda \in [1,\bar\lambda]$
we decompose 
$\varphi_x^{\lambda} = \sum_{k = 1}^n ({\psi_k})_{x_k}^{1}$ for suitable $x_k \in K$, 
$\psi_k \in \mathscr{B}^r$ and $n$ (uniformly bounded, depending on
$\bar{\lambda}$ and $K$); by the second bounds in \eqref{eq:weak-homogeneity}
and \eqref{eq:weak-coherence} we get $|F_y((\psi_k)_x^{1})| \le
|(F_y-F_x)((\psi_k)_x^{1})| + |F_x((\psi_k)_x^{1})| \lesssim 1$, hence
$|F_y(\varphi_x^{\lambda})| \lesssim 1$ for $\lambda \in [1, \bar\lambda]$, from
which the first bounds in \eqref{eq:weak-homogeneity}
and \eqref{eq:weak-coherence} follow.
\end{remark}

We introduce semi-norms for weakly homogeneous and coherent germs,
corresponding to \eqref{eq:weak-homogeneity}
and \eqref{eq:weak-coherence}:
\begin{align}\label{eq:semi-norm-weak-homogeneity}
	\| F \|_{\mathcal{G}_{\mathrm{weak\,hom}; K,\bar\lambda,r}^{\bar{\alpha}}}
	& := \sup_{\substack{x\in K, \, \lambda \in (0,\bar\lambda] \\
	\varphi \in \mathscr{B}_{\bar{\alpha}}^{r}}} 
	\frac{| F_x ( \varphi_x^{\lambda})|}{\lambda^{\bar{\alpha}}} + \sup_{\substack{x\in K, \, \psi \in \mathscr{B}_{}^{r}}} | F_x ( \psi_x)| \,, \\
	\label{eq:semi-norm-weak-coherence}
	\| F \|_{\mathcal{G}_{\mathrm{weak\,coh}; K,\bar\lambda,r}^{\alpha, \gamma}}
	& := \sup_{\substack{x,y\in K, \, \lambda \in (0,\bar\lambda] \\
	\varphi \in \mathscr{B}_{\gamma}^{r}}} 
	\frac{| ( F_y - F_x) ( \varphi_x^{\lambda})|}%
	{\lambda^{\alpha} (| y - x| + \lambda)^{\gamma - \alpha}} + \sup_{\substack{x, y \in K, \, \psi \in \mathscr{B}_{}^{r}}} | ( F_y - F_x) ( \psi_x)| \,.
\end{align}
We next define the joint semi-norm where we fix $r = r_{\bar\alpha,\alpha}$
as in \eqref{eq:canonical-r}, i.e.\
the smallest non-negative integer $r > \max\{-\bar{\alpha},-\alpha\}$
(see Proposition~\ref{prop:independence_in_r} in Appendix~\ref{section:independence_in_r}):
\begin{equation}
	\label{eq:semi-norm-weak-coherence+homogeneity}
	\| F \|_{\mathcal{G}_{\mathrm{weak};K,\bar\lambda}^{\bar\alpha;\alpha, \gamma}} 
	:= \| F \|_{\mathcal{G}_{\mathrm{weak\,hom}; K,\bar\lambda,r_{\bar\alpha,\alpha}}^{\bar{\alpha}}} +
	\| F \|_{\mathcal{G}_{\mathrm{weak\,coh}; K,\bar\lambda,r_{\bar\alpha,\alpha}}^{\alpha, \gamma}} \,.	
\end{equation}
By Remark~\ref{remark:lower_scales_II}, we may set $\bar{\lambda} = 1$ 
and omit it from the notation.

\subsection{Conditional proof of Theorem~\ref{thm:Schauder_for_germs}}

We state two basic results on (weakly) coherent and homogeneous germs,
which will yield Theorem~\ref{thm:Schauder_for_germs} as a corollary.
We also deduce Proposition~\ref{thm:canonicity}.

The first result, proved in Section~\ref{subsubsection:proof_convolution_germs} below,
describes how a regularising kernel $\sfK$ acts on germs $F = (F_x)_{x\in\R^d}$
by integration, i.e.\ we consider $\sfK F := (\sfK F_x)_{x\in\R^d}$.
We recall that $\sfK F$ is well-defined for homogeneous germs,
see Remark~\ref{rem:bounded-order}.
For coherent germs, only $\sfK(F_y - F_x)$ is ensured to be well-defined,
see again Remark~\ref{rem:bounded-order}, but also in this case we will consider the germ
$\sfK F := (\sfK F_x)_{x\in\R^d}$ proving that it is weakly coherent: this is an abuse of notation,
justified by the fact that for the weak coherence relation \eqref{eq:weak-coherence}
only the differences $\sfK F_y - \sfK F_x = \sfK(F_y - F_x)$ matter.

\begin{theorem}[Integration of germs]\label{thm:convolution_germs_in_G_checkG}
Let $\sfK$ be a $\beta$-regularising kernel of order $(m, r)$ with range $\rho$, see Definition~\ref{def:regularising_kernel}.
For every compact $K\subseteq \R^d$ and $\bar\lambda \in [1,\infty)$
there is a constant $\mathrm{cst}'_{K,\bar\lambda} < \infty$ such that the following holds,
for any $\bar\alpha, \alpha, \gamma \in \R$ with $\bar{\alpha}, \alpha \le \gamma$:
\begin{itemize}
\item if $m > \bar\alpha+\beta$, integration by
$\sfK$ maps continuously $\cG^{\bar\alpha}_{\ho,r}$
to $\mathcal{G}_{\mathrm{weak\,hom};\,r}^{\bar\alpha+\beta}$:
\begin{equation}\label{eq:Kconv-hom}
	\forall r \in \N: \qquad
	\| \sfK F \|_{\mathcal{G}_{\mathrm{weak\,hom}; K,\bar\lambda,r}^{\bar{\alpha}+\beta}}
	\le \mathrm{cst}'_{K,\bar\lambda} \,
	\| F \|_{\mathcal{G}_{\mathrm{hom}; K,2(\bar\lambda + \rho),r}^{\bar{\alpha}}} \,;
\end{equation}

\item if $m > \gamma+\beta$, integration by
$\sfK$ maps continuously $\mathcal{G}_{\mathrm{coh};\,r}^{\alpha,\gamma}$ 
to $\mathcal{G}_{\mathrm{weak\,coh};\,r}^{\alpha+\beta,\gamma+\beta}$:
\begin{equation}\label{eq:Kconv-coh}
	\forall r \in \N: \qquad
	\| \sfK F \|_{\mathcal{G}_{\mathrm{weak\,coh}; K,\bar\lambda,r}^{\alpha+\beta,\gamma+\beta}}
	\le \mathrm{cst}'_{K,\bar\lambda} \, 
	\| F \|_{\mathcal{G}_{\mathrm{coh}; K,2(\bar\lambda + \rho),r}^{\alpha,\gamma}} \,.
\end{equation}
\end{itemize}

As a consequence (see Remark~\ref{rem:uniformity}), if we assume that
\begin{equation*}
	r > \max\{-\bar{\alpha},-\alpha\} \,, \qquad m > \gamma+\beta \,,
\end{equation*}
then integration by $\sfK$ is a continuous linear map from 
		$\mathcal{G}^{\bar{\alpha}; \alpha, \gamma}$ to 
		$\mathcal{G}_{\mathrm{weak}}^{\bar\alpha + \beta; \alpha + \beta, \gamma + \beta}$:
\begin{equation}\label{eq:Kconvbound}
	\| \sfK F \|_{\mathcal{G}_{\mathrm{weak};K,\bar\lambda}^{\bar\alpha+\beta;
	\alpha+\beta, \gamma+\beta}} 
	\le \mathrm{cst}'_{K,\bar\lambda} \, 
	\| F \|_{\mathcal{G}_{K,2(\bar\lambda + \rho)}^{\bar\alpha;\alpha, \gamma}} \,.
\end{equation}

If furthermore $\sfK$ preserves polynomials at level $\gamma$, see 
Assumption~\ref{assumption:preserving_polynomial_annihilation},  
then integration by $\sfK$ is also a continuous linear map from 
$\mathcal{G}_{\mathrm{weak}}^{\bar{\alpha}; \alpha, \gamma}$ to 
$\mathcal{G}_{\mathrm{weak}}^{\bar\alpha + \beta; \alpha + \beta, \gamma + \beta}$.
\end{theorem}

\begin{remark}[Classical Schauder estimates] \label{remark:classical_schauder}
Given any $\bar\alpha \in \R$ and
any distribution $f \in \cZ^{\bar\alpha}$, see Definition~\ref{def:Hoelder_Zygmund_spaces},
we can consider the constant germ $(F_x = f)_{x\in\R^d}$ which is clearly coherent for any
exponents $\alpha,\gamma$ and \emph{weakly} homogeneous with exponent~$\bar\alpha$,
that is $F \in \mathcal{G}_{\mathrm{weak}}^{\bar\alpha; \alpha,\gamma}$.
By Theorem~\ref{thm:convolution_germs_in_G_checkG}, we see that $(\sfK F_x) \in
\mathcal{G}_{\mathrm{weak}}^{\bar\alpha+\beta; \alpha+\beta,\gamma+\beta}$,
which means that $\sfK f \in \cZ^{\bar\alpha+\beta}$
(compare \eqref{eq:weak-homogeneity} 
with \eqref{eq:Cgamma}).
We thus obtain the classical Schauder estimates, Theorem~\ref{thm:classical_schauder},
as a corollary of Theorem~\ref{thm:convolution_germs_in_G_checkG}.
\end{remark}

Our second basic result links weakly coherent and homogeneous germs with ordinary ones,
\emph{in the special case when homogeneity and coherence exponents coincide: $\bar\alpha = \gamma$}.
This will be proved in Section~\ref{sec:embedding_G_checkG} below,
together with Lemma~\ref{lemma:pointwise_derivative} which ensures the existence
of pointwise derivatives for suitable distributions.

\begin{theorem}[Positive renormalisation]\label{thm:embedding_G_checkG}
Let $\alpha, \gamma \in \mathbb{R}$ with $\alpha \leq \gamma$ and 
	\begin{equation*}
		\alpha \neq 0 , \quad \gamma \notin \mathbb{N}_0 .
	\end{equation*}
If a germ $F = (F_x)_{x\in\R^d} \in \mathcal{G}_{\mathrm{weak}}^{\gamma; \alpha, \gamma}$
is weakly $(\alpha,\gamma)$-coherent and weakly $\gamma$-homogeneous, subtracting the family $\cT^\gamma(F) = (\cT^{\gamma}_x(F_x))_{x\in\R^d}$
of its Taylor polynomials, see \eqref{eq:Taylor-f}, we obtain the germ
	\begin{equation*}
		G = F - \cT^\gamma F \,, \qquad \text{that is} \qquad
\begin{array}{rl}
		\rule{0pt}{1.1em}G_x \coloneqq
		& F_x - \cT^{\gamma}_x(F_x)  \\
		\rule{0pt}{1.8em}= & \displaystyle F_x - \sum_{0 \le | k | < \gamma} 
		D^k F_x ( x )\, \frac{(\,\cdot - x)^k}{k!}  \,,
\end{array}
	\end{equation*}
which is well defined, $(\alpha\wedge 0, \gamma)$-coherent and $\gamma$-homogeneous, i.e.\
$G \in \mathcal{G}^{\gamma; \alpha \wedge 0, \gamma}$.
The map $F \mapsto G$ is linear and continuous:
for any compact $K \subseteq \R^d$ and $\bar\lambda > 0$
\begin{equation}\label{eq:contG}
	\|G \|_{\mathcal{G}_{K,\bar\lambda}^{\gamma;\alpha, \gamma}} 
	\le \mathrm{cst}_{K,\bar\lambda} 
	\, \|F \|_{\mathcal{G}_{\mathrm{weak};K,\bar\lambda'}^{\gamma;\alpha, \gamma}} 
\end{equation}
for $\bar{\lambda}^{\prime} \coloneqq 4 (\bar{\lambda} + \mathrm{diam} (K))$.
\end{theorem}

\begin{remark}
The terminology ``positive renormalisation'' is inspired by \cite{MR3935036}, where this notion is related
to an operator called $\Delta^+$
which yields an algebraic description of the subtraction of Taylor polynomials,
see \cite[Lemma 6.10 and Remark 6.11]{MR3935036} and \cite[section 15.3]{MR4174393}.
\end{remark}

We can now deduce Theorem~\ref{thm:Schauder_for_germs}
from Theorems~\ref{thm:convolution_germs_in_G_checkG}
and~\ref{thm:embedding_G_checkG}.

\begin{proof}[Proof of Theorem~\ref{thm:Schauder_for_germs}]
Let $F$ be an $(\alpha, \gamma)$-coherent germ for some 
$\alpha \le \gamma$. Let $\cR^\gamma F$ be a $\gamma$-reconstruction of $F$, see 
\eqref{eq:reconstruction_bound}, so that (see Remark~\ref{rem:reconstruction-bounds})
	\begin{equation*}
		F - \mathcal{R}^\gamma F \in \mathcal{G}^{\gamma; \alpha, \gamma} .
	\end{equation*}
We stress that $F - \mathcal{R}^\gamma F$ is both coherent and homogeneous, 
even when $F$ is not homogeneous.
By Theorem~\ref{thm:convolution_germs_in_G_checkG}, using the assumptions 
\eqref{eq:assumption_schauder_on_germs}, it follows that
	\begin{equation} \label{eq:wealre}
		\sfK \{F-\cR^\gamma F\} :=
		(\sfK \{F_x - \mathcal{R}^\gamma F\} )_{x\in\R^d} \in 
		\mathcal{G}_{\mathrm{weak}}^{\gamma + \beta; \alpha + \beta, \gamma + \beta} .
	\end{equation}

By assumption $\alpha + \beta \neq 0$, $\gamma + \beta \notin \mathbb{N}_0$,
see \eqref{eq:assumption_remove_polynomial},
therefore we can apply 
Theorem~\ref{thm:embedding_G_checkG} to
the germ $\sfK \{F-\cR^\gamma F\}$ to obtain
\begin{equation*}
	\sfK \{F-\cR^\gamma F\} - \cT^{\gamma+\beta}\big( \sfK \{F-\cR^\gamma F\} \big)
	\in \mathcal{G}^{\gamma + \beta; (\alpha + \beta) \wedge 0, \gamma + \beta} \,.
\end{equation*}
Recalling \eqref{eq:definition_of_KF}-\eqref{eq:definition_of_KF_alt}, this can be rewritten as 
	\begin{equation} \label{eq:KconvR-G}
		\mathcal{K}^{\gamma,\beta} F  - \sfK\, (\mathcal{R}^\gamma F) \in 
		\mathcal{G}^{\gamma + \beta; (\alpha + \beta) \wedge 0, \gamma + \beta} ,
	\end{equation}
which proves \eqref{eq:main-gen}.
If we assume that $\sfK\, (\mathcal{R}^\gamma F) \in \cD'$ is well-defined
(note that it is a fixed distribution, which does
not depend on~$x$), it follows that $\mathcal{K}^{\gamma,\beta} F$ is well-defined
and is $( (\alpha + \beta) \wedge 0, \gamma + \beta )$-coherent.
The property of homogeneity in \eqref{eq:KconvR-G}
means precisely that $\sfK\, (\mathcal{R}^\gamma F)$
is a $(\gamma+\beta)$-reconstruction of $\mathcal{K}^{\gamma,\beta} F$,
see \eqref{eq:reconstruction_bound}, that is
	\begin{equation*}
		\mathcal{R}^{\gamma+\beta} ( \mathcal{K}^{\gamma,\beta} F ) 
		= \sfK ( \mathcal{R}^\gamma F ) \,.
	\end{equation*}
The continuity estimate \eqref{eq:cont-schauder1} follows by
\eqref{eq:contG} and \eqref{eq:Kconvbound}.

We finally assume that $F$ is also $\bar{\alpha}$-homogeneous
and we establish the homogeneity of $\mathcal{K}^{\gamma,\beta} F$.
We know that $\mathcal{R}^\gamma F \in \cZ^{\bar\alpha}$, see \eqref{eq:target},
hence $\sfK \mathcal{R}^\gamma F \in \cZ^{\bar\alpha + \beta}$ 
by the classical Schauder estimates, see Theorem~\ref{thm:classical_schauder}.
If we view $\sfK \mathcal{R}^\gamma F$
as a constant germ, it is $(\alpha',\gamma')$-coherent for 
any $\alpha', \gamma'$ and has homogeneity $(\bar\alpha+\beta)\wedge 0$, hence
	\begin{equation} \label{eq:Kco}
		\sfK \mathcal{R}^\gamma F \in \mathcal{G}^{(\bar{\alpha} + \beta) \wedge 0;
		\alpha', \gamma'} .
	\end{equation}
Since $\bar\alpha \le \gamma$, summing \eqref{eq:KconvR-G} and \eqref{eq:Kco} we obtain
	\begin{equation*}
		\mathcal{K}^{\gamma,\beta} F \in \mathcal{G}^{(\bar{\alpha} + \beta) \wedge 0;
		(\alpha + \beta) \wedge 0, \gamma + \beta} \,.
	\end{equation*}

Finally, the continuity estimate \eqref{eq:cont-schauder2} follows from \eqref{eq:cont-schauder1}-\eqref{eq:RT-bounds}.
Furthermore, for \eqref{eq:cont-schauder3} we use the fact that 
	\begin{equation*}
		\| \sfK (\mathcal{R}^{\gamma} F) \|_{\mathcal{G}_{\mathrm{hom; K, \bar{\lambda}, r}}^{(\bar\alpha + \beta) \wedge 0}} \lesssim \| \sfK (\mathcal{R}^{\gamma} F) \|_{\mathcal{Z}_{\mathrm{K^{\prime}, \bar{\lambda}, r}}^{\bar\alpha + \beta}} ,
	\end{equation*}
\noindent for $K^{\prime} \coloneqq K \oplus B (0, \bar{\lambda})$, which follows from the definition of homogeneity when $\bar\alpha + \beta < 0$ and from \eqref{eq:estimate_on_F_x^k(x)} when $\bar{\alpha} + \beta > 0$. Now the right-hand side can be bounded by the classical Schauder estimates and the reconstruction theorem.
\end{proof}

We then prove Proposition~\ref{thm:canonicity}.

\begin{proof}[Proof of Proposition~\ref{thm:canonicity}]
Set $\cK^{\gamma,\beta} F := \sfK F_x - P_x$ as in \eqref{eq:postulate},
where $P = (P_x)_{x\in\R^d}$ for the moment is an arbitrary germ.
Assume that \eqref{eq:reco-schauder} holds, i.e.\ that
$\sfK (\mathcal{R}^\gamma F)$ is
a $(\gamma+\beta)$-reconstruction  of $\cK^{\gamma,\beta} F$.
Then, if we define $G := \sfK \{F-\cR^\gamma F\}$, for any $\psi \in \cD(B(0,1))$
we must have, as $\lambda \downarrow 0$,
\begin{equation*}
	(G_x - P_x )(\psi_x^\lambda)
	= (\cK^{\gamma,\beta} F   - \cR^\gamma F )(\psi_x^\lambda) 
	= O(\lambda^{\gamma+\beta}) \,.
\end{equation*}
We already observed in \eqref{eq:wealre} that $G \in 
\mathcal{G}_{\mathrm{weak}}^{\gamma + \beta; \alpha + \beta, \gamma + \beta}$,
by Theorem~\ref{thm:convolution_germs_in_G_checkG}.
In particular, for any $x\in\R^d$, the distribution $f = G_x$ satisfies
the assumption \eqref{eq:local-hom} of Lemma~\ref{lemma:pointwise_derivative}
with $\delta = \gamma+\beta$, hence
$(G_x - \cT^{\gamma+\beta}_x(G))(\psi_x^\lambda) = O(\lambda^{\gamma+\beta})$
 by \eqref{eq:gbound}. We then obtain
\begin{equation*} 
	(P_x - \cT^{\gamma+\beta}_x(G))(\psi_x^\lambda) = O(\lambda^{\gamma+\beta}) \,.
\end{equation*}

Assume now that $P_x$ is a polynomial of degree $<\gamma+\beta$. Then
$Q_x := P_x - \cT^{\gamma+\beta}_x(G)$ is a polynomial of degree $<\gamma+\beta$
with $Q_x(\psi_x^\lambda) = O(\lambda^{\gamma+\beta})$,
for any $\psi \in \cD(B(0,1))$. But this implies that $Q_x = 0$:
indeed, for any $k\in\N_0^d$ with $0 \le |k| <\gamma+\beta$ we have
\begin{equation*}
	(\partial^k Q_x)(\psi_x^\lambda) = (-\lambda)^{-|k|} Q_x((\partial^k \psi)_x^\lambda)
	= O(\lambda^{\gamma+\beta-|k|}) \xrightarrow[\,\lambda\downarrow 0\,]{} 0 \,,
\end{equation*}
hence, if we choose $\psi$ with $\int\psi = 1$, we obtain
$\partial^k Q_x(x) = \lim_{\lambda\downarrow 0} (\partial^k Q_x)(\psi_x^\lambda) = 0$.
\end{proof}

The rest of this section is devoted to the proof of Theorems~\ref{thm:convolution_germs_in_G_checkG}
and~\ref{thm:embedding_G_checkG}. 
In the next subsection we first discuss some technical tools, which will also yield
the proof of Proposition~\ref{prop:well_posedness_convolution} 
(integration of \texorpdfstring{$\sfK$}{K} with a sufficiently nice distribution $f$ is well-defined) and
Theorem~\ref{thm:classical_schauder} (classical Schauder estimate).

\subsection{Preliminary tools} \label{section:preliminary_tools_I}

Fix a test function $\varphi \in \cD$.
A key ingredient of our proofs is a convenient representation for
the function $\sfK_n^* \varphi_x^{\lambda}$,
recall \eqref{eq:K*}, provided by Lemma~\ref{lemma:test_functions_K_n} below. 
This result is analogous to \cite[Proposition~14.11]{MR4174393},
which however only considers the translation invariant case $\sfK_n (x, y) = \sfK_n ( y - x)$.

We first describe heuristically the result, focusing for simplicity on 
$\sfK_n^* \varphi^{\lambda} = \sfK_n^* \varphi_0^{\lambda}$.
Given a $\beta$-regularising kernel $\sfK$ with range $\rho$, 
as in Definition~\ref{def:regularising_kernel},
by properties~\eqref{item:regularising_3} and~\eqref{item:regularising_4}
we can write approximately
\begin{equation*}
	\sfK_n(z,y) \simeq 2^{- \beta n} \, \psi^{\rho 2^{- n}}(y-z) 
\end{equation*}
for some test-function $\psi \in \mathcal{D}$, and thus, by \eqref{eq:K_n*},
\begin{equation} \label{eq:inte-conv}
	\sfK_n^* \varphi^{\lambda} (y) 
	\simeq 
	2^{- \beta n} \, \int_{\R^d} \varphi^\lambda(z) \, \psi^{\rho 2^{-n}}(y-z) \, \d z \,.
\end{equation}
For $ \lambda \ge \rho 2^{-n}$, we can pretend that the test function
$\varphi^\lambda$, which has the larger scale,
is approximately constant on the support of $\psi^{\rho 2^{-n}}$. This yields
the approximation
	\begin{equation*}
		\text{ for } \lambda \ge \rho 2^{-n} : \qquad
		\sfK_n^* \varphi^{\lambda} (y)
		\simeq 2^{- \beta n}
		\, \eta^{2\lambda}(y)
		\qquad (\text{with } \eta \simeq \varphi) \,.
	\end{equation*}
For $ \lambda \leq \rho 2^{- n}$, exchanging the roles of $\varphi$ and $\psi$
would yield $\sfK_n^* \varphi^{\lambda} (y) \simeq 2^{- \beta n}\, \eta^{\rho 2^{-n}}(y)$
with $\eta \simeq \psi$, but a better approximation can be obtained
if we assume that $\varphi$ annihilates polynomials 
up to some degree $c \in \mathbb{N}_0$: subtracting
the Taylor polynomial of $\psi^{\rho 2^{-n}}$ 
at order $c$ based at $y$ in the integral in \eqref{eq:inte-conv}, we obtain
	\begin{equation*}
		\text{ for } \lambda \le \rho 2^{-n} : \qquad
		\sfK_n^* \varphi^{\lambda}(y)
		\simeq 2^{- \beta n} \, (2^{n} \lambda)^{c + 1} \,
		\zeta^{(2\rho) 2^{- n}}(y) \quad \text{ (for some $\zeta \in \mathcal{D}$)}\, .
	\end{equation*}

We can now state the precise result. Its proof is given in Appendix~\ref{sec:test_functions_K_n}.

\begin{lemma}[Integrating $\sfK_n$ with test functions]\label{lemma:test_functions_K_n}
Let $\sfK$ be a $\beta$-regularising kernel of order $(m, r)$ with range $\rho$,
see Definition~\ref{def:regularising_kernel}.
Fix $\bar\lambda \in [1,\infty)$ and an integer $c \in \mathbb{N}_0 \cup \lbrace -1 \rbrace$.
For every compact $K\subseteq \R^d$ 
and any test functions $\varphi \in \mathscr{B}_{c}^{r}$ we can write,
for all $n \in \mathbb{N}_0$, $x \in K$ and $\lambda \in ( 0, \bar\lambda]$,
	\begin{equation}\label{eq:scaled_recentered_eta_zeta}
		\sfK_n^* \varphi_{x}^{\lambda}(y) =
		\begin{dcases}
			 2^{- \beta n} 
			 \, \eta_{x}^{2 \lambda}(y)
			 & \text{for } \rho 2^{- n} \leq \lambda \,, \\
			2^{- \beta n} \, (2^{n} \lambda)^{\min\{c+1,m\}}
			\, 
			\zeta_{x}^{2(\rho 2^{-n})}(y) & \text{for } \rho 2^{- n} \geq \lambda \,,
		\end{dcases}
	\end{equation}
for suitable (explicit) test-functions
	\begin{equation}\label{eq:eta_zeta_are_in_Bcr}
	\eta = \eta^{[n, \lambda, x, \varphi]} \in \mathrm{cst}_{K,\bar\lambda}\, \mathscr{B}^{r} , \qquad 
	\zeta = \zeta^{[n, \lambda, x, \varphi]} \in \mathrm{cst}_{K,\bar\lambda} \,\mathscr{B}^{r} \,.
	\end{equation}
where $\mathrm{cst}_{K,\bar\lambda} < \infty$ depends only on $K$, $\bar\lambda$
and on the kernel $\sfK$.
	
If furthermore $\sfK$ preserves polynomials at some level $c_0 \in \mathbb{N}_0$,
see Assumption~\ref{assumption:preserving_polynomial_annihilation}, 
we can improve \eqref{eq:eta_zeta_are_in_Bcr} to
	\begin{equation} \label{eq:eta_zeta_are_in_Bcr_when_K_preserves_polynomials}
	\eta = \eta^{[n, \lambda, x, \varphi]} \in \mathrm{cst}_{K,\bar\lambda}\, 
	\mathscr{B}_{\min(c, c_0)}^{r} \,, 
	\qquad \zeta = \zeta^{[n, \lambda, x, \varphi]} \in \mathrm{cst}_{K,\bar\lambda}\, 
	\mathscr{B}_{\min(c, c_0)}^{r} \,.
	\end{equation}
\end{lemma}

\begin{remark}\label{remark:new_scale_bar_lambda}
It follows from the proof of this result that the constants $\mathrm{cst}_{K,\bar\lambda}$ in \eqref{eq:eta_zeta_are_in_Bcr}
and \eqref{eq:eta_zeta_are_in_Bcr_when_K_preserves_polynomials}
depend on the constants $c_K$ appearing in
items~\eqref{item:regularising_4} and~\eqref{item:regularising_5} 
in Definition~\ref{def:regularising_kernel}, as well as on the range $\rho$ of the kernel $\sfK$.
\end{remark}

We close this subsection with
the proofs of Proposition~\ref{prop:well_posedness_convolution}
(singular integration)
and Theorem~\ref{thm:classical_schauder} (classical Schauder estimates).
	
\begin{proof}[Proof of Proposition~\ref{prop:well_posedness_convolution}]
Let $\sfK$ be a $\beta$-regularising kernel
of order $(0, r)$ with range~$\rho$. Let $f$ be a distribution of order~$r$, so that
by Remark~\ref{rem:singint}, see \eqref{eq:order2} with $z=0$,
	\begin{equation} \label{eq:finiteorder}
		\forall \bar\lambda' \in [1,\infty): \qquad
		\sup\limits_{\eta \in \mathscr{B}^r, \, \lambda \in [1,\bar\lambda']} 
		| f  ( \eta^{\bar\lambda'} ) | \eqcolon C(\bar\lambda') < + \infty \,.
	\end{equation}
Let us show that $\sfK f$ is well-defined by  \eqref{eq:definition_of_convolution}
as a distribution of order~$r$, i.e.\ by \eqref{eq:order2bis}
	\begin{equation} \label{eq:local-conv-reg}
		\forall \bar\lambda \in \N: \qquad
		\sup\limits_{\varphi \in \mathscr{B}^r} | \sfK f ( \varphi^{\bar\lambda} ) | < + \infty .
	\end{equation}

Given $\bar\lambda \in [1,\infty)$,
we apply Lemma~\ref{lemma:test_functions_K_n} with $c=-1$, in particular
by \eqref{eq:scaled_recentered_eta_zeta}
	\begin{equation*}
	\sfK_n^* \varphi^{\bar\lambda} =
	\begin{cases}
	 2^{- \beta n} \, \eta^{2\bar\lambda} & \text{if } \rho 2^{-n} \le \bar\lambda \,, \\
	 2^{- \beta n} \, \zeta^{2(\rho 2^{-n})} & \text{if } \rho 2^{-n} \ge \bar\lambda \,,
 	\end{cases}
	\end{equation*}
for some test-functions $\eta, \zeta \in \mathrm{cst} \, \mathscr{B}^{r}$.
Both scaling exponents $2\bar\lambda$ and $2(\rho 2^{-n}) \ge 2\bar\lambda$ 
take values in $[1, \bar\lambda']$, where we set $\bar\lambda':= \max\{2\bar\lambda, 2\rho\}$,
hence by relation \eqref{eq:finiteorder} 
we can bound $| f( \sfK_n^* \varphi^{\bar\lambda} ) | \le C(\bar\lambda') \, 2^{- \beta n}$
uniformly over $n\in\N$.
Thus, the sum in \eqref{eq:definition_of_convolution} converges 
and furthermore \eqref{eq:local-conv-reg} holds.
\end{proof}		

\begin{proof}[Proof of Theorem~\ref{thm:classical_schauder}]		
We give a simple proof of Theorem~\ref{thm:classical_schauder}
exploiting Lemma~\ref{lemma:test_functions_K_n},
in the same manner as in \cite[Section~14]{MR4174393}.
We will also perform similar calculations when proving Theorem~\ref{thm:convolution_germs_in_G_checkG} below.

We fix $\gamma \in \mathbb{R}$, a $\beta$-regularising kernel $\sfK$
of order $(m, r)$ with range $\rho$, where we assume that
$m > \gamma + \beta$ and
$r > - \gamma$, see \eqref{eq:condmr}.
When $\gamma \ge 0$, we also assume that $\sfK$ preserves polynomials at 
level~$\gamma$, see Assumption~\ref{assumption:preserving_polynomial_annihilation}.

Let us fix $f \in \mathcal{Z}^{\gamma}$.
First we note that $f$ is a distribution of order~$r$
by Remark~\ref{rem:singint},
hence by Proposition~\ref{prop:well_posedness_convolution} the integration $\sfK f$ 
is well-defined by  \eqref{eq:definition_of_convolution}.
It remains to show that $\sfK f \in \cZ^{\gamma+\beta}$, that is,
$\|\sfK f\|_{\cZ^{\gamma+\beta}_K} < \infty$ for any compact $K \subset \R^d$,
see \eqref{eq:Cgamma}.

Fix $K \subset \mathbb{R}^d$ compact.
We take $x \in K$ and $\lambda \in (0, 1]$.
We first estimate $\sfK f(\varphi_x^\lambda)$ for
$\varphi \in \mathscr{B}_{\gamma + \beta}^{r}$, see \eqref{eq:Cgamma}.
We define
\begin{equation*}
	N_{\lambda} \coloneqq \min \lbrace n \in \mathbb{N}
	\colon \rho 2^{- n} \leq \lambda \rbrace \,,
\end{equation*}
and we cut the series  \eqref{eq:definition_of_convolution} in two regimes.
By Lemma~\ref{lemma:test_functions_K_n} 
with $\bar\lambda = 1$, $c=\lfloor\gamma+\beta\rfloor$ and $c_0=\lfloor\gamma\rfloor$,
we can express $\sfK_n^* \varphi_{x}^{\lambda}$ through formula 
\eqref{eq:scaled_recentered_eta_zeta}
(where $\min\{c+1,m\} = c+1$, since $m > \gamma + \beta$):
for suitable test-functions $\eta = \eta^{[n, \lambda, x, \varphi]}$ and 
$\zeta = \zeta^{[n, \lambda, x, \varphi]}$ we have
\begin{equation}\label{eq:usude}
\begin{split}
	\sfK f (\varphi_x^{\lambda}) & = \sum\limits_{n = 0}^{N_{\lambda} - 1} 
	f ( \sfK_n^* \varphi_{x}^{\lambda}) \, + \sum\limits_{n = N_{\lambda}}^{+ \infty} 
	f ( \sfK_n^* \varphi_{x}^{\lambda}) \\
	& = \sum\limits_{n = 0}^{N_{\lambda} - 1} 2^{- \beta n} (2^{n} \lambda)^{c+1} 
	f (\zeta_{x}^{2\rho 2^{-n}} ) \, 
	+ \sum\limits_{n = N_{\lambda}}^{+ \infty} 2^{- \beta n} 
	f ( \eta_{x}^{2 \lambda} ) .
\end{split}
\end{equation}
We have $\eta, \zeta \in \mathrm{cst}_{K,1}\, \mathscr{B}_{\gamma}^{r}$, see 
\eqref{eq:eta_zeta_are_in_Bcr_when_K_preserves_polynomials}.
Since $f \in \mathcal{Z}^{\gamma}$
and $r > - \gamma$ by assumption,
we can bound $f (\zeta_{x}^{2\rho 2^{-n}} )$ and 
$f ( \eta_{x}^{2 \lambda} ) $ by \eqref{eq:Cgamma}
and sum the geometric series to get (note that $c+1 > \gamma+\beta$)
			\begin{equation}\label{eq:estimate_Kconvf_1}
				\left| \sfK f (\varphi_x^{\lambda}) \right| \lesssim \sum\limits_{n = 0}^{N_{\lambda} - 1} 2^{- \beta n} (2^{n} \lambda)^{c+1} 2^{- n \gamma} \, + \sum\limits_{n = N_{\lambda}}^{+ \infty} 2^{- \beta n} \lambda^{\gamma} \lesssim \lambda^{\gamma + \beta} ,
			\end{equation}
where the multiplicative constant depends only on the kernel $\sfK$, the compact $K$
and the distribution $f$, as well as on $\bar\lambda$. This concludes our first estimate.

We next bound $\sfK f(\varphi_x)$ for $\varphi \in \mathscr{B}_{}^{r}$ 
which may not annihilate polynomials, see \eqref{eq:Cgamma}.
By Lemma~\ref{lemma:test_functions_K_n} 
with $\bar\lambda = 1$ and $c=-1$, we have an analogue of \eqref{eq:usude}
with $\lambda=1$:
			\begin{align*}
				\sfK f (\varphi_x^{}) 
			=	\sum\limits_{n = 0}^{N_{1} - 1} 2^{- \beta n} \,
	f (\zeta_{x}^{2\rho 2^{-n}} ) \, 
	+ \sum\limits_{n = N_{1}}^{+ \infty} 2^{- \beta n} \,
	f ( \eta_{x}^{2 } )  \,,
			\end{align*}
for suitable $\eta, \zeta \in \mathrm{cst}_{K,1}\, \mathscr{B}^{r}$.
Since $2 \le 2\rho 2^{-n} \le 2\rho$ for $n \le N_1-1$,
we can bound $|f (\zeta_{x}^{2\rho 2^{-n}} )| \lesssim 1$ and
$|f ( \eta_{x}^{2 } )| \lesssim 1$ by \eqref{eq:order2}, because
$f$ is a distribution of order~$r$, hence
			\begin{equation}\label{eq:estimate_Kconvf_2}
				\left| \sfK f (\varphi_x^{}) \right| \lesssim \sum\limits_{n \in \mathbb{N}} 2^{- \beta n} \lesssim 1 .
			\end{equation}

From the calculations above, the estimates \eqref{eq:estimate_Kconvf_1}
and \eqref{eq:estimate_Kconvf_2} hold uniformly over $x \in K$, $\lambda \in (0, 1]$, $\varphi \in \mathscr{B}^{r}_{\gamma + \beta}$ for any $r > - \gamma$.
In fact, from the results of Appendix~\ref{section:independence_in_r}, see Proposition~\ref{prop:independence_in_r}, this remains true even for $r > - \gamma - \beta$, and thus $f \in \mathcal{Z}^{\gamma + \beta}$.

By tracking the constants in the estimates, we have shown that
	\begin{equation*}
		\| \sfK f \|_{\mathcal{Z}_{K}^{\gamma + \beta}} \lesssim \| f \|_{\mathcal{Z}_{K^{\prime}}^{\gamma}} ,
	\end{equation*}

\noindent where the compact $K^{\prime} \supset K$ on the right-hand side depends only on $K$ and the kernel $\sfK$, whence the continuity of the map $\mathcal{Z}^{\gamma} \to \mathcal{Z}^{\gamma + \beta}, f \mapsto \sfK f$.
\end{proof}			
	
	\subsection{Proof of Theorem~\ref{thm:convolution_germs_in_G_checkG}} \label{subsubsection:proof_convolution_germs}

We are going to prove the bounds \eqref{eq:Kconv-hom} and \eqref{eq:Kconv-coh}.
Then, in order to obtain the estimate \eqref{eq:Kconvbound}
which proves continuity from $\mathcal{G}^{\bar{\alpha}; \alpha, \gamma}$ to 
$\mathcal{G}_{\mathrm{weak}}^{\bar\alpha + \beta; \alpha + \beta, \gamma + \beta}$,
it suffices to choose $r = r_{\bar\alpha,\alpha}$ as in \eqref{eq:canonical-r},
recalling Remark~\ref{rem:uniformity} and Proposition~\ref{prop:independence_in_r}
in Appendix~\ref{section:independence_in_r}.
The second part, i.e.\ continuity from 
$\mathcal{G}_{\mathrm{weak}}^{\bar{\alpha}; \alpha, \gamma}$ to 
$\mathcal{G}_{\mathrm{weak}}^{\bar\alpha + \beta; \alpha + \beta, \gamma + \beta}$,
is proved in the same way, 
using \eqref{eq:eta_zeta_are_in_Bcr_when_K_preserves_polynomials} rather 
than \eqref{eq:eta_zeta_are_in_Bcr} of Lemma~\ref{lemma:test_functions_K_n}.

Let $\sfK$ be a $\beta$-regularising kernel of order $(m, r)$.
We fix a compact $K \subset \mathbb{R}^d$ and $\bar\lambda \in [1,\infty)$
and we derive \emph{estimates that are uniform for $x, y \in K$ and $\lambda \in (0, \bar\lambda]$}.
The decomposition argument is the same as in 
the proof of Theorem~\ref{thm:classical_schauder}.

\begin{enumerate}
	\item \emph{Weak homogeneity: proof of \eqref{eq:Kconv-hom}.} 
	We assume that $m > \bar\alpha+\beta$.
	Recalling \eqref{eq:semi-norm-weak-homogeneity}, we need to show that
	there is a constant $\mathrm{cst}' = \mathrm{cst}'_{K,\bar\lambda} < \infty$
	such that, uniformly for $x\in K$, for $\lambda \in (0,\bar\lambda]$ and for
	$\varphi \in \mathscr{B}_{\bar{\alpha} + \beta}^{r}$ and
	$\psi \in \mathscr{B}^{r}$,
\begin{equation}\label{eq:goa1}
	\left| \sfK F_x (\varphi_x^{\lambda}) \right| \le \mathrm{cst}'\, 
	\| F \|_{\mathcal{G}_{\mathrm{hom}; K,2(\bar\lambda + \rho),r}^{\bar{\alpha}}} 
	\lambda^{\bar{\alpha} + \beta} \,, \quad \
	\left| \sfK F_x (\psi_x) \right| \le \mathrm{cst}'\,
	\| F \|_{\mathcal{G}_{\mathrm{hom}; K,2(\bar\lambda + \rho),r}^{\bar{\alpha}}} \,.
\end{equation}
	
	Given $\varphi \in \mathscr{B}_{\bar{\alpha} + \beta}^{r}$,
	we apply Lemma~\ref{lemma:test_functions_K_n}
	with $c := \lfloor \bar{\alpha} + \beta \rfloor$
	(note that $\varphi \in \mathscr{B}_{c}^{r}$).
We set $N_{\lambda} \coloneqq \min \lbrace n \in \mathbb{N} \colon \rho 2^{- n} 
\leq \lambda \rbrace$ and we split the sum in \eqref{eq:definition_of_convolution}:
			\begin{equation*}
				\sfK F_x (\varphi_x^{\lambda}) = \sum\limits_{n = 0}^{N_{\lambda} - 1} F_x ( \sfK_n^* \varphi_{x}^{\lambda}) \, + \sum\limits_{n = N_{\lambda}}^{+ \infty} F_x ( \sfK_n^* \varphi_{x}^{\lambda}) .
			\end{equation*}
		From \eqref{eq:scaled_recentered_eta_zeta}, 
	since $\min\{c+1,m\} = c+1$
because $m > \bar{\alpha} + \beta$,
			\begin{equation} \label{eq:simito}
			\sfK F_x (\varphi_x^{\lambda}) = \sum\limits_{n = 0}^{N_{\lambda} - 1} 
			2^{- \beta n} (2^{n} \lambda)^{c+1} 
			F_x (\zeta_{x}^{2(\rho 2^{-n})} ) \, 
			+ \sum\limits_{n = N_{\lambda}}^{+ \infty} 2^{- \beta n} 
			F_x ( \eta_{x}^{2 \lambda} ) \,,
			\end{equation}
for $\zeta,\eta \in \mathrm{cst}_{K,\bar\lambda}\, \mathscr{B}^{r}$, see \eqref{eq:eta_zeta_are_in_Bcr}.
		Since $F$ is $\bar\alpha$-homogeneous, see \eqref{eq:semi-norm-homogeneity},
	\begin{equation*}
		\left| \sfK F_x (\varphi_x^{\lambda}) \right| 
		\lesssim \mathrm{cst}_{K,\bar\lambda}\,
		\| F \|_{\mathcal{G}_{\mathrm{hom}; K,2(\bar\lambda+\rho),r}^{\bar{\alpha}}} \Bigg\{
		\sum\limits_{n = 0}^{N_{\lambda} - 1} 2^{- \beta n} (2^{n} \lambda)^{c+1} 
		2^{- n \bar{\alpha}} \, + \sum\limits_{n = N_{\lambda}}^{+ \infty} 
		2^{- \beta n} \lambda^{\bar{\alpha}} \Bigg\} .
			\end{equation*}
		Since $c+1 > \bar\alpha+\beta$,
		the geometric series yield the first bound in \eqref{eq:goa1}.

Given $\psi \in \mathscr{B}_{}^{r}$ (which needs not annihilate polynomials),
we apply Lemma~\ref{lemma:test_functions_K_n} with $\varphi = \psi$ and
$\bar\lambda = 1$, $c=-1$.
Similarly to \eqref{eq:simito}, with $\lambda=1$, we can write
			\begin{align} \label{eq:asin1}
	\sfK F_x (\psi_x^{}) & = \sum\limits_{n = 0}^{N_{1} - 1} 
			2^{- \beta n} 
			F_x (\zeta_{x}^{2(\rho 2^{-n})} ) \, 
			+ \sum\limits_{n = N_{1}}^{+ \infty} 2^{- \beta n} 
			F_x ( \eta_{x}^{2 } )
			\end{align}
for $\zeta,\eta \in \mathrm{cst}_{K,1}\, \mathscr{B}^{r}$, see \eqref{eq:eta_zeta_are_in_Bcr}.
Since $F$ is $\bar\alpha$-homogeneous, see \eqref{eq:semi-norm-homogeneity},
	\begin{equation} \label{eq:asin2}
		\left| \sfK F_y (\psi_x^{}) \right| \lesssim \mathrm{cst}_{K,1}\,
		 \| F \|_{\mathcal{G}_{\mathrm{hom}; K,2(1+\rho),r}^{\bar{\alpha}}} 
		 \sum\limits_{n \in \mathbb{N}} 2^{- \beta n}  \,,
	\end{equation}
which completes the proof of \eqref{eq:goa1}.
		
	\item \emph{Weak coherence: proof of \eqref{eq:Kconv-coh}.} 
	We now assume $m > \gamma+\beta$.
	Recalling \eqref{eq:semi-norm-weak-coherence}, we need to show that,
	uniformly for $x\in K$, $\lambda \in (0,\bar\lambda]$ and for
	$\varphi \in \mathscr{B}_{\gamma + \beta}^{r}$, $\psi \in \mathscr{B}^{r}$,
\begin{equation}\label{eq:goa2}
\begin{gathered}
	\left| ( \sfK F_y - \sfK F_x) (\varphi_x^{\lambda}) \right| \le \mathrm{cst}'\, 
	\| F \|_{\mathcal{G}_{\mathrm{coh}; K,2(\bar\lambda+\rho),r}^{\alpha,\gamma}} \,
	\lambda^{\alpha + \beta} (| y - x | + \lambda)^{\gamma-\alpha} \,, \\
	\left| ( \sfK F_y - \sfK F_x) (\psi_x) \right| \le \mathrm{cst}'\,  
	\| F \|_{\mathcal{G}_{\mathrm{coh}; K,2(\bar\lambda+\rho),r}^{\alpha,\gamma}} \,.
\end{gathered}
\end{equation}
	
	By \eqref{eq:definition_of_convolution}
			\begin{equation*}
				( \sfK F_y - \sfK F_x) (\varphi_x^{\lambda}) = \sum\limits_{n \in \mathbb{N}} ( F_y - F_x ) ( \sfK_n^* \varphi_{x}^{\lambda}) 
			\end{equation*}
	For $\varphi \in \mathscr{B}_{\gamma + \beta}^{r}$,
	we apply Lemma~\ref{lemma:test_functions_K_n}
	with $c := \lfloor \gamma + \beta \rfloor$
	(note that $\varphi \in \mathscr{B}_{c}^{r}$)
		and we cut the sum as before, using \eqref{eq:scaled_recentered_eta_zeta}
		and $\min\{c+1,m\} = c+1$ by $m > \gamma+\beta$:
		\begin{align*}
				&( \sfK F_y - \sfK F_x) (\varphi_x^{\lambda}) =
	\\ &			= \sum\limits_{n = 0}^{N_{\lambda} - 1} 2^{- \beta n} (2^{n} \lambda)^{c+1} 
				( F_y - F_x ) ( \zeta_{x}^{2(\rho 2^{-n})} ) 
				+ \sum\limits_{n = N_{\lambda}}^{+ \infty} 2^{- \beta n} 
				( F_y - F_x ) ( \eta_{x}^{2 \lambda} ) \,,
		\end{align*}
for $\zeta,\eta \in \mathrm{cst}_{K,\bar\lambda}\, \mathscr{B}^{r}$, see \eqref{eq:eta_zeta_are_in_Bcr}.
		Since $F$
		is $(\alpha,\gamma)$-coherent, see \eqref{eq:semi-norm-coherence},
\begin{equation}\label{eq:notrouble}
			\begin{split}
				& \left| ( \sfK F_y - \sfK F_x) (\varphi_x^{\lambda}) \right| \le
				\\ & \le
	\mathrm{cst}_{K,\bar\lambda}\, \| F \|_{\mathcal{G}_{\mathrm{coh}; K,
	2(\bar\lambda+\rho),r}^{\alpha,\gamma}}
	\Bigg\{  \sum\limits_{n = 0}^{N_{\lambda} - 1} 2^{- \beta n} (2^{n} \lambda)^{c+1} 2^{- n \alpha} 
				(| y - x | + 2^{- n})^{\gamma - \alpha} \\
				&  \qquad+ \sum\limits_{n = N_{\lambda}}^{+ \infty} 2^{- \beta n} 
				\lambda^{\alpha} (|y - x| + \lambda)^{\gamma - \alpha}
				\Bigg\} .
			\end{split}
\end{equation}
Bounding $(| y - x | + 2^{- n})^{\gamma - \alpha} \lesssim | y - x |^{\gamma - \alpha}
+ 2^{-n(\gamma - \alpha)}$ and recalling that $c+1 > \gamma + \beta \ge \alpha+\beta$,
we can estimate the first sum in the r.h.s.\ by
\begin{equation*}
\begin{split}
	\lambda^{c+1}  \sum\limits_{n = 0}^{N_{\lambda} - 1}
	\big\{ 2^{n(c+1-\alpha-\beta)}| y - x |^{\gamma - \alpha} 
	+ 2^{n(c+1-\gamma-\beta)}\big\} 
	&\lesssim \lambda^{\alpha + \beta} 
	| y - x |^{\gamma-\alpha} + \lambda^{\gamma + \beta}  \\
	&\lesssim \lambda^{\alpha + \beta} 
	(| y - x | + \lambda)^{\gamma-\alpha}. 
\end{split}
\end{equation*}
Also the second sum in \eqref{eq:notrouble} gives
$\lambda^{\alpha+\beta} (|y - x| + \lambda)^{\gamma - \alpha}$, therefore we obtain
the first bound in \eqref{eq:goa2}.

For the second bound in \eqref{eq:goa2}, we argue as in \eqref{eq:asin1}-\eqref{eq:asin2}:
given $\psi \in \mathscr{B}_{}^{r}$ (which needs not annihilate polynomials), we can write
			\begin{align*}
	( \sfK F_y - \sfK F_x) (\psi_x) & 
	= \sum\limits_{n = 0}^{N_{1} - 1} 
			2^{- \beta n} 
			(F_y-F_x) (\zeta_{x}^{2(\rho 2^{-n})} ) \, 
			+ \sum\limits_{n = N_{1}}^{+ \infty} 2^{- \beta n} 
			(F_y-F_x) ( \eta_{x}^{2 } )
			\end{align*}
for $\zeta,\eta \in \mathrm{cst}_{K,1}\, \mathscr{B}^{r}$, see \eqref{eq:eta_zeta_are_in_Bcr}.
Since $F$ is $(\alpha,\gamma)$-coherent, see \eqref{eq:semi-norm-coherence},
	\begin{equation*}
		\left| ( \sfK F_y - \sfK F_x) (\psi_x) \right| \le
		\mathrm{cst}_{K,1}\, \| F \|_{\mathcal{G}_{\mathrm{coh}; K,
		2(1+\rho),r}^{\alpha,\gamma}}
		 \sum\limits_{n \in \mathbb{N}} 2^{- \beta n} \,,
	\end{equation*}
which completes the proof of \eqref{eq:goa2}.
\end{enumerate}

This completes the proof of  Theorem~\ref{thm:convolution_germs_in_G_checkG}.
\qed

\subsection{Proof of Theorem~\ref{thm:embedding_G_checkG}: homogeneity}
\label{sec:weak-to-strong}

In this subsection we prove ``half'' of Theorem~\ref{thm:embedding_G_checkG},
showing that a weakly homogeneous germ can be turned into an ordinary
homogeneous germ by subtracting a suitable Taylor polynomial.

\begin{theorem}[Positive renormalisation of weakly homogeneous germs]
\label{thm:existence_pointwise_evaluations}

If a germ
$F = (F_x)_{x\in\R^d} \in \mathcal{G}_{\mathrm{weak\,hom};\,r}^{\bar\alpha}$
is weakly $\bar\alpha$-homogeneous of order $r \in \N_0$,
then all pointwise derivatives $D^k F_x ( x )$ for $x\in\R^d$ and $0 \le |k| < \bar\alpha$
are well-defined by \eqref{eq:definition_Dkf(x)2} and they satisfy
the following bound, for any compact $K \subseteq \R^d$ and $\bar\lambda > 0$:
	\begin{equation}\label{eq:estimate_on_F_x^k(x)}
		\sup_{x \in K} 
		| D^k F_x ( x ) | \leq \mathrm{cst} \, 
		\| F \|_{\mathcal{G}_{\mathrm{weak\,hom}; K, \bar{\lambda}, r}^{\bar{\alpha}}} 
		\qquad \text{for } 0 \le |k| < {\bar\alpha} \,,
	\end{equation}
where $\mathrm{cst} > 0$ is a constant depending only on $\bar{\alpha}, \bar\lambda, r$
and on the dimension~$d$.

Recalling \eqref{eq:Taylor-f}, we can then define the germ
$G := F - \cT^{\bar\alpha}(F)$, that is
\begin{equation*}
	G_x \coloneqq F_x - \cT^{\bar\alpha}_x(F_x)
	= F_x - \sum_{0 \le | k | < \bar{\alpha}} D^k F_x ( x ) \, \mathbb{X}_x^k \,.
\end{equation*}

	\begin{enumerate}
		\item If $\bar{\alpha} \not\in \mathbb{N}_0$, then $G$ is $\bar\alpha$-homogenenous, i.e.\ $G \in \cG^{\bar\alpha}_{\hom}$, and
for all compact $K \subseteq \R^d$ and $\bar\lambda > 0$, $r\in\N_0$
\begin{equation}\label{eq:continuity_estimate_embedding_homogeneity}
	\left\| G \right\|_{\mathcal{G}_{\mathrm{hom}; K, \bar{\lambda}, r}^{\bar{\alpha}}} 
	\le \mathrm{cst'} \, 
	\left\| F \right\|_{\mathcal{G}_{\mathrm{weak\,hom}; K, \bar{\lambda}, r}^{\bar{\alpha}}} ,
\end{equation}
where $\mathrm{cst'}$ is a constant depending only on $\bar{\alpha}, r$
and on the dimension~$d$.
	
		\item If $\bar{\alpha} \in \mathbb{N}_0$ and if furthermore 
		$D^k F \in \mathcal{G}_{\mathrm{hom}}^{0}$ for all multi-indices $k$ with 
		$| k | = \bar{\alpha}$, then $G \in \cG^{\bar\alpha}_{\hom}$ and,
for all compact $K \subseteq \R^d$ and $\bar\lambda > 0$, $r\in\N_0$,
\begin{equation*}
	\left\| G \right\|_{\mathcal{G}_{\mathrm{hom}; K, \bar{\lambda}, r}^{\bar{\alpha}}} 
	\le \mathrm{cst'} \, 
	\Big ( \left\| F \right\|_{\mathcal{G}_{\mathrm{weak\,hom}; K, \bar{\lambda}, r}^{\bar{\alpha}}} + \sum_{| k | = \bar{\alpha}}  \left\| D^k F \right\|_{\mathcal{G}_{\mathrm{hom}; K, \bar{\lambda}, r}^{0}}  \Big) ,
\end{equation*}
where $\mathrm{cst'}$ is a constant depending only on $\bar{\alpha}, r$
and on the dimension~$d$.
	\end{enumerate}

\end{theorem}

This result turns out to be a corollary of Lemma~\ref{lemma:pointwise_derivative},
which we prove first.

\begin{proof}[Proof of Lemma~\ref{lemma:pointwise_derivative}]
We fix a distribution $f \in \cD'$ such that \eqref{eq:local-hom} holds, i.e.\
\begin{equation} \label{eq:local-hom2}
	\text{for any $\varphi \in \mathscr{B}_{\delta}$} : \qquad
	|f(\varphi_{x}^\lambda)| \lesssim \lambda^{\delta} \quad
	\text{uniformly for $\lambda \in (0,1]$} \,,
\end{equation}
for some fixed $x \in \R^d$ and $\delta > 0$
(recall the definition \eqref{eq:BB} of $\mathscr{B}_\delta$).
We first show that derivatives $D^k f(x)$ defined as in \eqref{eq:definition_Dkf(x)2}
exist, then we prove the bound \eqref{eq:gbound}.

\emph{Let us fix any test function $\eta \in \cD$ with
$\int \eta = 1$ and $\int \eta(x) \, x^l \, \d x = 0$ for all $1 \le |l| \le {\delta}$.}
Without loss of generality, we assume that $\supp(\eta) \subset B (0, 1)$. We claim that
\begin{equation} \label{eq:eta-polynomials}
	\forall k,l \in \N_0^d \ \text{ with } \ 0 \le |l| \le \delta \colon \ \quad
	\int_{\R^d} \frac{x^l}{l!} \, \partial^k\eta(x) \, d x
	= \begin{cases}
	(-1)^{|l|} & \text{ if } l=k \,, \\
	0 & \text{ if } l \ne k \,.
	\end{cases}
\end{equation}
If $l_i < k_i$ for some $i=1,\ldots, d$, this holds by integration by parts, 
because $\partial_i^{k_i} x_i^{l_i} = 0$.
If $l \ge k$, again by integration by parts,
the integral equals $(-1)^{|k|}\int_{\R^d} \frac{x^{l-k}}{(l-k)!} \, \eta(x) \, dx$, which
gives $(-1)^{|l|}$ for $l = k$ while it vanishes
for $l \ne k$, because
$0 < |l-k| \le |l| \le \delta$ and $\eta$ annihilates monomials with this degree.

Let us check that the limit in \eqref{eq:definition_Dkf(x)2} 
does not depend on the choice of $\eta$: if $\tilde{\eta}$ is another such function, then 
$\partial^k (\eta - \tilde{\eta})$ annihilates monomials $x^l$ for any 
$0 \le |l| \le {\delta}$, by \eqref{eq:eta-polynomials},
hence by \eqref{eq:local-hom2} and the definition of  distributional derivative,
for $|k| < \delta$,
	 \begin{equation*}
		D^k f (\eta_x^{\lambda}) - D^k f (\tilde{\eta}_x^{\lambda}) = ( - \lambda )^{- | k |} f 
		\left( (\partial^k (\eta - \tilde{\eta}))_x^{\lambda} \right) 
		= O\big(\lambda^{\delta - | k |}\big) \xrightarrow[\, \lambda\downarrow 0\,]{} 0 \,.
	\end{equation*}

We next establish the limit in \eqref{eq:definition_Dkf(x)2}.
For any $\lambda \in (0, 1]$ and $N \in \mathbb{N}_0$, we can write
\begin{equation} \label{eq:F_x_as_sum}
	D^k f (\eta_x^{\lambda 2^{- (N + 1)}}) 
	= D^k f ( \eta_x^{\lambda} ) + \sum\limits_{n = 0}^{N} 
	D^k f ( \eta_x^{\lambda 2^{- (n + 1)}} - \eta_x^{\lambda 2^{- n}} ) . 
\end{equation}
Let us define the function
\begin{equation} \label{eq:specialphi}
	\varphi(\cdot) := \eta^{\frac{1}{2}}(\cdot) - \eta(\cdot) = 2^d \, \eta(2\, \cdot) - \eta (\cdot) 
\end{equation}
so that
	\begin{align*}
		 D^k f ( \eta_x^{\lambda 2^{- (n + 1)}} - \eta_x^{\lambda 2^{- n}} ) 
		 & = (- \lambda^{-1} 2^{n + 1})^{| k |} f ( (\partial^k \eta)_x^{2^{- (n + 1)}} ) 
		 - (- \lambda^{-1} 2^{n})^{| k |} f ( (\partial^k \eta)_x^{2^{- n}} ) \\
		 & = (- \lambda^{-1} 2^{n})^{| k |} f ( (\partial^k \varphi)_{x}^{\lambda 2^{- n}} ) \,.
	\end{align*}
Note that $\partial^k \varphi \in \mathscr{B}_{\delta}$ for any $k\in\N_0^d$,
by \eqref{eq:eta-polynomials}. Then our assumption \eqref{eq:local-hom2} yields
	\begin{equation}\label{eq:estimate_of_Fx(phi_k_x)}
		\big|  D^k f ( \eta_x^{\lambda 2^{- (n + 1)}} - \eta_x^{\lambda 2^{- n}} ) \big| 
		= \lambda^{-|k|} \, 2^{n|k|} \,
		\big| f ( (\partial^k \varphi)_{x}^{\lambda 2^{- n}} ) \big|
		\lesssim \lambda^{\delta - | k |} 2^{- n (\delta - | k |)} \,,
	\end{equation}	  
hence the sum in \eqref{eq:F_x_as_sum} converges for $0 \le | k | < \delta$.
This shows that the limit 
\begin{equation} \label{eq:thelim}
	D^k f ( x | \lambda) := \lim_{N\to\infty} D^k f (\eta_x^{\lambda 2^{- N }})
\end{equation}
exists in $\R$ for any fixed $\lambda \in (0,1]$.

Let us show that 
the limit in \eqref{eq:thelim} does not depend on $\lambda$
and it coincides with the limit in \eqref{eq:definition_Dkf(x)2}.
Fix $\lambda, \lambda' \in (0,1]$ and define $\tilde\varphi := \eta^\lambda-\eta^{\lambda'}$,
so that we can write
\begin{equation} \label{eq:uican}
	D^k f ( x | \lambda) - D^k f ( x | \lambda') = \lim_{N\to\infty}
	 D^k f (\tilde\varphi_x^{ 2^{- N}}) = \lim_{N\to\infty}
	2^{N|k|} f ((\partial^k \tilde\varphi)_x^{ 2^{- N}}) \,.
\end{equation}
Note that \eqref{eq:eta-polynomials} still holds if we replace $\eta$ by $\eta^\lambda$ (because $\int \eta^\lambda 
= 1$
and $\int \eta^\lambda(x) \, x^l \, dx 
= 0$ for all $1 \le |l| \le \delta$), and similarly for $\eta^{\lambda'}$.
It follows that $\partial^k \tilde\varphi \in \cB_{\delta}$ for any $k\in\N_0^d$, hence
$|f( (\partial^k \tilde\varphi)_x^{2^{-N}})| = O( 2^{-N\delta})$ as $N\to\infty$, by \eqref{eq:local-hom2}.
In view of \eqref{eq:uican}, we obtain $D^k f ( x | \lambda) = D^k f ( x | \lambda')$
for $0 \le |k| < \delta$, hence the limit in \eqref{eq:thelim} does not depend on $\lambda$
and we simply call it $D^k f ( x )$. 
Recalling \eqref{eq:F_x_as_sum}, for any $\lambda\in(0,1]$ we can write

\begin{equation} \label{eq:F_x_as_sum_2}
\begin{split}
	D^k f (x) 
	&= D^k f ( \eta_x^{\lambda} ) + \sum\limits_{n = 0}^{+\infty} 
	D^k f ( \eta_x^{\lambda 2^{- (n + 1)}} - \eta_x^{\lambda 2^{- n}} ) \\
	&= (- \lambda^{-1})^{| k |} f^{} ( (\partial^k \eta)_x^{\lambda} ) 
	+ \sum\limits_{n = 0}^{+ \infty} (- \lambda^{-1} 2^{n})^{| k |} 
	f ( (\partial^k \varphi)_{x}^{\lambda 2^{- n}} ) \,,
\end{split}
\end{equation}
with $\varphi$ defined in terms of $\eta$ by \eqref{eq:specialphi}.
The first line in \eqref{eq:F_x_as_sum_2} shows that $D^k f (x)$ does
coincide with the limit in \eqref{eq:definition_Dkf(x)2}, because by \eqref{eq:estimate_of_Fx(phi_k_x)},
for $|k| < \delta$,
\begin{equation*}
	| D^k f ( \eta_x^{\lambda} ) -D^k f (x)|
	\lesssim
	\sum\limits_{n = 0}^{+\infty} \lambda^{\delta - | k |} \, 2^{- n (\delta - | k |)}
	\lesssim \lambda^{\delta - | k |} \xrightarrow[\, \lambda\downarrow 0\,]{} 0 \,.
\end{equation*}

We now prove \eqref{eq:gbound}.
Recalling \eqref{eq:Taylor-f}, we set for short
\begin{equation*}
	g \coloneqq f - \cT^{\delta}_x(f)
	= f - \sum_{0 \le | k | < \delta} D^k f ( x ) \, \mathbb{X}_x^k \qquad
	\text{where} \qquad \mathbb{X}_x^k(\cdot) := \frac{(\,\cdot\,-x)^k}{k!} \,.
\end{equation*}
We fix an arbitrary $\psi \in \cD(B(0,1))$ and we need to prove that
	\begin{equation} \label{eq:needto}
		g (\psi_x^{\lambda}) = f (\psi_x^{\lambda}) - \sum\limits_{0 \le | k | < \delta} 
		D^k f ( x )  \, \mathbb{X}_x^{k} (\psi_x^{\lambda})
		= O(\lambda^\delta) \,.
	\end{equation}
	
First consider the case $\delta \notin \mathbb{N}$.
Since $\mathbb{X}_x^{k} (\psi_x^{\lambda}) = \lambda^{|k|} \mathbb{X}_0^{k} (\psi)$,
the representation formula \eqref{eq:F_x_as_sum_2} for $D^k f ( x )$ yields 
\begin{equation}\label{eq:gsplit}
	\begin{split}
		g (\psi_x^{\lambda}) & = \underbrace{f (\psi_x^{\lambda}) 
		- \sum\limits_{0 \le | k | < \delta} \lambda^{|k |} \, \mathbb{X}_0^{k} (\psi)
		\, (-\lambda^{-1})^{| k |} \, f ( (\partial^k \eta)_x^{\lambda} )}_{\eqqcolon A} \\
		& \quad - \underbrace{\sum\limits_{0 \le | k | < \delta} \lambda^{| k |}
		\, \mathbb{X}_0^{k} (\psi) \sum\limits_{n = 0}^{+ \infty} (- \lambda^{-1} 2^{n})^{| k |}
		\, f ( (\partial^k \varphi)_{x}^{\lambda 2^{- n}} ) }_{\eqqcolon B} \,.
	\end{split}
\end{equation}
Recalling \eqref{eq:estimate_of_Fx(phi_k_x)}, and summing the geometric series, we obtain 
\begin{equation*} 
	|B| \lesssim \sum\limits_{0 \le | k | < \delta} \lambda^{| k |}
	\, \mathbb{X}_0^{k} (\psi) 
	\sum\limits_{n = 0}^{+ \infty} \lambda^{\delta-|k|} \, 2^{-n(\delta-|k|)}
	= \lambda^\delta \sum\limits_{0 \le | k | < \delta} 
	\frac{\mathbb{X}_0^{k} (\psi)}{1-2^{-(\delta-|k|)}} \lesssim \lambda^\delta \,.
\end{equation*}
To prove \eqref{eq:needto}, it remains to show that also $|A| \lesssim \lambda^{\delta}$.

It is convenient to define the function
\begin{equation}\label{eq:check-psi}
	\check{\psi} \coloneqq \psi
	- \sum\limits_{0 \le | k | \le \delta} (-1)^{| k |} \, \mathbb{X}_0^{k} (\psi) \, \partial^k \eta \,.
\end{equation}
When $\delta \not\in \N$, the sum ranges over $0 \le |k| < \delta$ and we can write
$A = f (\check{\psi}_x^{\lambda})$.
Then the desired bound $|A| \lesssim \lambda^{\delta}$ follows by assumption \eqref{eq:local-hom2}, 
because we claim that
\begin{equation}\label{eq:check-claim}
	\check{\psi} \in \mathscr{B}_\delta \,.
\end{equation}
Indeed, we clearly have $\check{\psi} \in \cD(B(0,1))$ and, moreover,
$\bbX_0^l(\check{\psi}) = \int \frac{x^l}{l!} \,\check{\psi}(x) \, d x = 0$
for any $l\in\N_0^d$ with $0 \le |l| \le \delta$, by \eqref{eq:eta-polynomials}.

Finally, consider the case where $\delta \in \mathbb{N}$.
We modify \eqref{eq:gsplit} writing $A = A_1 + A_2$ with
	\begin{equation*}
	\begin{split}
		A_1 & := f (\psi_x^{\lambda}) 
		- \sum\limits_{0 \le | k | \leq \delta} \lambda^{|k |} \, \mathbb{X}_0^{k} (\psi)
		\, (-\lambda^{-1})^{| k |} \, f ( (\partial^k \eta)_x^{\lambda} ) \\
		A_2 & := \sum\limits_{| k | = \delta} \lambda^{|k |} \, \mathbb{X}_0^{k} (\psi) \, 
		(-\lambda^{-1})^{| k |} \, f ( (\partial^k \eta)_x^{\lambda} )  \,.
	\end{split}
\end{equation*}
We can bound $|A_1| \lesssim \lambda^{\delta}$ as before, while $A_2$ can be rewritten as 
	\begin{equation} \label{eq:check-claim_2}
	A_2 =  \sum\limits_{| k | = \delta} \lambda^{|k |} \, \mathbb{X}_0^{k} (\psi) \, 
	D^k f ( \eta_x^{\lambda} ) = O ( \lambda^{\delta} ) ,
	\end{equation}
since when $\delta \in \mathbb{N}$ we furthermore assume $| D^k f ( \eta_x^{\lambda} ) | \lesssim 1$.
This completes the proof.
\end{proof}

\begin{proof}[Proof of Theorem~\ref{thm:existence_pointwise_evaluations}]
If $\bar\alpha \le 0$ there is nothing to prove,
hence we assume $\bar\alpha > 0$.
We use the same notation as in the proof of Lemma~\ref{lemma:pointwise_derivative}.

Given a weakly homogeneous germ
$F = (F_x)_{x\in\R^d} \in \mathcal{G}_{\mathrm{weak\,hom}}^{\bar\alpha}$,
the distribution $f = F_x$
satisfies \eqref{eq:local-hom} with $\delta = \bar\alpha$, hence 
$D^k F_x (x)$ is well-defined for $x\in\R^d$ and $0 \le |k| < \bar\alpha$.
To obtain the estimate \eqref{eq:estimate_on_F_x^k(x)},
we apply the second line of \eqref{eq:F_x_as_sum_2} with $\lambda=\bar\lambda$
and we note that,
for any compact $K \subseteq \R^d$ and $x\in K$, we can bound
\begin{gather*}
	|F_x ( (\partial^k \eta)_x^{\bar\lambda} ) | \le \bar\lambda^{\bar\alpha}\,
	\| F \|_{\mathcal{G}_{\mathrm{weak\,hom}; K, \bar{\lambda}, r}^{\bar{\alpha}}} 
	\, \|\partial^k\eta\|_{C^{r}}  \,, \\
	|F_x( (\partial^k \varphi)_{x}^{\bar\lambda 2^{- n}} )| \le 
	\bar\lambda^{\bar\alpha} \, 2^{-n \bar\alpha}\,
	\| F \|_{\mathcal{G}_{\mathrm{weak\,hom}; K, \bar{\lambda} , r}^{\bar{\alpha}}} 
	\, \|\partial^k\varphi\|_{C^{r}}  \,,
\end{gather*}
see \eqref{eq:weak-homogeneity} and \eqref{eq:semi-norm-weak-coherence}
(we recall that $\eta$ is a fixed suitable test function  and $\varphi = \eta^{1/2} - \eta$).
Then \eqref{eq:estimate_on_F_x^k(x)} follows 
by summing the second line of \eqref{eq:F_x_as_sum_2}.

Furthermore $|G_x(\psi_x^\lambda)| \le \, c(\psi) \, \lambda^{\bar\alpha}$
by \eqref{eq:gbound}, where $c(\psi)$ can be estimated by \eqref{eq:gsplit}-\eqref{eq:check-claim_2}:
more precisely, if we fix a compact $K \subseteq \R^d$, $\bar\lambda > 0$ and $r\in\N$, then
for any $\lambda \in (0,\bar\lambda]$, $x\in K$ and $\psi \in \cD(B(0,1))$ we have, in the case where $\bar{\alpha} \notin \mathbb{N}_0$,
\begin{equation*}
\begin{split}
	|G_x(\psi_x^\lambda)| 
	& \le 
	\| F \|_{\mathcal{G}_{\mathrm{weak\,hom}; K, \bar{\lambda}, r}^{\bar{\alpha}}} 
	\bigg\{
	\sum\limits_{0 \le | k | < \delta} 
	\frac{\mathbb{X}_0^{k} (\psi)}{1-2^{-(\bar\alpha-|k|)}}\, \|\partial^k\varphi\|_{C^{r}}
	+ \|\check{\psi}\|_{C^r}\bigg\} \lambda^{\bar\alpha} \\
	& \le  \mathrm{cst'} \, 
	\| F \|_{\mathcal{G}_{\mathrm{weak\,hom}; K, \bar{\lambda}, r}^{\bar{\alpha}}} 
	\, \|\psi\|_{C^r} \, \lambda^{\bar\alpha} \,,
\end{split}
\end{equation*}
where $\mathrm{cst'}$ depends only on $\bar{\alpha}, r$
and on the dimension~$d$, whence the result recalling the definition \eqref{eq:semi-norm-homogeneity}
of the homogeneity semi-norm 
$\left\| G \right\|_{\mathcal{G}_{\mathrm{hom}; K, \bar{\lambda}, r}^{\bar{\alpha}}} $.

The same argument also establishes the announced bound when $\bar{\alpha} \in \mathbb{N}_0$.
\end{proof}

\subsection{Proof of Theorem~\ref{thm:embedding_G_checkG}: coherence}
\label{sec:embedding_G_checkG}

In this subsection, we complete the proof of Theorem~\ref{thm:embedding_G_checkG}.
We first show how to decompose any test function
$\psi \in \mathscr{B}_{}^{r}$ as the sum of
a single test function localized at a large scale $2^M$
plus \emph{a sum of test functions which annihilate polynomials}, localised at scales
$2^n$ for $0 \le n \le M$.

\begin{lemma}[Large scale decomposition]\label{lemma:decomposition_test_functions}
Fix $r \in \mathbb{N}_0$ and $c \in \mathbb{N}_0 \cup \lbrace -1 \rbrace$.
For any test function $\psi \in \mathscr{B}_{}^{r}$ and for any $M \in \mathbb{N}_0$,
we can write
	\begin{equation} \label{eq:decomposition}
		\psi = (\tilde{\psi}^{[M]})^{2^M} + \sum\limits_{n = 0}^{M} (\check{\psi}^{[n]})^{2^{n}} 
	\end{equation}
for suitable test-functions
	\begin{equation} \label{eq:boundsB}
		\tilde{\psi}^{[M]} \in \mathrm{cst} \mathscr{B}^{r} \,,
		\qquad \check{\psi}^{[n]} \in \mathrm{cst} \mathscr{B}_c^{r}  \,,
		\qquad \text{ for } 0 \leq n \leq M  \,,
	\end{equation}
where the constant $\mathrm{cst} > 0$ only depends on $r, c$ and on the dimension~$d$.
\end{lemma}

\begin{proof}
We fix $c \in \N_0 \cup \{-1\}$.
We also fix a test-function $\eta \in \mathcal{D} ( B (0, 1) )$ such that $\int \eta ( x ) d x = 1$ 
and $\int x^k \, \eta ( x ) \, d x = 0$ for multi-indices $k$ with $1 \le | k | \leq c$.

Given any test function $\psi \in \cD(B(0,1))$, we define
$\check{\psi}(x)$ by \eqref{eq:check-psi} with $\delta = c$
and we already showed that $\check{\psi} \in \mathscr{B}_c$, see \eqref{eq:check-claim}.
We further observe that, given $r\in\N_0$ and a multi-index $l\in\N_0^d$
with $|l| \le r$, we can bound
\begin{equation*}
	|\partial^l \check{\psi}(x)| \le |\partial^l \psi(x)| + 
	\sum\limits_{0 \le | k | \leq c} \|\psi\|_\infty \, \frac{|B(0,1)|}{k!} \, |\partial^k \eta(x)|
	\le \mathrm{cst} \, \|\psi\|_{C^{r}} \,,
\end{equation*}
where $\mathrm{cst} > 0$ only depends on $c,r,d$ (via the chosen $\eta$). 
Overall, we have proved that
\begin{equation} \label{eq:psiimplies}
	\psi \in \mathscr{B}^r \qquad \Longrightarrow \qquad
		\check{\psi} \in \mathrm{cst} \, \mathscr{B}_{c}^{r} \,.
\end{equation}

We next perform the following telescoping sum, for any $| k | \leq c$:
\begin{align*}
	\partial^k \eta & = 2^{- M | k |} (\partial^k \eta)^{2^M} + \sum\limits_{n = 0}^{M - 1} 
	\left( 2^{- n | k |} (\partial^k \eta)^{2^{n}} - 2^{- (n + 1) | k |} 
	(\partial^k \eta)^{2^{n + 1}} \right) \\
	& = 2^{- M | k |} (\partial^k \eta)^{2^M} + \sum\limits_{n = 0}^{M - 1} 
	2^{- (n + 1) | k |} ( \partial^k \varphi )^{2^{n + 1}} \,,
\end{align*}
where $\varphi := \eta^{\frac{1}{2}} - \eta$ was defined in \eqref{eq:specialphi}
and it satisfies, for $|k|\le c$,
\begin{equation}\label{eq:phibelongs}
	\partial^k \varphi \in \mathrm{cst} \, \mathscr{B}_{c}^{r} \,,
\end{equation}
see \eqref{eq:eta-polynomials} (with $\delta=c$).
Plugging this into \eqref{eq:check-psi}, we obtain
	\begin{align*}
		\psi = \check{\psi} 
		\,+\, \sum\limits_{| k | \leq c} (-1)^{| k |} \mathbb{X}_0^{k} (\psi) 2^{- M | k |} 
		(\partial^k \eta)^{2^M} 
		\,+\, \sum\limits_{n = 1}^{M} \sum\limits_{| k | \leq c} (-1)^{| k |} \mathbb{X}_0^{k} (\psi) 
		2^{-n | k |} ( \partial^k \varphi )^{2^{n}} \,,
	\end{align*}
which yields the decomposition \eqref{eq:decomposition} once we define
	\begin{equation*}
	\left\{
	\begin{split}
			\rule{0pt}{1.3em}\tilde{\psi}^{[M]}(x) & \coloneqq 
			\sum\limits_{| k | \leq c} (-1)^{| k |} \, \mathbb{X}_0^{k} (\psi) \, 2^{- M | k |} 
			\,\partial^k \eta(x) \,, \\
			 \check{\psi}^{[0]}(x) & \coloneqq \check{\psi}(x) \,, \\
			\check{\psi}^{[n]}(x) & \coloneqq \sum\limits_{| k | \leq c} (-1)^{| k |} \,
			\mathbb{X}_0^{k} (\psi) \, 2^{- n | k |}
			\, \partial^k \varphi(x) \quad \text{for } n = 1, \cdots, M \,.
	\end{split}
	\right.
	\end{equation*}
Finally, relation \eqref{eq:boundsB} follows by \eqref{eq:psiimplies}
and \eqref{eq:phibelongs}.
\end{proof}

\begin{proof}[Proof of Theorem~\ref{thm:embedding_G_checkG}]
We fix a germ $F \in \mathcal{G}_{\mathrm{weak}}^{\gamma; \alpha, \gamma}
= \mathcal{G}_{\mathrm{weak\,hom}}^{\gamma} \cap
\mathcal{G}_{\mathrm{weak\,coh}}^{\alpha,\gamma}$.
We already know from Theorem~\ref{thm:existence_pointwise_evaluations} that $G$ is well-defined 
and it is $\gamma$-homogeneous, i.e.\ $G \in \mathcal{G}_{\mathrm{hom}}^{\gamma}$.
Thus, it remains to show that $G$ is $(\alpha\wedge 0, \gamma)$-coherent, i.e.\ 
$G \in \mathcal{G}_{\mathrm{coh}}^{\alpha\wedge 0,\gamma}$.

Let us fix $r\in\N$, $\bar\lambda > 0$ and a compact $K \subset \mathbb{R}^d$.
We shall estimate $(G_y - G_x) (\psi_x^{\lambda})$ for $\lambda \in (0, \bar\lambda]$, $x, y \in K$ and
a test function $\psi \in \mathscr{B}^{r}$ (which does not necessarily annihilate polynomials),
which yields a bound on the semi-norm
$\left\| G \right\|_{\mathcal{G}_{\mathrm{hom}; K, \bar{\lambda}, r}^{\gamma}}$,
see \eqref{eq:semi-norm-coherence}.

A first heuristic observation is that when $| y - x | \lesssim \lambda$, then 
$\psi_x^{\lambda}$ can also be seen as a test function centered in $y$ with scale $\lambda$, 
say $\psi_x^{\lambda} \simeq \tilde{\psi}_y^{\lambda}$.
Since $G \in \mathcal{G}_{\mathrm{hom}}^{\gamma}$ is $\gamma$-homogeneous
(even against test-functions which do not annihilate polynomials),
	\begin{equation*}
		| (G_y - G_x) (\psi_x^{\lambda}) | \lesssim |G_y (\tilde{\psi}_y^{\lambda})| 
		+ | G_x (\psi_x^{\lambda}) | \lesssim \lambda^{\gamma} 
		= \lambda^{\alpha} \, \lambda^{\gamma-\alpha} 
		\le \lambda^\alpha
		( | y - x | + \lambda)^{\gamma - \alpha} ,
	\end{equation*}
implying that $G$ is $(\alpha, \gamma)$-coherent.
However, in general $\lambda$ and $| y - x |$ need not be comparable, 
hence we resort to the large scale decomposition of Lemma~\ref{lemma:decomposition_test_functions}.

More precisely, let us define
	\begin{equation*}
		M \coloneqq M_{x, y, \lambda} \coloneqq \min 
		\bigg\lbrace n \in \mathbb{N}_0 \colon
		2^{- n} \leq \frac{\lambda}{\lambda + | y - x |} \bigg\rbrace \,,
	\end{equation*}
so that $\lambda + |y-x| \le \lambda 2^M \le 2(\lambda + |y-x|)$.
Applying the decomposition \eqref{eq:decomposition}, we write
	\begin{equation*}
		(G_y - G_x) (\psi_x^{\lambda}) = \underbrace{(G_y - G_x) 
		( (\tilde{\psi}^{[M]})_x^{\lambda 2^M} )}_{A} + 
		\underbrace{\sum\limits_{n = 0}^{M} 
		(G_y - G_x) ( (\check{\psi}^{[n]})_x^{\lambda 2^{n}} )}_{B} \,.
	\end{equation*}
We now estimate $A$ and $B$ separately.

\paragraph{Estimate of $A$}
Since the choice of $M$ implies $\lambda 2^M \approx \lambda + | y - x |$, 
we are at the correct scale to apply the argument sketched above.
We can recenter $(\tilde{\psi}^{[M]})_x^{\lambda 2^{M}}$ at point $y$, that is
we can write
\begin{equation*}
		(\tilde{\psi}^{[M]})_x^{\lambda 2^{M}} 
		= (\tilde{\tilde \psi}^{[M]})_y^{\lambda 2^{M + 1}} \qquad \text{where} \qquad
		\tilde{\tilde{\psi}}^{[M]} = \tilde{\tilde{\psi}}^{[M, y, x, \lambda]}
		:= (\tilde{\psi}^{[M]})_{\frac{x - y}{\lambda 2^{M + 1}}}^{\frac{1}{2}}
	\end{equation*}
Note that $|\frac{x - y}{\lambda 2^{M + 1}}| \leq \frac{1}{2}$ by definition of~$M$,
therefore $\tilde{\tilde{\psi}}^{[M]} \in \mathrm{cst} \mathscr{B}^{r}$.
As a consequence, using the fact that 
$G \in \mathcal{G}_{\mathrm{hom}}^{\gamma}$ and
$\lambda 2^{M + 1} \leq 4 (\lambda + \mathrm{diam} (K)) \eqqcolon \bar{\lambda}^{\prime}$ 
by definition of $M$,
\begin{equation}\label{eq:estimA}
	\begin{split}
		| A | 	& \leq 
		| G_y ( (\tilde{\tilde \psi}^{[M]})_y^{\lambda 2^{M + 1}}) | + 
		| G_x ( (\tilde{\psi}^{[M]})_x^{\lambda 2^M} ) | \\
		& \leq \mathrm{cst} \,
		\| G \|_{\mathcal{G}_{\mathrm{hom}; K, \bar{\lambda}^{\prime}, r}^{\gamma}} 
		(\lambda 2^{M})^{\gamma} \\
		& \leq \mathrm{cst'} \,
		\| F \|_{\mathcal{G}_{\mathrm{weak\,hom}; K, \bar{\lambda}^{\prime}, r}^{\gamma}} 
		\lambda^{0} (\lambda + | y - x |)^{\gamma - 0} ,
	\end{split} 
\end{equation}

\noindent where in the last estimate we used \eqref{eq:continuity_estimate_embedding_homogeneity} and the definition of $M$.

\paragraph{Estimate of $B$}
Since $\check{\psi}^{[n]}$ annihilates polynomials up to degree $\gamma$, 
see \eqref{eq:boundsB}, we have
$(G_y - G_x) ( (\check{\psi}^{[n]})_x^{\lambda 2^{n}} ) 
= (F_y - F_x) ( (\check{\psi}^{[n]})_x^{\lambda 2^{n}} )$.
Since $F$ is weakly $(\alpha,\gamma)$-coherent,
	\begin{align*}
		| B | & \le \mathrm{cst} \,
	\| F \|_{\mathcal{G}_{\mathrm{weak\,coh}; K, \bar{\lambda}^{\prime}, r}^{\alpha,\gamma}}
		\sum\limits_{n = 0}^{M} (\lambda 2^n)^{\alpha} (\lambda 2^n + | y - x |)^{\gamma - \alpha} \\
		& \lesssim \mathrm{cst} \,
	\| F \|_{\mathcal{G}_{\mathrm{weak\,coh}; K, \bar{\lambda}^{\prime}, r}^{\alpha,\gamma}}
	\bigg\{ \underbrace{\lambda^{\gamma} \sum\limits_{n = 0}^{M} 2^{n \gamma}}_{B_1}
	\,+\, \underbrace{\lambda^{\alpha} | y - x |^{\gamma - \alpha} 
	\sum\limits_{n = 0}^{M} 2^{n \alpha}}_{B_2} \bigg\}\, .
	\end{align*}
Let us estimate $B_1$ and $B_2$
for $\gamma \ne 0$ and $\alpha \ne 0$ 
(by the assumptions in Theorem~\ref{thm:embedding_G_checkG}).
\begin{itemize}
\item \emph{Estimate of $B_1$.}
If $\gamma < 0$ then $B_1 \lesssim \lambda^{\gamma} \le \lambda^0 \, 
(\lambda+|y-x|)^{\gamma-0}$, while if $\gamma > 0$ we can bound
$B_1 \lesssim (\lambda 2^{M})^{\gamma} \le 2 \lambda^0 \, (\lambda + | y - x |)^{\gamma-0}$ 
by definition of $M$. 

\item \emph{Estimate of $B_2$.}
If $\alpha > 0$ then $B_2
\lesssim (\lambda 2^M)^{\alpha} | y - x |^{\gamma - \alpha} \lesssim 
\lambda^0 \, (\lambda+|y-x|)^{\gamma-0}$, while
if $\alpha < 0$ then $B_2 \lesssim \lambda^{\alpha} | y - x |^{\gamma - \alpha} 
\lesssim \lambda^{\alpha} (| y - x | + \lambda)^{\gamma - \alpha}$.
\end{itemize}
In all cases, we have shown that for any $r\in\N$
	\begin{equation} \label{eq:estimB}
		| B | \le \mathrm{cst} \,
	\| F \|_{\mathcal{G}_{\mathrm{weak\,coh}; K, \bar{\lambda}^{\prime}, r}^{\alpha,\gamma}}
		\, \lambda^{\alpha \wedge 0} (| y - x| + \lambda)^{\gamma - \alpha \wedge 0} .
	\end{equation}

\paragraph{Conclusion}
It follows from \eqref{eq:estimA} and \eqref{eq:estimB} that
$G$ is $(\alpha\wedge 0, \gamma)$-coherent with
\begin{equation*}
	 \| G \|_{\mathcal{G}_{\mathrm{coh}; K, \bar{\lambda}, r}^{\alpha \wedge 0, \gamma}} 
	 \lesssim \mathrm{cst} \, \big( 
	 \| F \|_{\mathcal{G}_{\mathrm{weak\,hom}; K, \bar{\lambda}^{\prime}, r}^{\gamma}} +
	 \| F \|_{\mathcal{G}_{\mathrm{weak\,coh}; K, \bar{\lambda}^{\prime}, r}^{\alpha,\gamma}} \big)
	\end{equation*}
where $\bar{\lambda}^{\prime} \coloneqq 4 (\bar{\lambda} + \mathrm{diam} (K))$
and $\mathrm{cst}$ depends only on $\alpha, \gamma, d$.
Together with the bound \eqref{eq:continuity_estimate_embedding_homogeneity}
on the coherence semi-norm that we obtained in
Theorem~\ref{thm:existence_pointwise_evaluations}, namely
\begin{equation*}
	\left\| G \right\|_{\mathcal{G}_{\mathrm{hom}; K, \bar{\lambda}, r}^{\gamma}} 
	\le \mathrm{cst} \,  
	\left\| F \right\|_{\mathcal{G}_{\mathrm{weak\,hom}; K, \bar{\lambda}, r}^{\gamma}} \,,
\end{equation*}
we have proved the announced continuity of the map $F \mapsto G$.
\end{proof}

\section{Proofs for of our Main Result II} \label{section:proof_main_result_II}

In this section we establish our second main result,
the \emph{multilevel Schauder estimates} in
Theorem~\ref{thm:multi_level_schauder_estimate}, alongside Proposition~\ref{prop:enhanced-continuity} (enhanced continuity).
We will also prove Proposition~\ref{prop:stability_of_properties_of_Gamma} (properties of reexpansion).

\subsection{Proof of Theorem~\ref{thm:multi_level_schauder_estimate}}
\label{sec:multilevel}

We fix a model $M = (\Pi,\Gamma)$,
a modelled distribution $f \in \cD^\gamma_M$ or order $\gamma$
with a reconstruction $\cR^\gamma\langle f,\Pi\rangle$,
and a $\beta$-regularising kernel $\sfK$
that satisfy the assumptions of Theorem~\ref{thm:multi_level_schauder_estimate}.
We need to prove that:
\begin{itemize}
\item $(\hat\Pi,\hat\Gamma)$ in \eqref{eq:hatPi} and \eqref{eq:hatGamma0}-\eqref{eq:hatGamma} is
a model;
\item $\hat f$ in \eqref{eq:hatf} is a modelled distribution;
\item the equality \eqref{eq:prophatf} holds.
\end{itemize}
We correspondingly split the proof in three parts.

\subsubsection{Proof that $(\hat\Pi,\hat\Gamma)$ is a model}
\label{sec:hatmodel}

We first show that $\hat\Pi$ and $\hat\Gamma$ are well-defined,
that is all terms appearing in \eqref{eq:hatPi} and 
\eqref{eq:hatGamma0}-\eqref{eq:hatGamma} are well-posed.
By assumption \eqref{eq:homogeneity-Pi}, each germ $\Pi^i = (\Pi^i_x)_{x\in\R^d}$ 
is $\alpha_i$-homogeneous, hence
 $\sfK\Pi^i = (\sfK\Pi^i_x)_{x\in\R^d}$ 
is well-defined and weakly $(\alpha_i+\beta)$-homogeneous, by
Theorem~\ref{thm:convolution_germs_in_G_checkG};
it follows by Theorem~\ref{thm:existence_pointwise_evaluations} that
pointwise derivatives $D^k (\sfK \Pi^i_x)(x)$ are well-defined
for $|k| < \alpha_i+\beta$, 
so the definitions \eqref{eq:hatPi} and 
\eqref{eq:hatGamma0}-\eqref{eq:hatGamma} of $\hat\Pi$ and $\hat\Gamma$ are well-posed.

It remains to show that the property of reexpansion \eqref{eq:relation_Pi_Gamma}
and the homogeneity condition \eqref{eq:homogeneity-Pi} hold
for $(\hat\Pi, \hat\Gamma)$.
We recall that $\hat{I} = I \sqcup \poly(\gamma+\beta)$, see \eqref{eq:hatI}.

\paragraph{Property of reexpansion}
Let us first check condition \eqref{eq:relation_Pi_Gamma} 
for $(\hat\Pi,\hat\Gamma)$, that is
\begin{equation} \label{eq:netoch1}
	\hat\Pi^a_y = \sum_{b \in \hat{I}} \hat\Pi^b_x \, \hat\Gamma^{ba}_{xy} \,.
\end{equation}
When $a = k \in \poly(\gamma+\beta)$, this holds by Example~\ref{ex:polynomial-model},
because $\hat\Pi^a = \bbX^k$ 
is a monomial and $\hat\Gamma^{ba}=0$ for $j\in I$ while
$\hat\Gamma^{ba} = (\Gamma^\poly)^{lk}$ for $b=l \in \poly(\gamma+\beta)$,
see \eqref{eq:hatGamma0}.

We then fix $a=i  \in {I}$ and we rephrase relation \eqref{eq:netoch1}, that we need to prove, as
\begin{equation} \label{eq:netoch2}
	\hat\Pi^i_y 
	= \sum_{j \in I} \hat\Pi^j_x \, \Gamma^{ji}_{xy}
	+ \sum_{l \in \poly(\gamma+\beta)} \bbX^l_x \, \hat\Gamma^{la}_{xy} 
\end{equation}
(note that $\hat\Gamma^{ba} = \Gamma^{ji}$ for $(b,a)=(j,i) \in I\times I$, see
\eqref{eq:hatGamma0}).
Let us set for short
\begin{equation}\label{eq:A}
	A_x^{i,l}:=\un{(|l| < \alpha_i + \beta)}\,  
	D^l(\sfK \Pi_x^i) (x) \qquad
	\text{for } \ x\in\R^d, \ i\in I, \ l\in\N_0^d \,,
\end{equation}
so that we can write, by \eqref{eq:hatPi},
\[
\begin{split}
	\hat\Pi^i_y 
	= \sfK \Pi^i_y - 
	\sum_{l\in\poly(\gamma+\beta)} A_y^{i,l} \, \bbX_y^l \, .
\end{split}
\]
Since $\Pi^i_y=\sum_{j\in I} \Pi^j_x \, \Gamma^{j i}_{xy}$, we obtain
\begin{align}
	\hat\Pi_y^i & = \sum_{j\in I} (\sfK \Pi^j_x) \, \Gamma^{j i}_{xy} 
	-\sum_{l\in\poly(\gamma+\beta)} A_y^{i,l} \, \bbX_y^l \nonumber \\ 
	&=\sum_{j\in I}\hat\Pi^j_x\,\Gamma_{xy}^{ji}+\sum_{l\in \poly(\gamma+\beta)} 
	\bigg( - A_y^{i,l}
	 \, \bbX^l_y+ \sum_{j\in I} A_x^{j,l}\, \Gamma_{xy}^{ji}\, \bbX^l_x\bigg) \,.
	 \label{eq:reexp}
 \end{align}
The first term in the RHS matches with the one in \eqref{eq:netoch2}.
Let us show that also the second terms match:
for $l\in\poly(\gamma+\beta)$ and $i\in I$,
we rewrite \eqref{eq:hatGamma} as
\begin{equation*}
\begin{split}
	\hat\Gamma^{li}_{xy} 
	& = \sum_{j\in I} A^{j,l}_x \, \Gamma^{ji}_{xy} - \sum_{k \in \poly(\gamma+\beta)} 
	A_y^{i,k} \, (\Gamma^\poly)^{lk}_{xy}  \,,
\end{split}
\end{equation*}
therefore the last term in \eqref{eq:netoch2} equals
\begin{equation*}
\begin{split}
	\sum_{l \in \poly(\gamma+\beta)} \bbX^l_x \, \hat\Gamma^{li}_{xy} 
	&= \sum_{l \in \poly(\gamma+\beta)} 
	\sum_{j\in I} A^{j,l}_x \, \Gamma^{ji}_{xy}
	\, \bbX^l_x - \sum_{k \in \poly(\gamma+\beta)} 
	A_y^{i,k} \sum_{l \in \poly(\gamma+\beta)} \bbX^l_x \,  (\Gamma^\poly)^{lk}_{xy} \\
	& = \sum_{l \in \poly(\gamma+\beta)} 
	\sum_{j\in I} A^{j,l}_x \, \Gamma^{ji}_{xy}
	\, \bbX^l_x - \sum_{k \in \poly(\gamma+\beta)} 
	A_y^{i,k} \, \bbX^k_y \,,
\end{split}
\end{equation*}
which coincides with the last term in \eqref{eq:reexp} after renaming the sum index $k$ as $l$.

\paragraph{Homogeneity condition}

Let us now prove that each $\hat{\Pi}_x^a$, $a \in \hat{I}$, satisfies the homogeneity relation 
\eqref{eq:homogeneity-Pi} with exponent $\hat{\alpha}_a$.
On the one hand, if $a \in \poly(\gamma+\beta)$, this is straightforward, recall Example~\ref{ex:polynomial-model}.

On the other hand, if $a=i \in {I}$, then since $\Pi_x^{i} \in \mathcal{G}_{\mathrm{hom}}^{\alpha_i}$
and $\sfK$ is regularising of order $(m, r)$ with
$m > \gamma+\beta > \alpha_i + \beta$ and $r > r_\Pi$,
we have $\sfK \Pi^i \in \mathcal{G}_{\mathrm{weak\,hom}}^{\hat{\alpha}_i}$
by  Theorem~\ref{thm:convolution_germs_in_G_checkG}.

Now we can apply Theorem~\ref{thm:existence_pointwise_evaluations} to $\sfK \Pi^i$, noting that when 
$\hat{\alpha}_i = \alpha_i+\beta \in \mathbb{N}_0$, the assumptions 
are satisfied thanks to the compatibility condition, 
see \eqref{eq:techn}.
Thus, we obtain $\hat{\Pi}^i \in \mathcal{G}_{\mathrm{hom}}^{\hat{\alpha}_i}$
(recall the definition \eqref{eq:hatPi} of $\hat{\Pi}_x^i$), i.e.\ the estimate \eqref{eq:homogeneity-Pi} with exponent $\hat{\alpha}_i$
holds for $\hat\Pi^i$, as announced, and moreover
in this relation we can take $\varphi \in \mathscr{B}^r$
with $r = r_\Pi$, by the estimates \eqref{eq:continuity_estimate_embedding_homogeneity}
and \eqref{eq:Kconv-hom}. This means that we can take $r_{\hat{\Pi}} = r_{\Pi}$ 
for the model $(\hat{\Pi}, \hat{\Gamma})$, see Definition~\ref{def:model}.
Furthermore, the bound \eqref{eq:conhatpi} follows from by keeping track of the estimates.

\subsubsection{Proof of relation \texorpdfstring{\eqref{eq:prophatf}}{}}
\label{sec:hatf}

We next show that for any $x\in\R^d$
\begin{equation} \label{eq:check-hat-f}
	\langle \hat f, \hat \Pi \rangle_x = \sum_{a\in \hat{I}} \hat f^a(x) \, \hat\Pi^a_x
	= \left(\cK^{\gamma, \beta}\langle f, \Pi \rangle \right)_x \,,
\end{equation}
where we recall that $\cK^{\gamma, \beta}$ is defined in \eqref{eq:definition_of_KF}.

We obtain by \eqref{eq:hatf}, using the notation \eqref{eq:A},
\begin{equation*}
\begin{split}
	\sum_{a \in \hat{I}} \hat f^a(x) \, \hat\Pi^a_x 
	& = \sum_{i \in {I}} \hat f^i(x) \, \hat\Pi^i_x +
	\sum_{k \in \poly(\gamma+\beta)} \hat f^k(x) \, \bbX^k_x \\
	& = \sum_{i \in {I}} f^i(x) \bigg( 
	\sfK \Pi^i_x - \sum_{k\in\poly(\gamma+\beta)} \, A^{i,k}_x \, \bbX^k_x\bigg)
	\\ & \quad\ + \sum_{k \in \poly(\gamma+\beta)} \bigg(\sum_{j\in I} f^j(x) \, A^{j,k}_x 
	- D^k (\sfK \{\langle f,\Pi\rangle_x - \cR^\gamma\langle f,\Pi\rangle\}) (x) \bigg) \bbX^k_x
	\\
	& = \sfK \langle f, \Pi \rangle_x
	\, - \sum_{k\in\poly(\gamma+\beta)} D^k (\sfK \{\langle f,\Pi\rangle_x 
	- \cR^\gamma\langle f,\Pi\rangle\}) (x)
	\, \bbX^k_x \,,
\end{split}
\end{equation*}
which shows that \eqref{eq:check-hat-f} holds.

		\subsubsection{Proof that \texorpdfstring{$\hat{f}$}{hatf} is a modelled distribution} 
		\label{sec:proof_hat_f_is_modelled_distribution}
	
We finally prove that $\hat{f}$ defined in \eqref{eq:hatf} is a modelled distribution 
of order $\gamma + \beta$ relative to the model $(\hat{\Pi}, \hat{\Gamma})$,
more precisely the continuity estimate \eqref{eq:conhatf} holds.
Recalling 
Definition~\ref{def:modelled_distribution}, see \eqref{eq:definition_modelled_distribution_norm},
we fix 
a compact $K \subset \mathbb{R}^d$ and we show that,
uniformly for $a \in \hat{I}$ and $x,y\in K$, 
\begin{align}\label{eq:quantities1}
		| \hat{f}^a ( x ) | &\lesssim \| \Pi \|_{\mathcal{M}^{\balpha}_{K^{\prime}, 1}} 
		\left\vvvert f \right\vvvert_{\mathcal{D}_{\Gamma, \balpha; K^{\prime}}^{\gamma}} \,, \\
		\label{eq:quantities2}
		\bigg| \sum_{b \in \hat{I}} \hat{\Gamma}_{xy}^{ab} 
		\hat{f}^{b} ( y ) - \hat{f}^a (x) \bigg|
		&\lesssim \big( \| \Pi \|_{\mathcal{M}_{K^{\prime},1}^{\balpha}} + [\sfK\Pi]_{K^{\prime}, 1} \big)
		\left\vvvert f \right\vvvert_{\mathcal{D}_{\Gamma, \balpha; K^{\prime}}^{\gamma}}
		|y-x|^{\gamma-\hat\alpha_a} \,,
\end{align}
where $K^\prime = K \oplus B(0,2)$ and the implicit constants depend on the compact~$K$
(as well as on the kernel~$\sfK$).
We split the proof in two parts.

\paragraph{Proof of \eqref{eq:quantities1}}
We estimate $| \hat{f}^a ( x ) |$.
One the one hand, if $a=i \in {I}$ then, by \eqref{eq:definition_modelled_distribution_norm},
\begin{equation*}
	| \hat{f}^a ( x ) | = | f^i ( x ) | 
	\leq \vvvert f \vvvert_{\mathcal{D}_{\Gamma, \balpha; K}^{\gamma}} \,.
\end{equation*}
On the other hand, if $a = k \in \poly(\gamma+\beta)$, then by \eqref{eq:hatf}
	\begin{align*}
	\hat{f}^a ( x )
	= \hat{f}^k ( x ) & = \sum_{\substack{j \in I \colon \\ \alpha_j + \beta > |k|}} f^j(x) \, 
	D^k (\sfK \Pi^j_x) (x) - D^k G_x (x) .
	\end{align*}
where we have introduced the shorthand
\begin{equation} \label{eq:G}
	G_x \coloneqq \sfK \{\langle f,\Pi\rangle_x - \cR^\gamma \langle f,\Pi\rangle\} \,.
\end{equation}

Since $\sfK$ is regularising of order $(m, r)$ with 
$m > \gamma+\beta > \alpha_j + \beta$ and $r \ge r_\Pi$,
the fact that
$\Pi^{j} \in \mathcal{G}_{\mathrm{hom}}^{\alpha_j}$, for any $j \in I$,
implies that $\sfK \Pi^j \in \mathcal{G}_{\mathrm{weak\,hom}}^{\alpha_j+\beta}$
by Theorem~\ref{thm:convolution_germs_in_G_checkG}
and, recalling \eqref{eq:norm-model}, the estimate \eqref{eq:Kconv-hom} 
yields, for any $\bar\lambda > 0$, 
	\begin{equation} \label{eq:estimate_KPi}
		\| \sfK \Pi^j \|_{\mathcal{G}_{\mathrm{weak\,hom}; 
		K, \bar{\lambda}, r_\Pi}^{\alpha_j+\beta}} \lesssim \mathrm{cst}'_{K, \bar{\lambda}} \,
		\| \Pi^j \|_{\mathcal{G}_{\mathrm{hom}; 
		K, 2 ( \bar{\lambda} +\rho ), r_\Pi}^{\alpha_j}}
		\le \mathrm{cst}'_{K, \bar{\lambda}} \, \| \Pi \|_{\mathcal{M}_{K,
		2 (\bar{\lambda}+\rho)}^{\balpha}} .
	\end{equation}
Applying Theorem~\ref{thm:existence_pointwise_evaluations},
see \eqref{eq:estimate_on_F_x^k(x)} with $\bar\lambda = 1$,
for $|k| < \alpha_j + \beta$ we can bound
	\begin{equation*}
		| D^k (\sfK \Pi^j_x) (x) | \lesssim 
		\mathrm{cst}'_{K,1} \,\| \Pi \|_{\mathcal{M}_{K, 2(1+\rho)}^{\balpha}} .
	\end{equation*}
	
Similarly, since $(\langle f,\Pi\rangle_x - \cR^\gamma \langle f,\Pi\rangle)_{x\in\R^d} \in 
\mathcal{G}_{\mathrm{hom}}^{\gamma}$ by the Reconstruction Theorem,
see \eqref{eq:RT-bounds},
exploiting \eqref{eq:Kconv-hom} from Theorem~\ref{thm:convolution_germs_in_G_checkG}
and \eqref{eq:boundnormF} we get, for any $\bar\lambda > 0$,
\begin{equation} \label{eq:estG}
		\| 
		G \|_{\mathcal{G}_{\mathrm{weak\,hom}; K, \bar\lambda, r_\Pi}^{\gamma+\beta}}
		\lesssim \mathrm{cst}'_{K,\bar\lambda} \, 
		\| \Pi \|_{\mathcal{M}^{\balpha}_{K^{\prime}, 2(\bar\lambda+\rho)}} 
		\left\vvvert f \right\vvvert_{\mathcal{D}_{\Gamma, \balpha; K^{\prime}}^{\gamma}},
\end{equation}
 for $K^{\prime} = K \oplus B(0,\bar\lambda+1) \supset K$.
Then, setting $\bar\lambda = 1$ and applying \eqref{eq:estimate_on_F_x^k(x)}
from Theorem~\ref{thm:existence_pointwise_evaluations}, we obtain with
$K^{\prime} := K \oplus B(0,2 )$,
\begin{equation*}
		| D^k 
		G_x (x) | 
		\lesssim \mathrm{cst}'_{K,1} \, \| \Pi \|_{\mathcal{M}^{\balpha}_{K^{\prime}, 2(1+\rho)}} 
		\left\vvvert f \right\vvvert_{\mathcal{D}_{\Gamma, \balpha; K^{\prime}}^{\gamma}} .
	\end{equation*}
Collecting the estimates above, we have established \eqref{eq:quantities1}.

\paragraph{Proof of \eqref{eq:quantities2}}

We estimate $| \textstyle{\sum_{b \in \hat{I}}} \hat{\Gamma}_{xy}^{ab} \hat{f}^{b} ( y ) - \hat{f}^a (x) |$.
Again we distinguish the cases $a \in {I}$ and $a \in \poly(\gamma+\beta)$.
If $a = i \in {I}$ then by construction, see \eqref{eq:hatGamma0} and \eqref{eq:hatf},
	\begin{equation*}
		\bigg| \sum_{b \in \hat{I}} \hat{\Gamma}_{x y}^{ab} \hat{f}^{b} ( y ) - \hat{f}^a (x) \bigg| 
		= \bigg| \sum_{j \in I} \Gamma_{x y}^{i j} f^{j} ( y ) - f^i (x) \bigg| 
		\leq \vvvert f \vvvert_{\mathcal{D}_{\Gamma, \balpha; K}^{\gamma}}
		| y - x |^{\gamma - \alpha_i} ,
	\end{equation*}
whence the desired estimate \eqref{eq:quantities2},
because $\gamma - \alpha_i = ( \gamma + \beta ) - \hat{\alpha}_i$.

Now fix $a = l \in \poly(\gamma+\beta)$, i.e.\ $l$ denotes some multi-index with 
$| l | < \gamma + \beta$:
by definition \eqref{eq:hatGamma0} of $\hat{\Gamma}$, 
since $\hat{f}^i = f^i$ for $i\in I$, see \eqref{eq:hatf}, we can write
\begin{equation*}
		\sum\limits_{b \in \hat{I}} \hat{\Gamma}_{x y}^{ab} \hat{f}^{b} ( y ) - \hat{f}^a (x) 
		= \sum_{k \in \poly(\gamma+\beta)} 
		(\Gamma^\poly)_{xy}^{lk}
		\hat{f}^k(y) + \sum_{i\in I} \hat{\Gamma}_{xy}^{li} \, f^i(y) - \hat{f}^l(x) \,.
\end{equation*}
Plugging in $\hat{\Gamma}_{xy}^{li}$ from \eqref{eq:hatGamma}
and $\hat{f}$ from \eqref{eq:hatf} yields, after some simplifications,
\begin{equation}\label{eq:deco-hat-Gamma-f}
\begin{split}
		\sum\limits_{b \in \hat{I}} \hat{\Gamma}_{x y}^{ab} \hat{f}^{b} ( y ) - \hat{f}^a (x) 
	& = \sum_{\substack{j\in I \colon\\ \alpha_j+\beta > |l|}}  
	D^l (\sfK \Pi^j_x)(x) \bigg( \sum_{i\in I} \Gamma^{ji}_{xy} \, f^i(y) - f^j(x) \bigg) \\
	& \quad + D^l G_x(x) - \sum_{k \in \poly(\gamma+\beta)} 
		(\Gamma^\poly)_{xy}^{lk} \, D^k G_y(y) \,,
\end{split}
\end{equation}
where we recall that $G_x$ was defined in \eqref{eq:G}.

We now replace pointwise derivatives in in the RHS of \eqref{eq:deco-hat-Gamma-f} 
by the multi-scale formula \eqref{eq:F_x_as_sum_2}:
if we fix $\eta \in \mathcal{D} ( B (0, 1) )$ such that $\int \eta ( x ) dx = 1$ and 
$\int \eta ( x ) x^l \, d x = 0$ for $1 \leq| l | \le \gamma + \beta$, and define
$\varphi \coloneqq \eta^{\frac{1}{2}} - \eta$
as in \eqref{eq:specialphi}, then by \eqref{eq:F_x_as_sum_2} with $\lambda=\bar\lambda > 0$
\begin{alignat}{2}
	\label{eq:D1}
		D^l (\sfK \Pi^j_x) (x) & = (- \bar\lambda^{-1})^{| l |} \sfK \Pi^j_x ( (\partial^l 
		\eta)_x^{\bar\lambda} ) 
		+ \sum\limits_{n = 0}^{+ \infty} (- \bar\lambda^{-1}\, 2^{n})^{| l |} \sfK \Pi^j_x 
		( (\partial^l \varphi)_{x}^{\bar\lambda2^{- n}} ) \\
	\nonumber
		D^l G_x (x) & = (- \bar\lambda^{-1})^{| l |} G_x^{} ( (\partial^l \eta)_x^{\bar\lambda} ) 
		+ \sum\limits_{n = 0}^{+ \infty} (- \bar\lambda^{-1}\,
		2^{n})^{| l |} G_x ( (\partial^l \varphi)_{x}^{\bar\lambda 2^{- n}} ) ,
\end{alignat}
and similarly for $D^k G_y(y)$. Plugging these expressions into \eqref{eq:deco-hat-Gamma-f},
we can write
	\begin{equation} \label{eq:deco-hat-Gamma-f2}
		\sum\limits_{b \in \hat{I}} \hat{\Gamma}_{x y}^{ab} \hat{f}^{b} ( y ) - \hat{f}^a (x) 
		= \Delta_{xy}^{l; 0} (\eta^{}) + \sum\limits_{n = 0}^{+ \infty} \Delta_{xy}^{l; n} (\varphi) ,
	\end{equation}
where we recall that $a=l \in \poly(\gamma+\beta)$
and we set, for a test-function $\psi$ and $n \in \mathbb{N}_0$,
\begin{align*}
		\Delta_{xy}^{l; n} ( \psi ) & \coloneqq 
	\sum_{\substack{j\in I \colon\\ \alpha_j+\beta > |l|}}  
	(- \bar\lambda^{-1} \, 2^{n})^{| l |}  \, \sfK \Pi_x^j ( (\partial^l \psi)_{x}^{\bar\lambda 2^{- n}} ) 
	\bigg( \sum_{i\in I} \Gamma^{ji}_{xy} \, f^i(y) - f^j(x) \bigg) \\
			& + (- \bar\lambda^{-1} \, 2^{n})^{| l |} \, G_x ( (\partial^l \psi)_{x}^{\bar\lambda 2^{- n}} ) 
			- \!\!\! \sum_{k \in \poly(\gamma+\beta)} 
		(\Gamma^\poly)_{xy}^{lk} \, (- \bar\lambda^{-1} \, 2^{n})^{| k |}
		\, G_y ( (\partial^k \psi)_{y}^{\bar\lambda 2^{- n}} ) . 
	\end{align*}
It remains to show that the RHS of \eqref{eq:deco-hat-Gamma-f2}
satisfies the bound in the RHS of \eqref{eq:quantities2}.

\smallskip

We already observed after \eqref{eq:specialphi} that 
$\partial^l \varphi \in \mathscr{B}_{\gamma + \beta}$ 
for all $l\in\N_0^d$, in view of \eqref{eq:eta-polynomials}.
Then, for $r = r_\Pi$,
we have $\partial^l \varphi \in \mathrm{cst} \, \mathscr{B}_{\gamma + \beta}^{r}$ 
for all $| l | \le \gamma + \beta$, for a suitable $\mathrm{cst} > 0$.
Since $\sfK \Pi^j \in \mathcal{G}_{\mathrm{weak\,hom}}^{\alpha_j+\beta}$
and $G \in \mathcal{G}_{\mathrm{weak\,hom}}^{\gamma + \beta}$, 
by \eqref{eq:estimate_KPi} and \eqref{eq:estG} with $\bar\lambda = \frac{1}{2} (1+\rho)^{-1}$
we can estimate, uniformly for $x\in K$ and $n\in\N_0$,
\begin{align}
	\label{eq:estimate_KPi2}
	|\sfK \Pi_x^j ( (\partial^l \varphi)_{x}^{\bar\lambda 2^{- n}} ) | 
	& \lesssim 
	\| \Pi \|_{\mathcal{M}_{K, 1}^{\balpha}}
	\, 2^{-(\alpha_j + \beta)n} \,, \\
	\label{eq:estG2}
	|G_x ( (\partial^k \varphi)_{x}^{\bar\lambda 2^{- n}} ) | 
	& \lesssim
	\| \Pi \|_{\mathcal{M}^{\balpha}_{K^{\prime}, 1}} 
	\left\vvvert f \right\vvvert_{\mathcal{D}_{\Gamma, \balpha; K^{\prime}}^{\gamma}}
	\, 2^{-(\gamma + \beta)n} \,,
\end{align}
where $K' = K \oplus B(0,2)$ and the implicit constants depend on the compact $K$
and on the kernel $\sfK$. 
We can thus bound $\Delta_{xy}^{l; n} (\varphi)$:
recalling \eqref{eq:definition_modelled_distribution_norm} and \eqref{eq:alphaGammapoly} we have
\begin{equation}\label{eq:estmoddist}
	\bigg|\sum_{i\in I} \Gamma^{ji}_{xy} \, f^i(y) - f^j(x) \bigg| \le
	\vvvert f \vvvert_{\mathcal{D}_{\Gamma, \balpha; K}^{\gamma}} |y-x|^{\gamma-\alpha_j} ,
	\quad
	|(\Gamma^\poly)_{xy}^{lk}| \le |x-y|^{|k|-|l|} \ind_{\{l \le k\}} \,,
\end{equation}
therefore by \eqref{eq:estimate_KPi2} and \eqref{eq:estG2} we get
\begin{equation}\label{eq:estgeome}
\begin{split}
		| \Delta_{x y}^{l; n} ( \varphi ) | & \lesssim 
		\Bigg( \sum_{\substack{j\in I \colon\\ \alpha_j+\beta > |l|}}   2^{- n (\alpha_j+\beta-|l|)} 
		| y - x |^{\gamma - \alpha_j} \\
		& \quad \quad \quad + 
		\sum\limits_{\substack{k \in \mathbb{N}_0^d \colon \\ 
		k \ge l, \, | k | < \gamma + \beta}} 
		| x - y |^{|k| - |l|} \, 2^{- n (\gamma + \beta - |k|)} \Bigg) 
		\| \Pi \|_{\mathcal{M}_{K^{\prime}}^{\balpha}} 
		\vvvert f \vvvert_{\mathcal{D}_{\Gamma, \balpha; K^{\prime}}^{\gamma}}  \,.
\end{split}
\end{equation}

We now estimate the tail of the sum in
the RHS of \eqref{eq:deco-hat-Gamma-f2}: if we set
\begin{equation} \label{eq:Nxy}
	N_{x y} \coloneq \min \lbrace n \in \mathbb{N}_0 \colon 2^{- n} \leq | y - x | \rbrace \,,
\end{equation}
summing the geometric series we get from \eqref{eq:estgeome}, since $2^{-N_{xy}} \le |y-x|$,
\begin{equation*}
		\bigg| \sum\limits_{n = N_{x y}}^{+ \infty} \Delta_{x y}^{l; n} (\varphi) 
		\bigg| \lesssim \| \Pi \|_{\mathcal{M}_{K^{\prime}}^{\balpha}} 
		\vvvert f \vvvert_{\mathcal{D}_{\Gamma, \balpha; K^{\prime}}^{\gamma}} 
		\, | y - x |^{\gamma + \beta - |l|} ,
\end{equation*}
which agrees with the RHS of \eqref{eq:quantities2} 
since $\hat\alpha_a = |l|$
for $a=l \in \poly(\gamma+\beta)$, 
see \eqref{eq:hatalpha}.

We finally bound the contribution of $\Delta_{xy}^{l; n} (\varphi)$ for $n \le N_{xy}$
and of $\Delta_{xy}^{l; 0} (\eta)$ in the RHS of \eqref{eq:deco-hat-Gamma-f2}.
Observe that by the reexpansion property \eqref{eq:relation_Pi_Gamma} we can write
	\begin{align*}
		\sum\limits_{j \in I } \bigg( 
		\sum_{i\in I} \Gamma^{ji}_{xy} \, f^i(y) - f^j(x) \bigg)  
		\sfK \Pi_x^j & = \sfK \langle f,\Pi\rangle_y - \sfK \langle f,\Pi\rangle_y = G_y - G_x ,
	\end{align*}
where we recall that $G_x$ is defined in \eqref{eq:G}. Then we can rewrite
\begin{align*}
		\Delta_{xy}^{l; n} ( \varphi ) & = 
	\, - \sum_{\substack{j\in I \colon\\ \alpha_j+\beta \le |l|}}  
	(- \bar\lambda^{-1} \, 2^{n})^{| l |}  \, \sfK \Pi_x^j ( (\partial^l \varphi)_{x}^{\bar\lambda 2^{- n}} ) 
	\bigg( \sum_{i\in I} \Gamma^{ji}_{xy} \, f^i(y) - f^j(x) \bigg) \\
			& + (- \bar\lambda^{-1} \, 2^{n})^{| l |} \, G_y 
			\bigg(  (\partial^l \varphi)_{x}^{\bar\lambda 2^{- n}} 
			- \!\!\! \sum_{k \in \poly(\gamma+\beta)} 
		(\Gamma^\poly)_{xy}^{lk} \, (- \bar\lambda^{-1} \, 2^{n})^{| k |-|l|}
		\, (\partial^k \varphi)_{y}^{\bar\lambda 2^{- n}}  \bigg) . 
	\end{align*}

Let us single out the contribution of those $j \in I$ (if any) such that $\alpha_j + \beta = | l |$,
which we bound separately using the \Cref{ass:compatibility} of compatibility: denote
	\begin{align*}
		\Delta_{=; xy}^{l; n} ( \varphi ) & \coloneqq - \sum_{\substack{j\in I \colon\\ \alpha_j+\beta = |l|}} (- \bar\lambda^{-1} \, 2^{n})^{| l |}  \, \sfK \Pi_x^j ( (\partial^l \varphi)_{x}^{\bar\lambda 2^{- n}} ) \bigg( \sum_{i\in I} \Gamma^{ji}_{xy} \, f^i(y) - f^j(x) \bigg) , \\
		\Delta_{<; xy}^{l; n} ( \varphi ) & \coloneqq \Delta_{xy}^{l; n} ( \varphi ) - \Delta_{=;xy}^{l; n} ( \varphi ) ,
	\end{align*}
then 
	\begin{align*}
		& \Delta_{=;xy}^{l; 0} ( \eta ) + \sum\limits_{n = 0}^{N_{x, y}} \Delta_{=;x y}^{l; n} (\varphi) \\
		 & \quad = - \bigg( \sum_{i\in I} \Gamma^{ji}_{xy} \, f^i(y) - f^j(x) \bigg) \sum_{\substack{j\in I \colon\\ \alpha_j+\beta = |l|}} D^l \sfK \Pi_x^j \Big( \eta_x^{\bar{\lambda}} + \sum\limits_{n = 0}^{N_{x, y}} \varphi_x^{\bar{\lambda} 2^{- n}} \Big) \\
		& \quad = - \bigg( \sum_{i\in I} \Gamma^{ji}_{xy} \, f^i(y) - f^j(x) \bigg) \sum_{\substack{j\in I \colon\\ \alpha_j+\beta = |l|}} D^l \sfK \Pi_x^j \Big( \eta_x^{\bar{\lambda} 2^{- (N_{x, y} + 1)}} \Big) ,
	\end{align*}
which is bounded by the fact that $f$ is a modelled distribution, 
the assumption of compatibility, see \eqref{eq:technorm},
and the observation that $\gamma - \alpha_j = \gamma + \beta - |l|$, by
	\begin{equation*}
		\bigg| \Delta_{=;xy}^{l; 0} ( \eta ) + \sum\limits_{n = 0}^{N_{x, y}}  
		\Delta_{=;x y}^{l; n} (\varphi) \bigg| \lesssim [\sfK\Pi]_{K^{\prime}, 1}
		\vvvert f \vvvert_{\mathcal{D}_{\Gamma, \balpha; K^{\prime}}^{\gamma}} 
		| y - x |^{\gamma + \beta - |l|} .
	\end{equation*}

Now we bound $\Delta_{<}$: note that for $n \le N_{xy}$
we we have $|x-y| \le 2^{-n}$, hence we can
apply Lemma~\ref{lemma:test-functions_Taylor_expansions} below 
(with $c = \lceil \gamma+\beta \rceil^+ - 1$): we can thus write
\begin{equation} \label{eq:phipsi}
\begin{split}
	& (\partial^l \varphi)_{x}^{\bar\lambda 2^{- n}} 
			- \!\!\! \sum_{k \in \poly(\gamma+\beta)} 
		(\Gamma^\poly)_{xy}^{lk} \, (- \bar\lambda^{-1} \, 2^{n})^{| k |-|l|}
		\, (\partial^k \varphi)_{y}^{\bar\lambda 2^{- n}}\\
	& \ = ( \bar\lambda^{-1} \, 2^{n} (x-y))^{\lceil\gamma+\beta\rceil^+} \,
	\psi_y^{2\bar\lambda 2^{-n}}
\end{split}
\end{equation}
for a suitable $\psi = \psi^{[x,y,\bar\lambda 2^{-n}]} \in \mathrm{cst}\, \mathscr{B}^r_{\gamma+\beta}$
with $r=r_\Pi$. Recalling the property of homogeneity of $G_y$, see \eqref{eq:estG}, 
as well as \eqref{eq:estimate_KPi2} and \eqref{eq:estmoddist}, we obtain
\begin{align*}
	| \Delta_{<;x y}^{l; n} ( \varphi ) | 
	& \lesssim \bigg( \sum_{\substack{j\in I \colon\\ \alpha_j+\beta < |l|}}   
	2^{n ( | l | - \alpha_j - \beta )} \, | y - x |^{\gamma - \alpha_j} \\
	& \quad \quad \quad + 2^{n ( | l | - (\gamma + \beta) + \lceil \gamma 
	+ \beta \rceil^+) } 
	| y - x |^{\lceil \gamma  + \beta \rceil^+} \bigg) \| \Pi \|_{\mathcal{M}_{K^{\prime},1}^{\balpha}} 
	\vvvert f \vvvert_{\mathcal{D}_{\Gamma, \balpha; K^{\prime}}^{\gamma}} ,
\end{align*}
and, with similar arguments, the same bound with $n=0$ also applies to $\Delta_{<;x y}^{l; 0} (\eta^{})$.

Since $\gamma+\beta \not\in \N_0$, we have 
$| l | - (\gamma + \beta) + \lceil \gamma + \beta \rceil^+ > 0$ and we can
sum the geometric series to obtain, since $2^{N_{xy}} \le 2 |y-x|^{-1}$,
	\begin{equation*}
		\Big| \Delta_{x y}^{l; 0} (\eta^{}) + \sum\limits_{n = 0}^{N_{x, y}} \Delta_{x y}^{l; n} (\varphi) 
		\Big| \lesssim \big( \| \Pi \|_{\mathcal{M}_{K^{\prime},1}^{\balpha}}  + [\sfK\Pi]_{K^{\prime}, 1} \big)
		\vvvert f \vvvert_{\mathcal{D}_{\Gamma, \balpha; K^{\prime}}^{\gamma}} 
		| y - x |^{\gamma + \beta - |l|} ,
	\end{equation*}
which agrees with the RHS of \eqref{eq:quantities2},
since $\hat\alpha_a = |l|$
for $a=l \in \poly(\gamma+\beta)$.
This concludes the proof that $\hat{f}$ is a modelled distribution 
with the estimate \eqref{eq:conhatf}.\qed

\begin{lemma}[Taylor remainder]\label{lemma:test-functions_Taylor_expansions}
Given $r \in \mathbb{N}_0$, $c \in \mathbb{N}_0 \cup \lbrace -1 \rbrace$.

\noindent Then there exists a constant $\mathrm{cst} > 0$ depending only on $r$, $c$, $d$, such that for all test-functions $\varphi \in \mathscr{B}_{c}^{r + c + 1}$ and $x, y \in \mathbb{R}^d$, $n \in \mathbb{N}_0$ with $| y - x | \leq 2^{- n}$, there exists a test-function
	\begin{equation*}
		\psi = \psi^{[x, y, n]} \in \mathrm{cst} \mathscr{B}_{c}^{r} ,
	\end{equation*}

\noindent such that for such $x, y, n$, 
	\begin{equation*}
		\varphi_x^{2^{-n}} - \sum\limits_{| k | \leq c} \frac{(x - y)^k}{k !} (- 2^n)^{| k |} (\partial^k \varphi)_{y}^{2^{-n}} = (2^n (x - y))^{c + 1} (\psi^{[x, y, n]})_y^{2^{- n + 1}} .
	\end{equation*}
\end{lemma} 

\begin{remark}
Note that the scale $2^{- n + 1}$ may be greater than 1 for $n = 0$.
\end{remark}

\begin{proof}
The test-function $\psi^{[x, y, n]}$ is defined by:
	\begin{equation*}
		\psi^{[x, y, n]} ( \cdot ) \coloneqq 2^{d} (2^n(x - y))^{- c - 1} \left( \varphi (2 \cdot + 2^n (y - x)) - \sum\limits_{| k | \leq c} \frac{(2^n ( y - x ))^k}{k !} \partial^k \varphi ( 2 \cdot ) \right) .
	\end{equation*}
The required properties on $\psi$ follow from this expression, in particular after applying Taylor-Lagrange's formula.
\end{proof}

\subsection{Proof of Proposition~\ref{prop:stability_of_properties_of_Gamma}}
\label{sec:reexpansion}
We first look at the \emph{group property} \eqref{it:uno} 
from Remark~\ref{remark:comparison_definitions_of_model}.
It follows by direct computation from
the definition \eqref{eq:hatGamma0}-\eqref{eq:hatGamma} of 
$\hat{\Gamma}$ that, using labels $i,j \in {I}$
and $k,l \in \poly(\gamma+\beta)$ for clarity,
\begin{equation*}
	\hat \Gamma_{xy} \, \hat \Gamma_{yz} =
	\left( \begin{matrix}
	(\hat \Gamma_{xy} \, \hat \Gamma_{yz})^{ji} & (\hat \Gamma_{xy} \, \hat \Gamma_{yz})^{jk} \\
	(\hat \Gamma_{xy} \, \hat \Gamma_{yz})^{li}
	& (\hat \Gamma_{xy} \, \hat \Gamma_{yz})^{lk}
	\end{matrix}\right)
	=
	\left( \begin{matrix}
	(\Gamma_{xy} \, \Gamma_{yz})^{ji} & 0 \\
	\star
	& (\Gamma^\poly_{xy}\, \Gamma^\poly_{yz})^{lk}
	\end{matrix}\right)
\end{equation*}
where, after some cancelation, we obtain
\begin{equation*}
	\star
	= \sum_{\substack{j\in I \colon\\ \alpha_j+\beta > |l|}}  
	D^l (\sfK \Pi^j_x)(x) \, (\Gamma_{xy} \, \Gamma_{yz})^{ji}
	- 
	\sum_{k\in\poly(\alpha_i+\beta)} 
	(\Gamma^\poly_{xy}\, \Gamma^\poly_{yz})^{lk}
	\, D^k (\sfK \Pi^i_z)(z) \,.
\end{equation*}
Since the group property holds for $\Gamma^\poly$, that is
$\Gamma^\poly_{xy}\, \Gamma^\poly_{yz} = \Gamma^\poly_{xz}$
(see Remark~\ref{ex:polynomial-model}),
it follows that the group property holds for $\hat\Gamma$
as soon as it holds for $\Gamma$.

\smallskip

We next consider the \emph{triangular structure}~\eqref{it:due}
(which we know to hold for $\Gamma^\poly$).
Assuming that it is satisfied by $\Gamma$, that is
$\Gamma_{xy}^{ii}=1$ and
$\Gamma_{xy}^{ji}=0$ for $j\ne i$ with $\hat\alpha_j \ge \hat\alpha_i$,
let us prove that it is satisfied by $\hat\Gamma$.
By \eqref{eq:hatGamma0}, we only need to check that
\begin{equation*}
	\hat{\Gamma}^{li} = 0 \quad \text{for all} \quad i \in I, \ l\in \poly(\gamma+\beta)
	\quad \text{with} \quad \hat\alpha_l = |l| \ge \hat\alpha_i = \alpha_i+\beta \,.
\end{equation*}
It suffices to note that, in the definition \eqref{eq:hatGamma} of $\hat{\Gamma}^{li}$,
\emph{both sums vanish for $|l| \ge \alpha_i+\beta$}: indeed, the first sum
is restricted to $\alpha_j+\beta > |l|$, hence $\alpha_j > \alpha_i$ and
then $\Gamma_{xy}^{ji}=0$ by the triangular structure of $\Gamma$;
similarly, the second sum is restricted to
$|k| < \alpha_i+\beta$, hence $|k| < |l|$ and then 
$(\Gamma^\poly)_{xy}^{lk}=0$ by the triangular structure of $\Gamma^\poly$.

\smallskip

We finally focus on the \emph{analytic bound}~\eqref{it:tre}, that we
assume to hold for $\Gamma$, that is
$| \Gamma_{x y}^{ji} | \lesssim | y - x |^{\alpha_i - \alpha_j}$.
By definition \eqref{eq:hatGamma0}-\eqref{eq:hatGamma} of $\hat{\Gamma}$, the
corresponding bound $| \hat{\Gamma}_{x y}^{ba} | \lesssim |y - x|^{\hat{\alpha}_a - \hat{\alpha}_b}$
is immediate to check except when $b = l \in \poly(\gamma+\beta)$
and $a = i \in {I}$, which is the case we tackle now: we need to show that
\begin{equation}\label{eq:analytic-b}
	\forall i \in I, \ \forall l \in \poly(\gamma+\beta): \qquad
	| \hat{\Gamma}_{x y}^{li} | \lesssim |y - x|^{\alpha_i+\beta - |l|} \,,
\end{equation}
uniformly for $x,y$ in compact sets, where we recall that 
$\hat{\Gamma}_{x y}^{li}$ is defined in \eqref{eq:hatGamma}.

We argue as in the proof that $\hat{f}$ is a modelled distribution, see
Section~\ref{sec:proof_hat_f_is_modelled_distribution}: replacing the pointwise
derivatives in \eqref{eq:hatGamma} by formula \eqref{eq:D1}, we can write
\begin{equation} \label{eq:hatGammaalt}
	\hat{\Gamma}_{x y}^{li} = \hat{\Gamma}_{xy}^{li; 0} (\eta^{}) 
	+ \sum\limits_{n = 0}^{+ \infty} \hat{\Gamma}_{x y}^{li; n} (\varphi) ,
	\end{equation}
where for a test-function $\psi \in \mathcal{D}$ and $n \in \mathbb{N}_0$ we set
\begin{align*}
		\hat{\Gamma}_{x y}^{li; n} ( \psi ) 
		& \coloneqq \sum_{\substack{j \in I \colon \\ \alpha_j+\beta > |l|}} 
		(-\bar\lambda^{-1} 2^n)^{| l |} \,
		(\sfK \Pi^j_x) ( (\partial^l \psi)_x^{\bar\lambda 2^{- n}} ) \, \Gamma^{ji}_{xy} \\
		& \quad \quad - \sum_{k\in \poly( \alpha_i+\beta)} 
		(-\bar\lambda^{-1} 2^n)^{| k |} \,
		(\Gamma^\poly)^{l k}_{xy}  \, (\sfK \Pi^i_y)( (\partial^k \psi)_y^{\bar\lambda 2^{- n}} ) \,.
\end{align*}
It follows by the property of homogeneity \eqref{eq:estimate_KPi2} 
and the analytic bound on $\Gamma$ that
\begin{align*}
	| \hat{\Gamma}_{x y}^{li; n} ( \varphi )  | & \lesssim 
	\sum_{\substack{j \in I \colon \\ \alpha_j+\beta > |l|}} 
	2^{- n ( \alpha_j+\beta - | l |)} | y - x |^{\alpha_i - \alpha_j} + 
	\sum_{\substack{k\in \poly( \alpha_i+\beta)\colon \\ k \ge l}}  2^{-n (\alpha_i+\beta - | k |) } 
	| y - x |^{| k | - | l |} .
\end{align*}
We can then bound the tail of the sum in
\eqref{eq:hatGammaalt}: recalling $N_{xy}$ from \eqref{eq:Nxy}, we have
\begin{equation*}
	\bigg| \sum\limits_{n = N_{x, y}}^{+ \infty} \hat{\Gamma}_{x y}^{li; n} ( \varphi )
	\bigg| \lesssim | y - x |^{\alpha_i+\beta - | l |} \,,
\end{equation*}
which agrees with \eqref{eq:analytic-b}.
	
On the other hand, since $\sum_{j\in I} \sfK \Pi^j_x \, \Gamma^{ji}_{xy} = \sfK \Pi^j_y$
by the reexpansion property \eqref{eq:relation_Pi_Gamma}, 
we can rewrite, recalling \eqref{eq:phipsi},
\begin{align*}
		\hat{\Gamma}_{x y}^{li; n} ( \varphi ) 
		& = - \sum_{\substack{j \in I \colon \\ \alpha_j+\beta \le |l|}} 
		(-\bar\lambda^{-1} 2^n)^{| l |} \,
		(\sfK \Pi^j_x) ( (\partial^l \varphi)_x^{\bar\lambda 2^{- n}} ) \, \Gamma^{ji}_{xy} \\
		& \quad\quad  + (-\bar\lambda^{-1} 2^n)^{| l |} \,
		( \bar\lambda^{-1} \, 2^{n} (x-y))^{\lceil\gamma+\beta\rceil^+}  \,
		\sfK \Pi^i_y ( \psi_y^{2\bar\lambda 2^{-n}} ) \,,
\end{align*}
for a suitable $\psi = \psi^{[x,y,\bar\lambda 2^{-n}]} \in \mathrm{cst}\, \mathscr{B}^r_{\gamma+\beta}$
with $r=r_\Pi$, thanks to Lemma~\ref{lemma:test-functions_Taylor_expansions}.

We again single out the indices $j \in I$ such that $\alpha_j + \beta = |l|$ which we tackle using the assumption of compatibility: denote
	\begin{align*}
		\hat{\Gamma}_{=; x y}^{li; n} ( \varphi ) 
		& = - \sum_{\substack{j \in I \colon \\ \alpha_j+\beta = |l|}} 
		(-\bar\lambda^{-1} 2^n)^{| l |} \,
		(\sfK \Pi^j_x) ( (\partial^l \varphi)_x^{\bar\lambda 2^{- n}} ) \, \Gamma^{ji}_{xy} , \\
		\hat{\Gamma}_{<; x y}^{li; n} ( \varphi ) & \coloneqq \hat{\Gamma}_{x y}^{li; n} ( \varphi ) - \hat{\Gamma}_{=; x y}^{li; n} ( \varphi ) ,
	\end{align*}
so that arguing as in the proof of \ref{thm:multi_level_schauder_estimate}, 
	\begin{align*}
		\Big| \hat{\Gamma}_{=; x y}^{li; 0} ( \eta ) + \sum\limits_{n = 0}^{N_{x, y}} \hat{\Gamma}_{=; x y}^{li; n} ( \varphi ) \Big| &= \Big|\sum_{\substack{j \in I \colon \\ \alpha_j+\beta = |l|}}  \Gamma^{ji}_{xy} \, D^l (\sfK \Pi^j_x) ( \varphi_x^{\bar\lambda 2^{- n}} ) \Big| \\
		& \lesssim \sum_{\substack{j \in I \colon \\ \alpha_j+\beta = |l|}} |y - x |^{\alpha_i - \alpha_j} \lesssim |y - x |^{\alpha_i - | l | + \beta} ,
	\end{align*}

\noindent while for the other terms, recalling the property of homogeneity \eqref{eq:estimate_KPi}, we obtain
\begin{equation*}
	|\hat{\Gamma}_{<; x y}^{li; n} ( \varphi ) | \lesssim 
	\sum_{\substack{j \in I \colon \\ \alpha_j+\beta < |l|}} 
	(2^n)^{|l|-\alpha_j-\beta} \, |y-x|^{\alpha_i-\alpha_j} +
	(2^n)^{|l|+\lceil\gamma+\beta\rceil^+ - \alpha_i - \beta} \,
	|y-x|^{\lceil\gamma+\beta\rceil^+} \,,
\end{equation*}
and the same estimate with $n=0$ also applies to $\hat{\Gamma}_{x y}^{li; n} ( \eta )$.
Since $\alpha_i + \beta < \gamma+\beta \le \lceil \gamma+\beta\rceil^+$, 
a geometric sum yields
\begin{equation*}
	\bigg| \hat{\Gamma}_{x y}^{li; 0} ( \eta ) + 
	\sum\limits_{n = 0}^{N_{x, y}} \hat{\Gamma}_{x y}^{li; n} ( \varphi ) \bigg| 
	\lesssim | y - x |^{\alpha_i + \beta - | l |}  \,,
\end{equation*}
which completes the proof of \eqref{eq:analytic-b}
and of the whole Proposition~\ref{prop:stability_of_properties_of_Gamma}.
The continuity bound \eqref{eq:conhatgamma} follows from keeping track 
of the constants in the estimates above.
\qed

\appendix

\section{Technical proofs} \label{section:technical_proofs}

\subsection{Proof of Lemma~\ref{lemma:test_functions_K_n}}

We proceed as in \cite[Proposition~14.11]{MR4174393}.

\label{sec:test_functions_K_n}

We fix a $\beta$-regularising kernel $\sfK$ of order $(m,r)$ which 
preserves polynomial at level $c_0 \in \mathbb{N}_0 \cup \lbrace - 1 \rbrace$, see
Assumption~\ref{assumption:preserving_polynomial_annihilation} 
(when $c_0 = -1$ this imposes no extra assumption).
We also fix a test function $\varphi \in \mathscr{B}_{c}^{r}$,
for some $c\in\N_0 \cup \{-1\}$, and
we assume without loss of generality that $c+1 \ge m$
(we can just redefine $m$ as $\min\{c+1,m\}$).
Recalling \eqref{eq:scaling} and \eqref{eq:K_n*}, 
we can express $\sfK_n^*\varphi_x^\lambda(y)$ 
as in \eqref{eq:scaled_recentered_eta_zeta} provided we define
	\begin{equation}\label{eq:definition_eta_zeta}
	\begin{split}
	\eta(y) = \eta^{[n, \lambda, x, \varphi]} ( y ) & \coloneqq 2^{\beta n} (2 \lambda)^d 
	\int\limits_{\mathbb{R}^d} \varphi ( z ) \sfK_n (x + \lambda z, x + 2 \lambda y) d z  , \\
	\zeta(y) = \zeta^{[n, \lambda, x, \varphi]} ( y ) & \coloneqq 2^{\beta n} (2^n \lambda)^{- m} 
	(2 \rho 2^{- n})^{d} \int\limits_{\mathbb{R}^d} \varphi ( z ) 
	\sfK_n (x + \lambda z, x + 2\rho 2^{- n} y) d z  \,,
	\end{split}
	\end{equation}
hence it only remains to prove \eqref{eq:eta_zeta_are_in_Bcr_when_K_preserves_polynomials}
(which reduces to \eqref{eq:eta_zeta_are_in_Bcr} when $c_0 = -1$).

From Assumption~\ref{assumption:preserving_polynomial_annihilation}, see \eqref{eq:polypres}, 
we see that $\eta$ and $\zeta$ annihilate polynomials at level $\min (c, c_0)$.
It remains to control the support and the $C^{r}$ norm of $\eta$ resp.\ $\zeta$.

\paragraph{Support of $\eta$}
Let $y \in \mathbb{R}^d$ be such that $\eta( y ) \neq 0$.
By \eqref{eq:definition_eta_zeta} there is 
$z \in \supp ( \varphi ) \subset B (0, 1)$ such that $\sfK_n (x + \lambda z, x + 2 \lambda y) \neq 0$,
hence $\lambda | 2 y - z | \leq \rho
2^{- n}$ by property \eqref{item:regularising_3} of Definition~\ref{def:regularising_kernel}.
Since we are in the regime $\rho 2^{- n} \leq \lambda$, this implies that 
$| 2 y - z | \leq 1$ and thus by triangle inequality 
$2 | y | \leq |z| + 1 \leq 2$, that is $|y| \le 1$.
This shows that $\supp ( \eta) \subset B (0, 1)$ as wanted.

\paragraph{Support of $\zeta$}
Let $y \in \mathbb{R}^d$ be such that $\zeta ( y ) \neq 0$.
By \eqref{eq:definition_eta_zeta} there is $z \in \supp ( \varphi ) \subset B (0, 1)$ 
such that $\sfK_n (x + \lambda z, x + 2\rho 2^{- n} y) \neq 0$,
therefore $| 2\rho 2^{- n} y - \lambda z | \leq \rho 2^{- n}$
by property \eqref{item:regularising_3} of Definition~\ref{def:regularising_kernel}.
Then $2 | y | \leq 1 + \frac{\lambda}{\rho 2^{- n}} |z|$  by the triangle inequality,
and since we consider $\lambda \leq \rho 2^{- n}$, we obtain $|y| \le 1$,
that is $\supp ( \zeta) \subset B (0, 1)$.

\paragraph{Bound on $C^{r}$ norm of $\eta$}

Let $l \in \mathbb{N}_0^d$ be a multi-index with $| l | \leq r$, and $y \in \mathbb{R}^d$, then by differentiation under the integral,
	\begin{equation*}
		\partial^l \eta ( y ) = 2^{\beta n} (2 \lambda)^{d +|l|} 
		\int_{\mathbb{R}^d} \varphi ( z ) \partial_2^l 
		\sfK_n (x + \lambda z, x + 2 \lambda y) d z .
	\end{equation*}
In this integral, we subtract and add the Taylor polynomial of $\varphi$ at $2 y$ 
of order $|l| - 1$:
	\begin{align*}
		\partial^l \eta ( y ) = 
		& \ 2^{\beta n} (2 \lambda)^{d +|l|} \int_{\mathbb{R}^d} \bigg( \varphi ( z ) - 
		\sum\limits_{| k | \leq |l | - 1} \frac{\partial^k \varphi ( 2 y )}{k !} (z - 2 y)^k \bigg) 
		\partial_2^l \sfK_n (x + \lambda z, x + 2 \lambda y) d z \\
		& \quad + 2^{\beta n} (2 \lambda)^{d +|l|} \sum\limits_{| k | \leq |l | - 1} 
		\frac{\partial^k \varphi ( 2 y )}{k !} \int_{\mathbb{R}^d} (z - 2 y)^k \partial_2^l 
		\sfK_n (x + \lambda z, x + \bar{\lambda }\lambda y) d z .
	\end{align*}
Using Taylor-Lagrange's inequality in the first integral (and absorbing $2^{d +|l|}$ into the implicit constant), we obtain
	\begin{align*}
		| \partial^l \eta ( y ) |
		\lesssim & \ 2^{\beta n} \lambda^{d +|l|} \sum\limits_{| k | = | l |} \frac{1}{k !} 
		\| \varphi \|_{C^{|l|}} \int_{\mathbb{R}^d} | z - 2 y|^{|l|} | \partial_2^l 
		\sfK_n (x + \lambda z, x + 2 \lambda y) | d z \\
		& \quad  + 2^{\beta n} \lambda^{d +|l|} \sum\limits_{| k | \leq |l | - 1} 
		\frac{\| \varphi \|_{C^{|l|}}}{k !} \left| \int_{\mathbb{R}^d} (z - 2 y)^k \partial_2^l 
		\sfK_n (x + \lambda z, x + 2 \lambda y) d z \right| .
	\end{align*}
The first integral can be estimated by the property \eqref{eq:regularising_kernel_bound_on_derivatives} 
of the kernel $\sfK$. For the second integral, we first rewrite it, by a change of variables,  as
\begin{equation*}
	\int_{\mathbb{R}^d} (z - 2 y)^k \partial_2^l 
	\sfK_n (x + \lambda z, x + 2 \lambda y) d z = \lambda^{- | k | - d} 
	\int_{\mathbb{R}^d} (\tilde{z} - 2 \lambda y)^k \partial_2^l 
	\sfK_n (x + \tilde{z}, x + 2 \lambda y) d \tilde{z}
\end{equation*}
and then we use property \eqref{eq:reg-int} of the kernel $\sfK$. Overall, for $x\in K$ we get
	\begin{align*}
		& | \partial^l \eta ( y ) | 
		\lesssim 
		c_{K'} \, 2^{\beta n} \lambda^{d +|l|} \| \varphi \|_{C^{|l|}} 2^{(d - \beta + |l|) n} 
		\int_{z \in B(2 y, \rho \frac{2^{- n}}{\lambda})} 
		| z - 2 y|^{|l|} d z \\
		& \qquad
		\qquad\qquad\qquad 
		+ c_{K'} \, 2^{\beta n} \lambda^{d +|l|} \sum\limits_{| k | \leq |l | - 1} 
		\frac{\| \varphi \|_{C^{|l|}}}{k !}
		\lambda^{- d - | k |} 2^{- \beta n} \,,
	\end{align*}
where $K' := K \oplus B(0,2\bar\lambda)$. It follows that
	\begin{equation*}
		\sup_{n\in\N, \, \lambda \in (0,\bar\lambda], \, x\in K} 
		\, \| \eta \|_{C^{r}} \lesssim
		\mathrm{cst}_{K,\bar\lambda} \,  \| \varphi \|_{C^{r}} \le
		\mathrm{cst}_{K,\bar\lambda} \,,
	\end{equation*}
where $\mathrm{cst}_{K,\bar\lambda} < \infty$ depends 
on $K$, $\bar\lambda$ and on the kernel $\sfK$ (in particular, on $\rho$).

\paragraph{Bound on $C^{r}$ norm of $\zeta$}
Let $l \in \mathbb{N}_0^d$ be a multi-index with $| l | \leq r$, and $y \in \mathbb{R}^d$, then by differentiation under the integral,
	\begin{equation*}
		\partial^l \zeta ( y ) = 2^{\beta n} (2^n \lambda)^{- m} 
		( 2\rho 2^{- n})^{d + |l|} \int_{\mathbb{R}^d} \varphi ( z ) \partial_2^l 
		\sfK_n (x + \lambda z, x + 2\rho 2^{- n} y) d z .
	\end{equation*}
Recall that by assumption $\varphi$ annihilates polynomials of degree $c \geq m - 1$ 
so in this integral we can subtract the Taylor polynomial of 
$\partial_2^l \sfK_n (\cdot, x + 2\rho 2^{- n} y)$
based at $x$ of order $m - 1$. It is convenient to denote:
	\begin{equation*}
		R_{n, \lambda, x, y, l}^{[m - 1]} ( z ) \coloneqq \partial_2^l 
		\sfK_n (x + \lambda z, x + 2\rho 2^{- n} y) - \sum\limits_{| k | \leq m - 1} 
		\partial_1^k \partial_2^l \sfK_n (x , x + 2\rho 2^{- n} y) \frac{(\lambda z)^k}{k !} ,
	\end{equation*}
\noindent then 
	\begin{equation*}
		\partial^l \zeta ( y ) = 2^{\beta n} (2^n \lambda)^{- m} ( 2\rho 2^{- n})^{d + |l|} \int_{\mathbb{R}^d} \varphi ( z ) R_{n, \lambda, x, y, l}^{[m - 1]} ( z ) d z .
	\end{equation*}
By Taylor-Lagrange's formula, 
	\begin{equation*}
		\left| R_{n, \lambda, x, y, l}^{[m - 1]} ( z ) \right| \leq \sum\limits_{| k | = m} \frac{1}{k !} \left\| \partial_1^k \partial_2^l \sfK_n \right\|_{\infty} | \lambda z |^{m} .
	\end{equation*}
Thus, by the property \eqref{eq:regularising_kernel_bound_on_derivatives} of the kernel $\sfK$, 
for $x\in K$ and $\lambda \in (0,\bar\lambda]$
	\begin{equation*}
		\left| R_{n, \lambda, x, y, l}^{[m - 1]} ( z ) \right| \lesssim
		2^{(d - \beta + m + |l|)n} \lambda^{m} | z |^{m} ,
	\end{equation*}
\noindent where the implicit multiplicative constant depends on the
compact $K$, on $\bar\lambda$ 
and on the kernel $\sfK$ (and on the dimension $d$ of the underlying space).
Consequently, 
	\begin{equation*}
		\left| \partial^l \zeta ( y ) \right| \lesssim
		\int_{\mathbb{R}^d} | \varphi ( z ) | | z |^{m} d z
		\le \int_{\mathbb{R}^d} | \varphi ( z ) | d z
		\lesssim 1  \,,
	\end{equation*}
because by assumption $\varphi \in \mathscr{B}_{c}^{r}$.
Thus, this establishes:
	\begin{equation*}
		\sup_{n\in\N, \, \lambda \in (0,\bar\lambda], \, x\in K} \,
		\| \zeta \|_{C^{r}} \lesssim 1  \,,
	\end{equation*}
which concludes the proof.\qed

\section{Spaces of germs and distributions are \enquote{independent of \texorpdfstring{$r$}{r}}} \label{section:independence_in_r}

In this section we prove (using wavelet techniques) that the choice of the regularity $r$ of test-functions in the different spaces of distributions and germs studied in this paper generally does not matter.

\begin{proposition}\label{prop:independence_in_r}
Let $\bar{\alpha}, \alpha, \gamma \in \mathbb{R}$ with $\bar{\alpha}, \alpha \leq \gamma$.
Then:
	\begin{enumerate}
		\item The Definition \ref{def:Hoelder_Zygmund_spaces} of H\"older-Zygmund spaces $\mathcal{Z}^{\gamma}$ does not depend on the choice of $r \geq r_{\gamma} \coloneqq \min \lbrace r \in \mathbb{N}_0 , r > - \gamma \rbrace$.
		\item The Definition \ref{def:homogeneity_coherence} of homogeneous and coherent germs $\mathcal{G}^{\bar{\alpha}; \alpha, \gamma}$ does not depend on the choice of $r \geq r_{\bar{\alpha}, \alpha} \coloneqq \min \lbrace r \in \mathbb{N}_0, r > \max ( - \bar{\alpha}, - \alpha ) \rbrace$.
		\item The Definition \ref{def:weak_homogeneity_coherence} of weakly homogeneous and coherent germs $\mathcal{G}_{\mathrm{weak}}^{\bar{\alpha}; \alpha, \gamma}$ does not depend on the choice of $r \geq r_{\bar{\alpha}, \alpha} \coloneqq \min \lbrace r \in \mathbb{N}_0, r > \max ( - \bar{\alpha}, - \alpha ) \rbrace$.
	\end{enumerate}
\end{proposition} 

A proof in the case of the H\"older-Zygmund spaces $\mathcal{Z}^{\gamma}$ can be found for instance in \cite[Lemma~14.13]{MR4174393}.

A proof in the case of the space of homogeneous and coherent germs $\mathcal{G}^{\bar{\alpha}; \alpha, \gamma}$ when $\gamma \neq 0$ can be found in \cite[Propositions~13.1 and~13.2]{CZ20}, see Remark \ref{rem:uniformity}.
However the approach in this reference fails to cover the case $\gamma = 0$.

We prove Proposition~\ref{prop:independence_in_r} using the following result from wavelet theory: 
\begin{theorem}[Daubechies' wavelets, see \cite{MR951745,MR1162107,MR1228209}]\label{thm:Daubechies_wavelets}
For any $r, d \in \mathbb{N}_0$, there exist a compactly supported function $\varphi \in C_c^r ( \mathbb{R}^d)$ and a finite family $\Psi$ of compactly supported functions $\psi \in C_c^r (\mathbb{R}^d)$ satisfying $\int_{\mathbb{R}^d} \psi ( x ) x^k d x = 0$ for all multi-indices $k \in \mathbb{N}_0^d$ with $| k | \leq r$, such that for all $n_0 \in \mathbb{Z}$, the family
	\begin{equation}\label{eq:wavelet_basis}
		\lbrace 2^{\frac{- n_0 d}{2}} \varphi_k^{2^{- n_0}} 
		\colon k \in 2^{- n_0} \mathbb{Z} \rbrace \cup 
		\lbrace 2^{\frac{- n d}{2}} \psi_k^{2^{- n}} \colon
		n \ge n_0, \, k \in 2^{- n} \mathbb{Z}, \, \psi \in \Psi \rbrace ,
	\end{equation}
	
\noindent is a Hilbert basis of $L^2 ( \mathbb{R}^d )$.
\end{theorem}

In fact, the convergence along the basis \eqref{eq:wavelet_basis} holds in $C^r$ norm.
This allows us to prove Proposition~\ref{prop:independence_in_r}.

\begin{proof}[Proof of Proposition~\ref{prop:independence_in_r}]
As a proof in the case of the spaces $\mathcal{Z}^{\gamma}$ can be found in the literature, see \cite[Lemma~14.13]{MR4174393}, we only consider the case of spaces of germs.
We argue slightly differently in the case of $\mathcal{G}$ and in the case of $\mathcal{G}_{\mathrm{weak}}$:
\begin{itemize}
	\item in the case of $\mathcal{G}$, we exploit the decomposition \eqref{eq:wavelet_basis} starting from $n_0 \in \mathbb{Z}$ with $2^{- n_0} \sim \lambda$,
	\item in the case of $\mathcal{G}_{\mathrm{weak}}$, we exploit the decomposition \eqref{eq:wavelet_basis} starting from $n_0 = 0$.
\end{itemize}

Let $\alpha, \bar{\alpha}, \gamma \in \mathbb{R}$ be such that $\alpha \leq \gamma$, $\bar{\alpha} \leq \gamma$, and define
	\begin{equation*}
		r_{\alpha, \bar{\alpha}} \coloneqq \min \lbrace r \in \mathbb{N}_0, r > \max \lbrace - \alpha, - \bar{\alpha} \rbrace \rbrace .
	\end{equation*}

For any $r \in \mathbb{N}_0$ arbitrary, we denote $\mathcal{G}_r^{\bar{\alpha}; \alpha, \gamma}$ the space of germs corresponding to the family of seminorms given by \eqref{eq:semi-norm-coherence+homogeneity}; and similarly $\mathcal{G}_{\mathrm{weak};\,r}^{\bar{\alpha}; \alpha, \gamma}$ corresponding to the family of seminorms \eqref{eq:semi-norm-weak-coherence+homogeneity}.

Let $r \in \mathbb{N}_0$ with $r \geq r_{\alpha, \bar{\alpha}}$, we shall show that
	\begin{align}
		\mathcal{G}_{r}^{\bar{\alpha}; \alpha, \gamma} & = \mathcal{G}_{r_{\alpha, \bar{\alpha}}}^{\bar{\alpha}; \alpha, \gamma} , \label{eq:independence_in_r_G} \\
		\mathcal{G}_{\mathrm{weak};\,r}^{\bar{\alpha}; \alpha, \gamma} & = \mathcal{G}_{\mathrm{weak};\,r_{\alpha, \bar{\alpha}}}^{\bar{\alpha}; \alpha, \gamma} . \label{eq:independence_in_r_checkG} 
	\end{align}

\medskip \textit{Proof of \eqref{eq:independence_in_r_G}}.
It suffices to show the inclusion
	\begin{equation*}
		\mathcal{G}_{r}^{\bar{\alpha}; \alpha, \gamma} \subset \mathcal{G}_{r_{\alpha, \bar{\alpha}}}^{\bar{\alpha}; \alpha, \gamma} ,
	\end{equation*}
\noindent because the other one follows from the definitions.
Let $F \in \mathcal{G}_{r}^{\bar{\alpha}; \alpha, \gamma}$, we start with the estimate of homogeneity.
Let $\varphi, \Psi$ be as in Theorem~\ref{thm:Daubechies_wavelets} applied to $r$.
Let $K \subset \mathbb{R}^d$ be compact, $x \in K$, $\lambda \in (0, 1]$, $\eta \in \mathscr{B}^{r_{\alpha, \bar{\alpha}}}$, we want to estimate $F_x (\eta_x^{\lambda})$.
Set $N \coloneqq N_{\lambda} \coloneqq \min \lbrace n \in \mathbb{N}, 2^{- n} \leq \lambda \rbrace$.
From the decomposition \eqref{eq:wavelet_basis} starting at $N_{\lambda}$, we have:
	\begin{align*}
		F_x (\eta_x^{\lambda}) & = \sum\limits_{k \in 2^{- N_{\lambda}} \mathbb{Z}} 2^{- N_{\lambda} d} \langle \eta_x^{\lambda} , \varphi_k^{2^{- N_{\lambda}}} \rangle F_x (\varphi_k^{2^{- N_{\lambda}}}) \\
		& \quad + \sum\limits_{n = N_{\lambda}}^{+ \infty} \sum\limits_{ k \in 2^{- n} \mathbb{Z}} \sum\limits_{\psi \in \Psi } 2^{- n d} \langle \eta_x^{\lambda}, \psi_k^{2^{- n}} \rangle F_x ( \psi_k^{2^{- n}} ) .
	\end{align*}

In the first line, for reasons of support one has $| x - k | \lesssim \lambda$ and only a finite number of $k$ contribute to the sum.
In the second line, for reasons of support one has $| x - k | \lesssim \lambda$ and $\sim 2^{(n - N_{\lambda}) d}$ values of $k$ contribute to the sum.
Thus, because of the coherence and homogeneity of $F$, one has $| F_x (\varphi_k^{2^{- N_{\lambda}}}) | \lesssim \lambda^{\bar{\alpha}} + \lambda^{\gamma} \lesssim \lambda^{\bar{\alpha}}$ in the first line, and $|F_x ( \psi_k^{2^{- n}} ) | \lesssim 2^{- n \bar{\alpha}} + 2^{- n \alpha} \lambda^{\gamma - \alpha}$ in the second line.
Also, since the functions $\psi$ cancel polynomials of degree up to $r \geq r_{\alpha, \bar{\alpha}}$, by subtracting a Taylor polynomial of degree $\tilde{r} \coloneqq r_{\alpha, \bar{\alpha}} - 1$ in the integral one obtains $| \langle \eta_x^{\lambda}, \psi_k^{2^{- n}} \rangle | \lesssim \| \eta \|_{C^{\tilde{r} + 1}} \lambda^{-d} \big( \frac{2^{-n}}{\lambda} \big)^{\tilde{r} + 1}$.
Thus, collecting these estimate:
	\begin{align*}
		| F_x (\eta_x^{\lambda}) | & \lesssim 2^{- N_{\lambda} d} \| \eta \|_{\infty} 2^{N_{\lambda} d} 2^{- N_{\lambda} \bar{\alpha}} \\
		& \quad + \sum\limits_{n = N_{\lambda}}^{+ \infty} 2^{- n d} 2^{(n - N_{\lambda}) d} \| \eta \|_{C^{\tilde{r} + 1}} \lambda^{-d} \left( \frac{2^{-n}}{\lambda} \right)^{\tilde{r} + 1} \left( 2^{- n \bar{\alpha}} + 2^{- n \alpha} \lambda^{\gamma - \alpha} \right) .
	\end{align*}

Recalling that by choice of $\tilde{r}$ one has $\tilde{r} + 1 > - \alpha$ and $\tilde{r} + 1 > - \bar{\alpha}$, by summing the geometric series one obtains the wanted homogeneity estimate
	\begin{equation*}
		| F_x (\eta_x^{\lambda}) | \lesssim \| \eta \|_{C^{r_{\alpha, \bar{\alpha}}}} \lambda^{\bar{\alpha}} .
	\end{equation*}

We establish the estimate of coherence similarly.
Let again $\varphi, \Psi$ be as in Theorem~\ref{thm:Daubechies_wavelets} applied to $r$.
Let $K \subset \mathbb{R}^d$ be compact, $x, y \in K$, $\lambda \in (0, 1]$, $\eta \in \mathscr{B}^{r_{\alpha, \bar{\alpha}}}$, we want to estimate $(F_y - F_x) (\eta_x^{\lambda})$.
As above, set $N \coloneqq N_{\lambda} \coloneqq \min \lbrace n \in \mathbb{N}, 2^{- n} \leq \lambda \rbrace$, from the decomposition \eqref{eq:wavelet_basis} starting at $N_{\lambda}$, we have:
	\begin{align*}
		(F_y - F_x) (\eta_x^{\lambda}) & = \sum\limits_{k \in 2^{- N_{\lambda}} \mathbb{Z}} 2^{- N_{\lambda} d} \langle \eta_x^{\lambda} , \varphi_k^{2^{- N_{\lambda}}} \rangle (F_y - F_x) (\varphi_k^{2^{- N_{\lambda}}}) \\
		& \quad + \sum\limits_{n = N_{\lambda}}^{+ \infty} \sum\limits_{ k \in 2^{- n} \mathbb{Z}} \sum\limits_{\psi \in \Psi } 2^{- n d} \langle \eta_x^{\lambda}, \psi_k^{2^{- n}} \rangle ( F_y - F_x ) ( \psi_k^{2^{- n}} ) .
	\end{align*}

We perform the same estimate as above except this time from the assumption of coherence on the germ $F$ (and the fact that $|x - k| \lesssim \lambda$ for reasons of support) one has $| (F_y - F_x) (\varphi_k^{2^{- N_{\lambda}}}) | \lesssim \lambda^{\alpha} (| y - x | + \lambda)^{\gamma - \alpha}$ in the first line, and similarly in the second line $| ( F_y - F_x ) ( \psi_k^{2^{- n}} ) | \lesssim 2^{- n \alpha} (| y - x| + \lambda)^{\gamma - \alpha}$.
Thus, collecting all the estimate, one obtains for $\tilde{r} \coloneqq r_{\alpha, \bar{\alpha}} - 1$:
	\begin{align*}
		| (F_y - F_x) (\eta_x^{\lambda}) | & \lesssim 2^{- N_{\lambda} d} \| \eta \|_{\infty} 2^{N_{\lambda} d} \lambda^{\alpha} (| y - x | + \lambda)^{\gamma - \alpha} \\
		& \quad + \sum\limits_{n = N_{\lambda}}^{+ \infty} 2^{- n d} 2^{(n - N_{\lambda}) d} \| \eta \|_{C^{\tilde{r} + 1}} \lambda^{-d} \left( \frac{2^{-n}}{\lambda} \right)^{\tilde{r} + 1} 2^{- n \alpha} (| y - x| + \lambda)^{\gamma - \alpha} ,
	\end{align*}
	
\noindent so that using the fact that $\tilde{r} + 1 > - \alpha$ one obtains after summing the geometric series
	\begin{equation*}
		| (F_y - F_x) (\eta_x^{\lambda}) | \lesssim \| \eta \|_{C^{r_{\alpha, \bar{\alpha}}}} \lambda^{\alpha} (| y - x | + \lambda)^{\gamma - \alpha} .
	\end{equation*}

This concludes the proof of \eqref{eq:independence_in_r_G}.

\medskip \textit{Proof of \eqref{eq:independence_in_r_checkG}}.
It suffices to show the inclusion
	\begin{equation*}
		\mathcal{G}_{\mathrm{weak};\,r}^{\bar{\alpha}; \alpha, \gamma} \subset \mathcal{G}_{\mathrm{weak};\,r_{\alpha, \bar{\alpha}}}^{\bar{\alpha}; \alpha, \gamma} ,
	\end{equation*}
\noindent because the other one follows from the definitions.
Let $F \in \mathcal{G}_{\mathrm{weak};\,r}^{\bar{\alpha}; \alpha, \gamma}$, we start with the estimate of homogeneity.
Let again $\varphi, \Psi$ be as in Theorem~\ref{thm:Daubechies_wavelets} applied to $r$.
Let $K \subset \mathbb{R}^d$ be compact, $x \in K$, $\lambda \in (0, 1]$, $\eta \in \mathscr{B}^{r_{\alpha, \bar{\alpha}}}$, $\check{\eta} \in \mathscr{B}_{\bar{\alpha}}^{r_{\alpha, \bar{\alpha}}}$.
We want to estimate $F_x (\eta_x^{})$ and $F_x (\check{\eta}_x^{\lambda})$.
From the decomposition \eqref{eq:wavelet_basis} starting at $0$, we have:
	\begin{align*}
		F_x (\eta_x^{}) & = \sum\limits_{k \in \mathbb{Z}} \langle \eta_x^{} , \varphi_k^{} \rangle F_x (\varphi_k^{}) \\
		& \quad + \sum\limits_{n = 0}^{+ \infty} \sum\limits_{ k \in 2^{- n} \mathbb{Z}} \sum\limits_{\psi \in \Psi } 2^{- n d} \langle \eta_x^{}, \psi_k^{2^{- n}} \rangle F_x ( \psi_k^{2^{- n}} ) .
	\end{align*}
	
In the first line, for reasons of support one has $| x - k | \lesssim 1$ and only a finite number of $k$ contribute to the sum.
In the second line, for reasons of support one has $| x - k | \lesssim 1$ and $\sim 2^{n d}$ values of $k$ contribute to the sum.
Thus, because of the coherence and homogeneity of the germ $F$, one has $| F_x (\varphi_k^{}) | \lesssim 1$ in the first line, and $| F_x ( \psi_k^{2^{- n}} ) | \lesssim 2^{- n \bar{\alpha}} + 2^{- n \alpha}$ in the second line.
Also, since the functions $\psi$ cancel polynomials of degree up to $r \geq r_{\alpha, \bar{\alpha}}$, by subtracting a Taylor polynomial of degree $\tilde{r} \coloneqq r_{\alpha, \bar{\alpha}} - 1$ in the integral one obtains $| \langle \eta_x^{}, \psi_k^{2^{- n}} \rangle | \lesssim \| \eta \|_{C^{\tilde{r} + 1}} 2^{-n (\tilde{r} + 1)}$.
Thus, collecting these estimates:	
	\begin{equation*}
		| F_x (\eta_x^{}) | \lesssim \| \eta \|_{\infty} + \sum\limits_{n = 0}^{+ \infty} 2^{- n d} 2^{n d} \| \eta \|_{C^{\tilde{r} + 1}} 2^{-n (\tilde{r} + 1)} \left( 2^{- n \bar{\alpha}} + 2^{- n \alpha} \right) ,
	\end{equation*}
	
\noindent so that summing the geometric series and recalling that $\tilde{r} + 1 > - \alpha$, $\tilde{r} + 1 > - \bar{\alpha}$, one obtains $| F_x (\eta_x^{}) | \lesssim  \| \eta \|_{C^{r_{\alpha, \bar{\alpha}}}}$.

Similarly:
	\begin{align*}
		F_x (\check{\eta}_x^{\lambda}) & = \sum\limits_{k \in \mathbb{Z}} \langle \check{\eta}_x^{\lambda} , \varphi_k^{} \rangle F_x (\varphi_k^{}) \\
		& \quad + \sum\limits_{n = 0}^{N_{\lambda}} \sum\limits_{ k \in 2^{- n} \mathbb{Z}} \sum\limits_{\psi \in \Psi } 2^{- n d} \langle \check{\eta}_x^{\lambda}, \psi_k^{2^{- n}} \rangle F_x ( \psi_k^{2^{- n}} ) \\
		& \quad + \sum\limits_{n = N_{\lambda} + 1}^{+ \infty} \sum\limits_{ k \in 2^{- n} \mathbb{Z}} \sum\limits_{\psi \in \Psi } 2^{- n d} \langle \check{\eta}_x^{\lambda}, \psi_k^{2^{- n}} \rangle F_x ( \psi_k^{2^{- n}} ) .
	\end{align*}

In the first line, for reasons of support one has $| x - k | \lesssim 1$ and only a finite number of $k$ contribute to the sum. 
Also, since $\check{\eta}$ annihilate polynomials of degree up to $\lfloor \bar{\alpha} \rfloor$, by subtracting a Taylor polynomial of $\varphi$ of degree $\lfloor \bar{\alpha} \rfloor$ in the integral one obtains $| \langle \check{\eta}_x^{\lambda} , \varphi_k^{} \rangle | \lesssim \| \check{\eta} \|_{\infty} \lambda^{\lfloor \bar{\alpha} \rfloor + 1} \lesssim \| \check{\eta} \|_{\infty} \lambda^{\bar{\alpha}}$.
Furthermore, because of the coherence and homogeneity of the germ $F$, $| F_x (\varphi_k^{}) | \lesssim 1$.

In the second line, for reasons of support one has $| x - k | \lesssim 2^{- n}$ and only a finite number of $k$ contribute to the sum.
Also, since $\check{\eta}$ annihilate polynomials of degree up to $\lfloor \bar{\alpha} \rfloor$, by subtracting a Taylor polynomial of $\varphi$ of degree $\lfloor \bar{\alpha} \rfloor$ in the integral one obtains $|\langle \check{\eta}_x^{\lambda}, \psi_k^{2^{- n}} \rangle | \lesssim \| \check{\eta} \|_{\infty} 2^{n d} \left( \frac{\lambda}{2^{-n}} \right)^{\lfloor \bar{\alpha} \rfloor + 1}$.
Furthermore, because of the coherence and homogeneity of the germ $F$, $| F_x ( \psi_k^{2^{- n}} ) | \lesssim 2^{- n \gamma} + 2^{- n \bar{\alpha}} \lesssim 2^{- n \bar{\alpha}}$ (since we assume $\bar{\alpha} \leq \gamma$).

In the third line, for reasons of support one has $| x - k | \lesssim \lambda$ and $\sim 2^{(n - N_{\lambda}) d}$ values of $k$ contribute to the sum.
Also, since the functions $\psi$ cancel polynomials of degree up to $r \geq r_{\alpha, \bar{\alpha}}$, by subtracting a Taylor polynomial of $\check{\eta}$ of degree $\tilde{r} \coloneqq r_{\alpha, \bar{\alpha}} - 1$ in the integral one obtains $| \langle \check{\eta}_x^{\lambda}, \psi_k^{2^{- n}} \rangle | \lesssim \| \check{\eta} \|_{C^{\tilde{r} + 1}} 2^{-n (\tilde{r} + 1)}$.
Furthermore, because of the coherence and homogeneity of the germ $F$, $| F_x ( \psi_k^{2^{- n}} ) | \lesssim 2^{- n \alpha} \lambda^{\gamma - \alpha} + 2^{- n \bar{\alpha}}$.
	
Collecting these estimates yields:
	\begin{align*}
		| F_x (\check{\eta}_x^{\lambda}) | & \lesssim \| \check{\eta} \|_{\infty} \lambda^{\bar{\alpha}} \\
		& \quad + \sum\limits_{n = 0}^{N_{\lambda}} 2^{- n d}  \| \check{\eta} \|_{\infty} 2^{n d} \left( \frac{\lambda}{2^{-n}} \right)^{\lfloor \bar{\alpha} \rfloor + 1} 2^{- n \bar{\alpha}} \\
		& \quad + \sum\limits_{n = N_{\lambda} + 1}^{+ \infty} 2^{- n d} 2^{(n - N_{\lambda}) d} \| \check{\eta} \|_{C^{\tilde{r} + 1}} 2^{-n (\tilde{r} + 1)} \left( 2^{- n \alpha} \lambda^{\gamma - \alpha} + 2^{- n \bar{\alpha}} \right) ,
	\end{align*}
	
\noindent so that summing the geometric series and recalling that $\tilde{r} + 1 > - \alpha$, $\tilde{r} + 1 > - \bar{\alpha}$, $\lfloor \bar{\alpha} \rfloor + 1 > \bar{\alpha}$, one obtains:
	\begin{equation*}
		| F_x (\check{\eta}_x^{\lambda}) | \lesssim \| \check{\eta} \|_{C^{r_{\alpha, \bar{\alpha}}}} \lambda^{\bar{\alpha}} .
	\end{equation*}

Once again, we establish the property of coherence similarly.
Let $\varphi, \Psi$ be as in Theorem~\ref{thm:Daubechies_wavelets} applied to $r$.
Let $K \subset \mathbb{R}^d$ be compact, $x, y \in K$, $\lambda \in (0, 1]$, $\eta \in \mathscr{B}^{r_{\alpha, \bar{\alpha}}}$, $\check{\eta} \in \mathscr{B}_{\gamma}^{r_{\alpha, \bar{\alpha}}}$, we want to estimate $(F_y - F_x) (\eta_x^{})$ and $( F_y - F_x ) (\check{\eta}_x^{\lambda})$.
From the decomposition \eqref{eq:wavelet_basis} starting at $0$, we have:
	\begin{align*}
		( F_y - F_x ) (\eta_x^{}) & = \sum\limits_{k \in \mathbb{Z}} \langle \eta_x^{} , \varphi_k^{} \rangle ( F_y - F_x ) (\varphi_k^{}) \\
		& \quad + \sum\limits_{n = 0}^{+ \infty} \sum\limits_{ k \in 2^{- n} \mathbb{Z}} \sum\limits_{\psi \in \Psi } 2^{- n d} \langle \eta_x^{}, \psi_k^{2^{- n}} \rangle ( F_y - F_x ) ( \psi_k^{2^{- n}} ) .
	\end{align*}

We perform the same estimates as in the case of the homogeneity, except for the fact that in the first line $| ( F_y - F_x ) (\varphi_k^{}) | \lesssim 1$ because of the assumption of coherence of $F$ (and the fact that $| x - k | \lesssim 1$ for reasons of support); and the fact that in the second line, $| ( F_y - F_x ) ( \psi_k^{2^{- n}} ) | \lesssim 2^{- n \alpha}$ because of the assumption of coherence of $F$ (and the facts that $x, y \in K$ for a compact $K$ and $| x - k | \lesssim 1$ for reasons of support).
Thus this yields for $\tilde{r} \coloneqq r_{\alpha, \bar{\alpha}} - 1$:
	\begin{align*}
		| ( F_y - F_x ) (\eta_x^{}) | & \lesssim \| \eta \|_{\infty} \\
		& \quad + \sum\limits_{n = 0}^{+ \infty} 2^{- n d} 2^{n d} \| \eta \|_{C^{\tilde{r} + 1}} 2^{-n (\tilde{r} + 1)} 2^{- n \alpha} ,
	\end{align*}
	
\noindent so that summing the geometric series yields $| ( F_y - F_x ) (\eta_x^{}) | \lesssim 1$.

Similarly:
\begin{align*}
		( F_y - F_x ) (\check{\eta}_x^{\lambda}) & = \sum\limits_{k \in \mathbb{Z}} \langle \check{\eta}_x^{\lambda} , \varphi_k^{} \rangle ( F_y - F_x ) (\varphi_k^{}) \\
		& \quad + \sum\limits_{n = 0}^{N_{\lambda}} \sum\limits_{ k \in 2^{- n} \mathbb{Z}} \sum\limits_{\psi \in \Psi } 2^{- n d} \langle \check{\eta}_x^{\lambda}, \psi_k^{2^{- n}} \rangle ( F_y - F_x ) ( \psi_k^{2^{- n}} ) \\
		& \quad + \sum\limits_{n = N_{\lambda} + 1}^{+ \infty} \sum\limits_{ k \in 2^{- n} \mathbb{Z}} \sum\limits_{\psi \in \Psi } 2^{- n d} \langle \check{\eta}_x^{\lambda}, \psi_k^{2^{- n}} \rangle ( F_y - F_x ) ( \psi_k^{2^{- n}} ) .
	\end{align*}

We perform the same estimates as in the case of the homogeneity above except that here, in the first line $| ( F_y - F_x ) (\varphi_k^{}) | \lesssim 1$; in the second line $| ( F_y - F_x ) ( \psi_k^{2^{- n}} ) | \lesssim 2^{- n \alpha} (| y - x| + 2^{- n})^{\gamma - \alpha}$; and in the third line $| ( F_y - F_x ) ( \psi_k^{2^{- n}} ) | \lesssim 2^{- n \alpha} (| y - x | + \lambda)^{\gamma - \alpha}$.
Thus this yields for $\tilde{r} \coloneqq r_{\alpha, \bar{\alpha}} - 1$:
	\begin{align*}
		| ( F_y - F_x ) (\check{\eta}_x^{\lambda}) | & \lesssim \| \check{\eta} \|_{\infty} \lambda^{\gamma} \\
		& \quad + \sum\limits_{n = 0}^{N_{\lambda}} 2^{- n d}  \| \check{\eta} \|_{\infty} 2^{n d} \left( \frac{\lambda}{2^{-n}} \right)^{\lfloor \gamma \rfloor + 1} 2^{- n \alpha} (| y - x| + 2^{- n})^{\gamma - \alpha}\\
		& \quad + \sum\limits_{n = N_{\lambda} + 1}^{+ \infty} 2^{- n d} 2^{(n - N_{\lambda}) d} \| \check{\eta} \|_{C^{\tilde{r} + 1}} 2^{-n (\tilde{r} + 1)} 2^{- n \alpha} (| y - x | + \lambda)^{\gamma - \alpha} ,
	\end{align*}

\noindent so that after summing the geometric series and recalling that $\tilde{r} + 1 > - \alpha$ and $\lfloor \gamma \rfloor + 1 > \gamma \geq \alpha$, we obtain:
	\begin{equation*}
		| ( F_y - F_x ) (\check{\eta}_x^{\lambda}) | \lesssim \| \check{\eta} \|_{C^{r_{\alpha, \bar{\alpha}}}} \lambda^{\alpha} ( | y - x | + \lambda)^{\gamma - \alpha} .
	\end{equation*}

This concludes the proof.
\end{proof}

\printbibliography

\end{document}